%% file: document.tex
\documentclass[bj,preprint]{imsart}

\newcommand{\co}[1]{{#1}}

\RequirePackage{amsthm,amsmath}
\RequirePackage[numbers]{natbib}
\RequirePackage[colorlinks,citecolor=blue,urlcolor=blue]{hyperref}

\startlocaldefs

  \newtheorem{theorem}{Theorem}
  \newtheorem{lemma}[theorem]{Lemma}
  \newtheorem{proposition}[theorem]{Proposition}
  \newtheorem{remark}[theorem]{Remark}
  \newtheorem{corollary}[theorem]{Corollary}

\endlocaldefs

\usepackage{amsmath}
\usepackage{amssymb}
\usepackage{natbib}
\usepackage{graphicx}
\usepackage{url}

\usepackage[english]{babel}
\usepackage[utf8x]{inputenc}
\usepackage[T1]{fontenc}

\usepackage{times}

\usepackage{mathtools,soul}
\usepackage{comment}
\DeclarePairedDelimiter\floor{\lfloor}{\rfloor}

\usepackage{bbm}
\newcommand{\E}{\mathbb{E}}
\newcommand{\V}{\mathbb{V}}
\newcommand{\pa}[1]{\left(#1\right)}
\newcommand{\cro}[1]{\left[#1\right]}

\newcommand{\Po}{\mathbb{P}}

\begin{document}

\begin{frontmatter}

\title{Local minimax rates for closeness testing of discrete distributions}
\runtitle{Instance-optimal closeness testing}

\begin{aug}
	\author[A]{\fnms{Joseph} \snm{Lam-Weil}\ead[label=e1]{joseph.lam@ovgu.de}},
	\author[A]{\fnms{Alexandra} \snm{Carpentier}\ead[label=e2]{alexandra.carpentier@ovgu.de}}
	\and
	\author[B]{\fnms{Bharath K.} \snm{Sriperumbudur}\ead[label=e3]{bks18@psu.edu}}
	
	\runauthor{J. Lam-Weil et al.}
	
	\address[A]{Magdeburg University, Germany, \printead{e1}, \printead{e2}}
		
	\address[B]{Pennsylvania State University, USA, \printead{e3}}
\end{aug}

\begin{abstract}
\\
We consider the closeness testing problem for discrete distributions. The goal is to distinguish whether two samples are drawn from the same unspecified distribution, or whether their respective distributions are separated in $L_1$-norm. In this paper, we focus on adapting the rate to the shape of the underlying distributions, i.e.~we consider \textit{a local minimax setting}. We provide, to the best of our knowledge, the first local minimax rate for the separation distance up to logarithmic factors, together with a test that achieves it. In view of the rate, closeness testing turns out to be substantially harder than the related one-sample testing problem over a wide range of cases. 
\end{abstract}




\end{frontmatter}

%
%
%
%

%
%
%

\input{intro.tex}

\input{results.tex}

\section{Conclusion \& Discussion}
\label{sec:discussion}
In this paper, we have established the local minimax near-optimal separation distance for 
the closeness testing problem defined in Equation~\eqref{eq:prob}. It represents 
the first near-tight lower bound for local minimax closeness testing, and the first test that matches it up to log terms. 
The minimax rate is adaptive to $\pi$ in the following sense. The test we construct only takes samples from $p$ and $q$, but its testing rate optimally depends on $\pi$, as evidenced by the lower bound. The construction of the lower bound heavily relies on our formalization of closeness testing from Equation~\eqref{eq:prob}. Such a formalization is critically different from its identity testing counterpart, because of $q$ remaining unfixed. So we end up considering a testing problem, where both hypotheses are composite.
Comparing our local minimax separation distance with the one achievable in local minimax identity testing, a gap can be noted. Indeed, closeness testing turns out to be more difficult, especially when there are terms which are rather small without being negligible (corresponding to the indices between $J_\pi$ and $m^*$). But it is also noteworthy that both rates match otherwise.

On the horizon, the corresponding local minimax sample complexity for closeness testing has yet to be found. Besides, the upper bound could be made tighter in order to bridge the gap caused by the log factors. 
Finally, our analysis focuses on discrete distributions but the formalization of the problem of closeness testing presented in this paper generalizes well to other settings. Indeed, our formalization relies on $q$ being restrained to the set $\mathbf P_\pi$. Now, an analog set to $\mathbf P_\pi$ can be defined with an additional regularity condition in a continuous setting. So the extension of our study to densities still remains a major direction to be explored and it would be interesting to extend our paper as \cite{balakrishnan2017hypothesis} does for \cite{valiant2017automatic} in the context of identity testing.

\newpage
\appendix
\input{UBproof.tex}
\input{LBproof.tex}

\section*{Acknowledgements}
The work of A. Carpentier is partially supported by the Deutsche Forschungsgemeinschaft (DFG) Emmy Noether grant MuSyAD (CA 1488/1-1), by the DFG - 314838170, GRK 2297 MathCoRe, by the DFG GRK 2433 DAEDALUS, by the DFG CRC 1294 'Data Assimilation', Project A03, and by the UFA-DFH through the French-German Doktorandenkolleg CDFA 01-18.

\bibliographystyle{imsart-number}
\bibliography{document}

\end{document}

%% file: intro.tex

\section{Introduction}
Closeness testing, also known as two-sample testing or equivalence testing, amounts to testing whether two sets of samples are drawn from the same unknown distribution. The null hypothesis of this statistical testing problem is true when both distributions are the same. In the alternative hypothesis, they are different and separated in $L_1$-norm. For a fixed number of samples, the goal is to find out how close both distributions can get to one another and still be distinguishable, depending on the shape of one of the distributions. In the following, we provide a formal setting for this problem.


\subsection{Setting}
\label{sec:setting}

Let $d \in \mathbb N^*$. We define the set of vectors of size $d$ that correspond to multinomial distributions over $d$ categories as 
$$\mathbf P = \left\{\pi\in (\mathbb{R}^+)^{d}, \sum_{i\leq d} \pi_i = 1\right\}.$$
Let $\pi \in \mathbf P$. Define for any $i \in \mathbb Z$
$$S_\pi(i) = \{j \in \{1,\ldots,d\}: \pi_j \in [2^{-i}, 2^{-i+1})\}.$$
Define
\begin{align*}
\mathbf P_\pi = \Bigg\{q \in \mathbf P:
&\forall i\in \mathbb Z, \frac{|S_\pi(i)|}{2}  \leq \sum_{j = i-1}^{i+1} |S_q(j)| \leq \frac32 \sum_{j = i-2}^{i+2} |S_\pi(j)|
\Bigg\},
\end{align*}
which represents a class of probability vectors very similar to $\pi$. Indeed, for any $q \in \mathbf P_\pi$, the discrete level sets $S_q$ and $S_\pi$ are close in size. 

Let $p\in\mathbf P$, $q\in \mathbf P_\pi$ and $k \in \mathbb N^*$. The independent sample sets $(\mathcal X,\mathcal Y)$ are obtained from the following two multinomial distributions.
\begin{equation}\label{eq:model}
\mathcal X \sim \mathcal M(k, p),~~~ \mathcal Y \sim \mathcal M(k, q),
\end{equation}
where $\mathcal M$ is the multinomial distribution. That is, for $i \leq k,$ we have independent $\mathcal X_i$ taking value $j \in \{1,\ldots,d\}$ with probability $p_j$ -- respectively, $\mathcal Y_i = j$ with probability $q_j$. In what follows, we write $\mathbb P_{p,q}$ for the probability associated to $(\mathcal X, \mathcal Y)$.

For any vector $x \in \mathbb R^d$, let $\|x\|_t = (\sum_{i = 1}^d |x_i|^t)^{1/t}$ denote the $L_t$-norm of $x$ for any $0<t<\infty$. For a fixed $\rho >0$, we formalize the closeness testing problem as
\begin{equation}\label{eq:prob}
\mathcal H^{(\text{Clo})}_{0}(\pi): p = q,~~q \in \mathbf P_\pi,~~ \textrm{versus}~~ \mathcal H^{(\text{Clo})}_{1}(\pi,\rho): \|p - q\|_1 \geq \rho,~~q \in \mathbf P_\pi, p\in \mathbf P,
\end{equation}
from the observations of $(\mathcal X,\mathcal Y)$. Here $\rho$ represents a \textit{separation distance} between $p$ and $q$ which amounts to assuming that the null and alternative hypotheses are different enough. For any hypothesis $\mathcal H$, we introduce the corresponding set $H$ such that elements of $H$ satisfy the hypothesis $\mathcal H$. So $H_0^{\text{(Clo)}}(\pi) = \{(p,q) \in \mathbf P_\pi^2: p=q\}$ and $H_1^{\text{(Clo)}}(\pi,\rho) = \{(p,q) \in \mathbf P \times \mathbf P_\pi: \|p - q\|_1 \geq \rho\}$.

\begin{remark}
To the best of our knowledge, closeness testing has never known any formal definition as a hypothesis testing problem. Our formalization satisfies a few important criteria for the purpose of instance-optimal closeness testing. Firstly, the null hypothesis is composite as well as the alternative hypothesis, in contrast to identity testing whose null hypothesis is simple. That is, under any one of both hypotheses from Equation~\eqref{eq:prob}, $q$ is allowed to be quite different from $\pi$, since there is no relation in the ordering of their entries. So there does not exist any test exploiting the full knowledge of $p$ or $q$, and our problem is inherently harder than identity testing, where $q = \pi$, presented in \citep{valiant2017automatic} and \citep{balakrishnan2017hypothesis}. Secondly, 
$q\in \mathbf P_\pi$ is still related to $\pi$ in the sense discussed above, so that our results can be instance-optimal and depend on $\pi$. The results can vary greatly depending on $\pi$ and a worst-case study from \cite{chan2014optimal} does not guarantee an optimal test in all cases. Intuitively, if $\pi$ is the uniform distribution for example, the testing problem is more difficult than if $\pi$ just has a few entries with non-zero probability. We want to capture this dependence on the distribution, as \citep{valiant2017automatic} and \citep{balakrishnan2017hypothesis} do for one-sample testing. Finally, we provide in Section~\ref{sec:discussion} a discussion on how this formalization could be generalized to other set-ups.

\end{remark}

\noindent

It is clear from Equation~\eqref{eq:prob} that the vectors that are too close to $q$ are removed from the alternative hypothesis. With $\pi$ fixed, we want to find the smallest $\rho$ such that both hypotheses are still distinguishable. 
The notion of distinguishability of both hypotheses is formalized by the definition of error risk and separation distance. The error risk is the sum of type I and type II error probabilities. For any separation distance $\rho > 0$ and probability vector $\pi$, we can define some testing problem from a set couple $(H_0(\pi)$, $H_1(\pi,\rho))$ -- e.g.,~$H_0^{\text{(Clo)}}(\pi)$ and $H_1^{\text{(Clo)}}(\pi, \rho)$ for Equation~\eqref{eq:prob}. Then the separation distance given a test $\varphi$ based on a sample set $(\mathcal X, \mathcal Y)$ is
$$
R(H_0(\pi),H_1(\pi,\rho),\varphi;k) = \sup_{(p,q)\in H_{0}(\pi)} \mathbb{P}_{p,q}(\varphi(\mathcal X, \mathcal Y)=1) + \sup_{(p,q)\in H_{1}(\pi,\rho)} \mathbb{P}_{p,q}(\varphi(\mathcal X, \mathcal Y)=0),
$$
where we remind that $\mathbb{P}_{p,q}$ is the probability measure associated with $(\mathcal X,\mathcal Y)$. Then, fixing some $\gamma\in (0,1)$, we say that a testing problem can be solved with error smaller than $\gamma$, if we can construct a uniformly $\gamma$-consistent test, that is, if there exists $\varphi$ such that:
$$
R(H_0(\pi),H_1(\pi,\rho),\varphi;k) \leq \gamma.
$$
Clearly, $\rho \mapsto R(H_0(\pi),H_1(\pi,\rho),\varphi;k)$ is non-increasing, and greater or equal to one when $\rho=0$. Define the separation distance for some fixed $\gamma\in (0,1)$ as
$$
\rho_\gamma(H_0,H_1,\varphi;\pi,k) = \inf\{\rho>0 : R(H_0(\pi),H_1(\pi,\rho),\varphi;k) \leq \gamma\}.
$$
A good test $\varphi$ is characterized by a small separation distance. So we define the local minimax separation distance, also known as local critical radius, as
$$
\rho_\gamma^*(H_0,H_1; \pi,k) = \inf_\varphi \rho_\gamma(H_0,H_1,\varphi;\pi,k).
$$
Besides, it is possible to consider the global minimax separation distance defined as
$$\sup_{\pi \in \mathbf P} \rho_\gamma^*(H_0,H_1; \pi,k).$$
A worst-case analysis is sufficient for finding the global minimax separation distance, so it is a weaker result than finding the local minimax separation distance.

A lot of relevant results from the literature that we present in Section~\ref{sec:literature} come from the field of property testing in computer science. So although our paper focuses on rates in separation distance, we will link this concept with that of sample complexity, favoured in computer science.
Sample complexity corresponds to the number of samples that are necessary and sufficient in order to achieve a certain testing error for a fixed separation distance. Formally, 
for a fixed $\rho$,
since $k \mapsto R(H_0(\pi),H_1(\pi,\rho),\varphi; k)$ is non-increasing, the sample complexity for some fixed $\gamma\in (0,1)$ is defined as
$$
k_\gamma(H_0,H_1,\varphi;\pi,\rho) = \inf\{k \in \mathbb N : R(H_0(\pi),H_1(\pi,\rho),\varphi; k) \leq \gamma\}.
$$
Then the minimax sample complexity is 
$$
k_\gamma^*(H_0,H_1;\pi,\rho) = \inf_\varphi k_\gamma(H_0,H_1,\varphi;\pi,\rho).
$$
So the local minimax sample complexity is written as $k_\gamma^*(H_0,H_1;\pi,\rho)$. And the global minimax sample complexity is $\sup_\pi k_\gamma^*(H_0,H_1;\pi,\rho)$. If $\rho_\gamma^*$ or $k_\gamma^*$ are invertible, then it is possible to obtain one from the other. Let us define the inverses $\rho \mapsto (\rho_\gamma^*)^{-1}(H_0,H_1;\pi,\rho)$ and $k \mapsto (k_\gamma^*)^{-1}(H_0,H_1;\pi,k)$. Then $(\rho_\gamma^*)^{-1}(H_0,H_1;\pi,\rho) = k_\gamma^*(H_0,H_1;\pi,\rho)$ and reciprocally.

 \paragraph{Additional notations.} We introduce the following notations. 
 For a vector $u\in \mathbb R^d$, let $s$ be a permutation of $\{1, \ldots, d\}$ such that $u_{s(1)} \geq u_{s(2)} \geq \ldots \geq u_{s(d)}$. We write $u_{(.)}:=u_{s(.)}$. 
Set also
$J_{u} = \min_{j \leq d} \big\{j : u_{(j)} \leq \frac{1}{k}\big\}.$



%
%
%

\subsection{Literature review}
\label{sec:literature}

Hypothesis testing is a classical statistical problem and we refer the reader to \citep{neyman1933ix} and \citep{lehmann2006testing} for a more global perspective on the problem. In parallel to the study of hypothesis testing, there exists a broad literature on the related problem of property testing tackled by the theoretical computer science community, with seminal papers like \citep{rubinfeld1996robust}, \citep{goldreich1998property}. 

In earlier studies, tests were built based on good asymptotic properties like having asymptotically normal limits, but this criterion often fails to produce tests which are efficient in high-dimensional cases notably, as stated in \citep{balakrishnan2017hypothesisreview}. An alternative and popular take on the study of hypothesis testing is minimax optimality, with the seminal work of \citep{ingster2012nonparametric} on goodness-of-fit testing.
%
%

The problem of goodness-of-fit testing, also known as identity testing, or one-sample testing, consists in distinguishing whether a sample set $\mathcal X \sim \mathcal M(k,p)$ is drawn from a specified distribution $\pi \in \mathbf P$, versus a composite alternative separated from the null in $L_1$-distance. We formalize it as follows.
\begin{equation}
\label{eq:idTest}
\mathcal H^{(\text{Id})}_{0}(\pi): p = \pi, ~~~~~ \textrm{versus}~~~~~ \mathcal H^{(\text{Id})}_{1}(\pi,\rho): \|p-\pi\|_1 \geq \rho,~ p\in \mathbf P.
\end{equation}
We define the corresponding sets $H_0^{\text{(Id)}}(\pi) = \{(p,q) \in \mathbf P^2: p=q=\pi\}$ and $H_1^{\text{(Id)}}(\pi,\rho) = \{(p,q) \in \mathbf P^2: q=\pi, \|p - q\|_1 \geq \rho\}$.
We will consider local and global minimax rates for both the separation distance and sample complexity. 
Indeed, if either the number of sample points or the $L_1$-separation between the distributions is reduced, then the problem becomes more difficult. Thus, it is possible to parametrize the difficulty of the problem using either the number of sample points $k$ or the $L_1$-separation distance $\rho$. Tables~\ref{tab:idTestGlobal} and \ref{tab:idTestLocal} capture the existing results in the literature on the global and local minimax separation distance and sample complexity for goodness-of-fit testing.
%
\begin{table}[h]
\caption{Global minimax separation distance and sample complexity obtained for identity testing defined in Equation~\eqref{eq:idTest}: $\sup_\pi \rho_\gamma^*(H^{(\text{Id})}_0,H^{(\text{Id})}_1; \pi,k)$ and $\sup_\pi k^*_\gamma(H^{(\text{Id})}_0,H^{(\text{Id})}_1;\pi,\rho)$. The rates are only worst-case considerations, presented up to some constant depending only on $\gamma$.}
\label{tab:idTestGlobal}
\centering
\begin{tabular}{|c|c|c|}
\hline
        & Separation distance & Sample complexity \\ \hline
\citep{paninski2008coincidence} &  $d^{1/4}/\sqrt{k}$  &     $\sqrt{d}/\rho^2$   \\ \hline
\end{tabular}
\end{table}

\begin{table}[h]

\caption{Local minimax separation distance and sample complexity obtained for identity testing defined in Equation~\eqref{eq:idTest}: $\rho_\gamma^*(H^{(\text{Id})}_0,H^{(\text{Id})}_1; \pi,k)$ and $k^*_\gamma(H^{(\text{Id})}_0,H^{(\text{Id})}_1;\pi,\rho)$. The rates depend on the distribution in the null hypothesis and are up to some constant depending only on $\gamma$. Here $(x)_+= \max(0,x)$, and $\mathbf 1\{A_i\}$ equals 1 if $A_i$ is true and 0 otherwise. So $\|\pi_{(i)}^{2/3} \mathbf 1\{2 \leq i < m\}\|_1 = \sum_{2 \leq i < m} \pi_{(i)}^{2/3}$ and $\|(\pi_{(i)} \mathbf 1\{i\geq m\})_i\|_1 = \sum_{i \geq m} \pi_{(i)}$.}
\label{tab:idTestLocal}
\centering
\begin{tabular}{|c|c|}
\hline
        & \citep{valiant2017automatic}\\ \hline
Separation distance  &  $\min_{m}\cro{\frac{\|(\pi_{(i)}^{2/3}\mathbf 1\{2\leq i< m\})_i\|_1^{3/4}}{\sqrt{k}} \lor \frac{1}{k} \vee \|(\pi_{(i)}\mathbf 1\{i\geq m\})_i\|_1}$  \\ \hline
Sample complexity &     $\min_m\cro{\frac{1}{\rho} \vee \frac{\|(\pi_{(i)}^{2/3}\mathbf 1\{2\leq i < m\})_i\|_1^{3/2}}{(\rho-\|(\pi_{(i)} \mathbf 1\{i\geq m\})_i\|_1)_+^2}}$   \\ \hline
\end{tabular}
\end{table}

\noindent Similarly for two-sample testing, Tables~\ref{tab:probGlobal} and \ref{tab:probLocal} capture the existing results in the literature on the global and local minimax separation distance and sample complexity.
\begin{table}[h]
\caption{Bounds on the global minimax separation distance and sample complexity obtained for closeness testing defined in Equation~\eqref{eq:prob}: $\sup_\pi \rho_\gamma^*(H^{(\text{Clo})}_0,H^{(\text{Clo})}_1;\pi,k)$ and $\sup_\pi k^*_\gamma(H^{(\text{Clo})}_0,H^{(\text{Clo})}_1;\pi,\rho)$, up to some constant depending only on $\gamma$. The result in~\citep{batu2000testing} is only an upper bound (UB). The result in~\citep{chan2014optimal} provides matching upper and lower bounds, and is hence global minimax optimal.}
\label{tab:probGlobal}
\centering
\begin{tabular}{|c|c|c|}
\hline
        & Separation distance & Sample complexity \\ \hline
\citep{batu2000testing} (UB only) &  $d^{1/6}\log(d)^{1/4}/k^{1/4}$  &     $d^{2/3}\log (d)/\rho^4$   \\ \hline
 \citep{chan2014optimal}  (minimax) & $\frac{d^{1/2}}{k^{3/4}} \vee \frac{d^{1/4}}{k^{1/2}}$  &     $\frac{d^{2/3}}{\rho^{4/3}} \vee \frac{d^{1/2}}{\rho^{2}}$  \\ \hline
\end{tabular}
\end{table}

\begin{table}[h]
\caption{Upper bounds (UB) on the local minimax separation distance and sample complexity obtained for closeness testing defined in Equation~\eqref{eq:prob}: $\rho_\gamma^*(H^{(\text{Clo})}_0,H^{(\text{Clo})}_1; \pi,k)$ and $k^*_\gamma(H^{(\text{Clo})}_0,H^{(\text{Clo})}_1;\pi,\rho)$. The rates are problem-dependent, even though both distributions are unknown. It is presented up to some poly$\log(dk)$ for the separation distance and up to a poly$\log(d/\rho)$ for the sample complexity. We present here a corollary from their Proposition~2.14, applied to our closeness testing problem - which is not defined in~\citep{diakonikolas2016new}.}
\label{tab:probLocal}
\centering
\begin{tabular}{|c|c|}
\hline
        & \citep{diakonikolas2016new}\\ \hline
Separation distance (UB only)  &  $\frac{\|\mathbf 1\{\pi < 1/k\}\|_1^{1/2}\|\pi^2 \mathbf 1\{\pi < 1/k\}\|_1^{1/4}}{\sqrt{k}} \vee \frac{\|\pi^{2/3}\|_1^{3/4}}{\sqrt{k}}$  \\ \hline
Sample complexity (UB only) &     $\min_m\pa{m \vee \frac{\|\mathbf 1\{\pi < 1/m\})\|_1\|\pi^2 \mathbf 1\{\pi < 1/m\}\|_1^{1/2}}{\rho^2} \vee \frac{\|\pi^{2/3}\|_1^{3/2}}{\rho^2}}$   \\ \hline
\end{tabular}
\end{table}

First, let us consider the results obtained for the classical problem of identity testing presented in Equation~\eqref{eq:idTest}. 
An upper bound on the global minimax sample complexity is given in \citep{paninski2008coincidence} and tightened in \citep{valiant2017automatic} for the class of multinomial distributions over a support of size $d$. The meaning of the global minimax sample complexity listed in Table~\ref{tab:idTestGlobal} is that an optimal algorithm will be able to test with fixed non-trivial probability, using $\sqrt{d}/\rho^2$ samples, up to some explicit constant. This sample complexity can be translated into the separation distance presented in the same table, as justified in Section~\ref{sec:setting}. The global minimax sample complexity is a worst-case analysis, that is, it corresponds to the rate obtained for the hardest problem overall. In the case of identity testing, the uniform distribution is the hardest distribution to test against.

From the observation that the sample complexity might take values substantially different from that of the worst case, the concept of minimaxity has been refined in recent lines of research. One such refinement corresponds to local minimaxity, also known as instance-optimality, where the local minimax sample complexity depends on $\pi$. Local minimax sample complexity and separation distance for identity testing are presented in Table~\ref{tab:idTestLocal}. \citep{valiant2017automatic}
obtains the local minimax sample complexity for Problem~\eqref{eq:idTest}. \citep{balakrishnan2017hypothesis} makes their test more practical and expresses the rate in terms of local minimax separation distance. The reformulation of their bounds, presented in Table~\ref{tab:idTestLocal}, comes from our Proposition~\ref{th:fromVal}. Note that the dependences in $d$ in the local minimax sample complexity and separation distance are contained in the vector norms. Finally, \citep{balakrishnan2017hypothesis} also obtains the local minimax sample complexity and separation distance for identity testing in the continuous case with Lipschitz densities, but we focus on the discrete case here. 

Let us now consider the literature involving closeness testing, for which we provide a formalization in Equation~\eqref{eq:prob}. The global minimax sample complexity and separation distance are summarized in Table~\ref{tab:probGlobal}, and upper bounds on the local minimax sample complexity and separation distance are given in Table~\ref{tab:probLocal}. In the case of closeness testing, \citep{batu2000testing} proposes a test and obtains a loose upper bound on the global minimax sample complexity. The actual global minimax sample complexity has been identified in \citep{chan2014optimal}, using the tools developed in \citep{valiant2011testing}. A very interesting message from \citep{chan2014optimal} is that there exists a substantial difference between identity testing and closeness testing, and that the latter is harder. It is interesting to note that while the uniform distribution is the most difficult distribution to test in identity testing, $\pi$ can be chosen in a different appropriate way in order to worsen the sample complexity and separation distance in closeness testing. 

Again, distribution-dependent minimax sample complexity and separation distance might differ greatly from global minimax sample complexity and separation distance, respectively. Attempts at obtaining finer results have been made for closeness testing of continuous distributions. Indeed, a large variety of classes of distributions can be defined in the continuous case and it makes sense to obtain minimax rates over rather small classes of distributions. In \citep{diakonikolas2017near}, the authors focus on closeness testing over the class of piecewise constant probability distributions (referred to as $h$-histograms) and obtain minimax near-optimal testers. In the same way, in \citep{diakonikolas2015optimal}, the authors display optimal closeness testers for various families of structured distributions, with an emphasis on continuous distributions. 

Now, as explained in the review of  \citep{balakrishnan2017hypothesisreview}, the definition of local minimaxity in closeness testing is more involved than in identity testing, and it is in fact an interesting open problem that we focus on in this paper. The difficulty arises from the fact that both distributions are unknown, although we would like the minimax sample complexity and separation distance to depend on them. Indeed, in contrast to Problem~\eqref{eq:idTest} whose null hypothesis is simple, Problem~\eqref{eq:prob} is composite-composite. So there is the additional difficulty of having to adapt to the unknown vector $q$. Now, the existence and the size of a difference in the local minimax rates between both problems depending on $\pi$ are open questions. We remind that \citep{chan2014optimal} sheds light on such a gap, but only in the worst case of $\pi$, whereas we look for instance-based minimax optimality. 

\citep{diakonikolas2016new} constructs a test for closeness testing with sample complexity adaptive to one of the distributions $(p,q)$, when either $p=q$ or $\|p-q\|_1 \geq \rho$. Their Proposition~2.14 states that their test achieves a sample complexity of $$\min_m \pa{m + \frac{\|\mathbf 1\{q < 1/m\})\|_1\|q^2 \mathbf 1\{q < 1/m\}\|_1^{1/2}}{\rho^2} + \frac{\|q^{2/3}\|_1^{3/2}}{\rho^2} }.$$ As explained at the end of Section~\ref{sec:setting}, their sample complexity can be translated into a separation distance, which is useful as a comparison with our own results. So in Table~\ref{tab:probLocal}, we present a corollary of their Proposition~2.14 in order to obtain a separation distance corresponding to an upper bound on $\rho_\gamma^*(H^{(\text{Clo})}_0,H^{(\text{Clo})}_1; \pi,k)$ in our setting. Our corollary relies on the definition of $\mathbf P_\pi$. Indeed, $q\in \mathbf P_\pi$ has level sets with similar sizes to those of $\pi$, and therefore $q$ has similar (partial) norms to $\pi$ up to a multiplicative constant. The sample complexity from \cite{diakonikolas2016new} matches the global minimax sample complexity $\sup_\pi k^*_\gamma(H^{(\text{Clo})}_0,H^{(\text{Clo})}_1;\pi,\rho)$ obtained in \citep{chan2014optimal} for some choice of $\pi$ and $m$. However they do not introduce any lower bound dependent on $\pi$, so the only known local lower bound comes from \cite{valiant2017automatic}, which is found for the problem of identity testing in Equation~\eqref{eq:idTest}. But the lower bound from \cite{valiant2017automatic} does not match the upper bound in~\cite{diakonikolas2016new} .

We now mention recent alternative viewpoints on the study of identity and closeness testing. In \citep{acharya2012competitive}, the authors compare their closeness tester against an oracle tester which is given the knowledge of the underlying distribution $q$. When an oracle tester needs $k$ samples, their closeness tester needs $k^{3/2}$ samples. Otherwise, some studies have been made in closeness testing when the number of sample points for both distributions is not constrained to be the same in \citep{bhattacharya2015testing}, \citep{diakonikolas2016new} and \citep{kim2018robust}. What is more, \citep{diakonikolas2017sample} works on identity testing in the high probability case, instead of a fixed probability as it is done usually. That is to say, the authors introduce a global minimax optimal identity tester which discriminates both hypotheses with probability converging to $1$.




\subsection{Contributions}
\label{sec:contributions}

The following are the major contributions of this work:

\begin{itemize}
\item We provide a lower bound on the local minimax separation distance for closeness testing presented in Equation~\eqref{eq:prob} -- see Equation~\eqref{eq:minmax} for $u=2.001$. 
\item We propose a test providing an upper bound that nearly matches the obtained lower bound for $u=1/2$. 
So it is local minimax near-optimal for closeness testing, but the test is also practical, since it does not take $\pi$ as a parameter even though the upper bound optimally depends on $\pi$.
\item We point out the similarities and differences in regimes with local minimax identity testing.
\end{itemize}
More precisely we prove in Theorems~\ref{th:upper} and \ref{th:lowerall} that the local minimax separation distance $\rho_\gamma^*(H^{(\text{Clo})}_0,H^{(\text{Clo})}_1; \pi,k)$ up to some poly$\log(dk)$ is
\begin{align}
\label{eq:minmax}
\begin{split}
&\min_{I \geq J_\pi}\Bigg[\frac{\sqrt{I}}{k}
\lor \Big(\sqrt{\frac{I}{k}}\|\pi^2\exp(-uk\pi)\|_1^{1/4}\Big) \lor  \|(\pi_{(i)}\mathbf 1\{i \geq I\})_i\|_1 \Bigg]\\
&\lor \frac{\Big\|(\pi_{(i)}^{2/3}\mathbf 1\{i \leq J_\pi\})_i\Big\|_1^{3/4}}{\sqrt{k}} \lor \sqrt{\frac{1}{k}},
\end{split}
\end{align}
where $J_\pi$ and $\pi_{(.)}$ are defined in Section~\ref{sec:setting}, $u=2.001$ for the lower bound and $u=1/2$ for the upper bound. The exponential and the powers are applied element-wise.
Let $I^*$ denote an $I$ where the minimum in Equation~\eqref{eq:minmax} is reached.

The local minimax separation distance $\rho_\gamma^*(H^{(Id)}_0,H^{(Id)}_1; \pi,k)$ obtained in \citep{valiant2017automatic,balakrishnan2017hypothesis} is
$$
\min_{m}\cro{\frac{\|(\pi_{(i)}^{2/3}\mathbf 1\{2\leq i< m\})_i\|_1^{3/4}}{\sqrt{k}} \lor \frac{1}{k} \vee \|(\pi_{(i)}\mathbf 1\{i\geq m\})_i\|_1}.
$$ 
We compare it with Equation~\eqref{eq:minmax}. Indeed, as explained in Proposition~\ref{th:fromVal}, $\rho_\gamma^*(H^{(Id)}_0,H^{(Id)}_1; \pi,k)$ also represents a lower bound on $\rho_\gamma^*(H^{(\text{Clo})}_0,H^{(\text{Clo})}_1; \pi,k)$. Let $m^*$ denote an $m$ where the minimum is reached.

Table~\ref{tab:rates} references the local minimax optimal separation distance we obtain for the closeness testing problem defined in Equation~\eqref{eq:prob} and compares it with the upper bound from \citep{diakonikolas2016new} and the local minimax optimal separation distance for identity testing found in \citep{valiant2017automatic,balakrishnan2017hypothesis}. In order to build Table~\ref{tab:rates}, 
we classify the coefficients of $\pi$ depending on their size and the corresponding contribution to the separation distance. As illustrated by the table, our local minimax separation distance fleshes out three main regimes. Looking at the coefficients of $\pi$, for the indices smaller than $J_\pi$, the part of the separation distance corresponding to them is $$\frac{\Big\|(\pi_{(i)}^{2/3}\mathbf 1\{i \leq J_\pi\})_i\Big\|_1^{3/4}}{\sqrt{k}}.$$ As for the indices greater than $I^*$, the part of the separation distance corresponding to them is $ \|(\pi_{(i)}\mathbf 1\{i \geq I\})_i\|_1$. And so regarding the contribution of the indices smaller than $J_\pi$ or greater than $m^*$ to the separation distance, the regimes are the same as in the local minimax separation distance $\rho_\gamma^*(H^{(Id)}_0,H^{(Id)}_1; \pi,k)$ for identity testing from \citep{valiant2017automatic}. However regarding the coefficients corresponding to the indices between $J_\pi$ and $I^*$ the local minimax separation distance for closeness testing is not of same order as for identity testing. The difference concerning local minimax separation distances between identity testing (\citep{valiant2017automatic,balakrishnan2017hypothesis}) and closeness testing then lies in the indices between $J_\pi$ and $m^*$. \citep{chan2014optimal} also notes a difference in the global minimax rates between both problems and we have refined this intuition to make it depend on $\pi$.

We now detail the comparison with the paper~\citep{diakonikolas2016new} that also studies the problem of local closeness testing. The authors of \citep{diakonikolas2016new} also obtain an upper bound on the local minimax separation distance  for closeness testing. Although it is adaptive to $\pi$ and matching the one from \citep{chan2014optimal} in the worst case, their upper bound is not local minimax optimal. In fact, they capture two of the three different phases we describe in Table~\ref{tab:rates}. But their regime corresponding to the very small coefficients, with indices greater than $I^*$, can be made tighter, matching the local minimax separation distance in identity testing.

We further illustrate the difference between the local minimax separation distance that we present and the upper bound from \cite{diakonikolas2016new} with the following example. Take $d = k^4 + 2$ and $0<h<1/2$.
 Let $\pi_1 = 1/2$, $\pi_2 = 1/2 - h$, for any $3\leq i \leq d,$ $\pi_i = h/k^4$. Then, up to multiplicative constants  depending on $\gamma$, our paper leads to the minimax separation distance $1/\sqrt{k} + h$, whereas \cite{diakonikolas2016new} obtains 
$k^{1/2}$ which leads to the trivial upper bound of $1$. This highlights a gap in their upper bound with respect to the local minimax rate in some specific regimes. However, our main contribution with respect to~\citep{diakonikolas2016new}  and to the rest of the literature is our lower bound. It is the evidence from a local perspective that two sample testing is more difficult than identity testing.

Our paper is organized as follows. In Section~\ref{sec:upper}, an upper bound on the local minimax separation distance for Problem~\eqref{eq:prob} is presented. This will entail the construction of a test based on multiple subtests. In Section~\ref{sec:lower}, a lower bound that matches the upper bound up to logarithmic factors is proposed. Finally, the Appendix contains the proofs of all the results presented in this paper.

\begin{table}[h]
\caption{Comparison of the upper bounds on the local minimax separation distances depending on which term dominates. Note that $J_\pi \leq I^* \leq m^*$, by definition of these quantities. Each index $i$ belongs to some index range $U$ and $\pi_{(i)}$ contributes to the separation distance rate differently depending on the index range $U$. The notation $|U|$ refers to the number of elements in $U$. Current paper corresponds to the local minimax separation distance that we prove for closeness testing as defined in Equation~\eqref{eq:prob} with $u=2.001$ for the lower bound and $u = 1/2$ for the upper bound. \citep{valiant2017automatic,balakrishnan2017hypothesis} present the local minimax separation distance found for identity testing defined in Equation~\eqref{eq:idTest}, which corresponds to a lower bound for Problem~\eqref{eq:prob}. \citep{diakonikolas2016new} presents a rate which corresponds to an upper bound for Problem~\eqref{eq:prob}. All the separation distances are presented up to $\log$-factors. $\mathbf 1_U$ is the indicator function of $U$ applied elementwise.}
\label{tab:rates}
\centering
\begin{tabular}{|c||c|c|}
\hline
Index range & $U=\{1,\ldots,J_\pi\}$    & $U=\{J_\pi, \ldots,I^*\}$ \\ \hline
Contribution of the terms & \multicolumn{2}{c|}{} \\ \hline \hline
Current paper &  $\frac{\Big\|\pi_{(.)}^{2/3}\mathbf 1_U\Big\|_1^{3/4}}{\sqrt{k}}$ & $\sqrt{\frac{|U|}{k}}\Big[\|\pi^2\exp(-uk\pi)\|_1^{1/4} \lor \frac{1}{\sqrt{k}}\Big]  $ \\ \hline
Lower bound from \citep{valiant2017automatic,balakrishnan2017hypothesis} &  $\frac{\Big\|\pi_{(.)}^{2/3}\mathbf 1_U\Big\|_1^{3/4}}{\sqrt{k}}$ & $\frac{\Big\|\pi_{(.)}^{2/3}\mathbf 1_U\Big\|_1^{3/4}}{\sqrt{k}}$ \\ \hline
Upper bound from \citep{diakonikolas2016new} &  $\frac{\Big\|\pi_{(.)}^{2/3}\mathbf 1_U\Big\|_1^{3/4}}{\sqrt{k}}$ & $\sqrt{\frac{|U|}{k}} \|\pi_{(.)}^2 \mathbf 1_U\|_1^{1/4}$ \\ \hline
\end{tabular}
\begin{tabular}{|c||c|c|}
\hline
Index range &  $U=\{I^*, \ldots,m^*\}$ & $U=\{m^*,\ldots,d\}$ \\ \hline
Contribution of the terms & \multicolumn{2}{c|}{}  \\ \hline \hline
Current paper & $\Big\|\pi_{(.)}\mathbf 1_U\Big\|_1$ & $\Big\|\pi_{(.)}\mathbf 1_U\Big\|_1$\\ \hline
Lower bound from \citep{valiant2017automatic,balakrishnan2017hypothesis}  & $\frac{\Big\|\pi_{(.)}^{2/3}\mathbf 1_U\Big\|_1^{3/4}}{\sqrt{k}}$ & $\Big\|\pi_{(.)}\mathbf 1_U\Big\|_1$\\ \hline
Upper bound from \citep{diakonikolas2016new} & $\sqrt{\frac{|U|}{k}} \|\pi_{(.)}^2 \mathbf 1_U\|_1^{1/4}$ & $\sqrt{\frac{|U|}{k}} \|\pi_{(.)}^2 \mathbf 1_U\|_1^{1/4}$\\ \hline
\end{tabular}
\end{table}

%% file: results.tex
\section{Upper bound}\label{sec:upper}

In this section, we build a test composed of several tests for Problem~\eqref{eq:prob}. One of them is related to the test introduced in the context of identity testing in~\cite{valiant2017automatic,balakrishnan2017hypothesis}. The others complement this test, in particular regarding what happens for smaller masses. 
Here, as explained in the setting, we observe the following independent sample sets.
$$
\mathcal X \sim \mathcal{M}(k,p), \quad \mathcal Y \sim \mathcal{M}(k,q).
$$
Assume from now on that $k \geq 3$. These two sample sets can each be split into $3$ independent sample sets. That is, for any $j \leq 3$ and $i\leq d$, we consider the independent sample sets
$$
\mathcal X^{(j)} \sim \mathcal{M}(\bar k,p), \quad \mathcal Y^{(j)} \sim \mathcal{M}(\bar k,q),
$$ 
where $\bar k = \floor{k/3}$. 

We will then apply a Poissonization trick in order to consider independent Poisson random variables instead of independent multinomial random variables -- see Section~\ref{ss:poisson} in the Appendix for the precise derivations. Firstly, for $j \in \{1, 2,3\}$ and $m \in \{1,2\}$, let $\bar k_m^{(j)}$ follow $\mathcal P(2\bar k/3)$ independently. Note that by concentration of Poisson random variables, we have with probability larger than $1 - 6 \exp(-\bar k/12)$, that $\bar k_m^{(j)} \leq \bar k$ for all $j \in \{1, 2,3\}$ and $m \in \{1,2\}$ at the same time. With this in mind, we define the following $6d$ counts. For any $i \leq d$ and $j \in \{1, 2,3\}$, let
$$
X_i^{(j)} = \sum_{r \leq \bar k_1^{(j)} \land \bar k} \mathbf 1\{\mathcal X^{(j)}_r = i\}, \quad  Y^{(j)}_i = \sum_{r \leq \bar k_2^{(j)}\land \bar k} \mathbf 1\{\mathcal Y^{(j)}_r = i\}.
$$
By definition of Poisson random variables, on the large probability event such that $\bar k_m^{(j)} \leq \bar k$ for all $j \in \{1, 2,3\}$ and $m \in \{1,2\}$ at the same time, we have that the $X_i^{(j)}$ coincide with independent $\mathcal P(2p_i \bar k/3)$ and that $Y_i^{(j)}$ coincide with independent $\mathcal P(2q_i \bar k/3))$, for $i\leq d$ and $j\leq 3$. Sample splitting and Poissonization allow for simpler derivations of guarantees for tests and they can be done without loss of generality in our setting. That is why we will construct our tests based on $(X^{(j)}, Y^{(j)})_{j \leq 3}$. All the probability statements in this section will be with respect to $(\bar k_m^{(j)})_{m \leq 2, j \leq 3}$ and $(X^{(j)}, Y^{(j)})_{j \leq 3}$.


Sections~\ref{ss:pretest}--\ref{ss:L1} introduce the individual tests as well as their guarantees. Section~\ref{ss:testFinal} combines all the tests into one and produces a problem-dependent upper bound for our setting.
The proofs for the upper bound are compiled in Appendix~\ref{sec:proofUpper}.





The general strategy behind our construction is to readjust the test presented in \cite{valiant2017automatic} to make it fit to our setting. Indeed, both distributions are unknown in our case, making instance-based minimax optimality all the more complicated. Instead of knowing $\pi$ directly, it is estimated up to some multiplicative constant when possible. This will induce a gap with the local minimax separation distances for identity testing presented in \cite{valiant2017automatic,balakrishnan2017hypothesis}. Using other tests, we offset this gap partially. However, as shown in the lower bound, the upper bound is local minimax optimal and the difference in separation distances highlighted in the upper bound is actually fundamental to the problem of closeness testing, making it harder than identity testing in some regimes.

\subsection{Pre-test: Detection of divergences coordinate-wise}
\label{ss:pretest}

We first define a pre-test. It is an initial test designed to detect cases where some coordinates of $p$ and $q$ are very different from one another. It relies on the $L_{\infty}$-distance between the observations.

Let $c>0$, $\hat{q} = (Y^{(3)} \lor 1)/\bar k$ and  $\hat{p} = (X^{(3)} \lor 1)/\bar k$, where the maximum is taken element-wise.
The pre-test is defined as 
$$
\varphi_\infty(X^{(3)}, Y^{(3)},c,k,d) = \begin{cases}
1,\text{ if there exists } i: |\hat p_i - \hat{q}_i| \geq c\sqrt{\frac{\hat q_i \log(\hat q_i^{-1} \land k)}{k}} + c\frac{ \log(k)}{k}.\\
0,\text{  otherwise.}
\end{cases}
$$
In order to simplify the notations, we will just write $\varphi_\infty(c)$ in the future. 

\begin{proposition}\label{th:pretest}
Let $\delta \in (0,1)$. Then there exist $c_{\delta,\infty}>0, \tilde c_{\delta,\infty}>0$ large enough depending only on $\delta$ such that the following holds.
\begin{itemize}
\item  If $p=q$, then with probability larger than $1 - 2\delta - 7k^{-1} - 6\exp(-k/100)$,
$$\varphi_{\infty}(c_{\delta,\infty})=0.$$ 
\item If there exists $i \leq d$ such that 
$$|p_i - q_i| \geq \tilde c_{\delta,\infty} \sqrt{q_i \frac{\log(q_i^{-1} \land k)}{k}} + \tilde c_{\delta,\infty}\frac{\log(k)}{k},$$
then with probability larger than $1 - 2\delta - 7k^{-1} - 6\exp(-k/100)$,
$$\varphi_{\infty}(c_{\delta,\infty})=1.$$
\end{itemize}

\end{proposition}

\subsection{Definition of the $2/3$-test on large coefficients}\label{ss:23}
We now consider a test that is related to the one in~\cite{valiant2017automatic,balakrishnan2017hypothesis} based on a weighted $L_{2}$-norm. But here, the weights are constructed empirically. Such an empirical twist on an existing test in order to obtain adaptive results was also explored in~\cite{diakonikolas2016new}. The objective is to detect differences in the coefficients that are larger than $1/k$ in an efficient way.

Let $c>0$. Set
\begin{equation}
\label{eq:T1}
T_{2/3} = \sum_{i \leq d} \hat{q}_i^{-2/3} (X^{(1)}_i-Y^{(1)}_i) (X^{(2)}_i-Y^{(2)}_i).
\end{equation}
We also define $\hat t_{2/3} =  \sqrt{k^{-2/3} \|(Y^{(1)})^{2/3}\|_1} + 1,$  and 
$$\varphi_{2/3}(c) := \varphi_{2/3}(X^{(1)}, Y^{(1)},X^{(2)}, Y^{(2)}, X^{(3)}, \hat q, c, k, d) = \mathbf 1\{T_{2/3} \geq c \hat t_{2/3}\}.$$
\begin{proposition}\label{th:test1}
Let $\delta >0$. Let $c_{\delta,\infty}$ defined as in Proposition~\eqref{th:pretest} and $k \geq 4\pa{80e^4/\sqrt{\delta/2}}^3$. Then there exist $c_{\delta,2/3}>0, \tilde c_{\delta,2/3}>0$ large enough depending only on $\delta$ such that the following holds.
\begin{itemize}
\item  If $p=q$, then with probability larger than $1 - 3\delta - 7 k^{-1} - 6\exp(-k/100)$,
$$\varphi_{2/3}(c_{\delta,{2/3}})=0~~~~~\mathrm{and}~~~~~\varphi_\infty(c_{\delta,\infty}) = 0.$$ 
\item If
$$\Big\| (p-q) \mathbf 1\{kq \geq 1\} \Big\|_1^2 \geq \frac{\tilde c_{\delta,{2/3}}}{k} \Big(\Big\|q^2 \frac{1}{(q\lor k^{-1})^{4/3}}\Big\|_1^{3/2}\lor \Big\|q^2 \frac{1}{(q\lor k^{-1})^{4/3}}\Big\|_1\Big),$$
then with probability larger than $1 - 3\delta - 7k^{-1} - 6\exp(-k/100)$,
$$\varphi_{2/3}(c_{\delta,{2/3}})=1~~~~~\mathrm{or}~~~~~\varphi_\infty(c_{\delta,\infty}) = 1.$$
\end{itemize}
\end{proposition}
This proposition provides an upper bound on the local minimax separation distance. It is related to the upper bound on the local minimax sample complexity obtained in Proposition~2.14 in~\cite{diakonikolas2016new}, and Table~\ref{tab:rates} makes a detailed comparison between the results obtained in~\cite{diakonikolas2016new} and ours. The authors of \cite{diakonikolas2016new} partition the distribution into different empirical level sets. Then for each level set, they apply a standard $L_2$-test to the pseudo-distributions restricted to that level set. In contrast, we apply only one test with appropriate weights. This is analog with comparing the max test to the $2/3$ test both depicted in \cite{balakrishnan2017hypothesis}. In the test statistic from \cite{diakonikolas2016new}, the partitioning of the distributions is empirical. Comparatively, we modified the $2/3$-test statistic in order to make the weights empirical.
\begin{remark}
	\begin{itemize}
		\item 	Let us compare $T_{2/3}$ defined in Equation~\eqref{eq:T1} with the test statistic presented in \cite{valiant2017automatic}: 
		\begin{equation}
			\label{eq:23IdTest}
				\sum_i \frac{(X_i - kq_i)^2 - X_i}{q_i^{2/3}}.
		\end{equation}
		We start by explaining their construction. Equation~\eqref{eq:23IdTest} is a modified chi-squared statistic producing a \textit{local} minimax optimal test for identity testing. Now, in closeness testing, $q$ is unknown. That is the reason why we estimate $q$ using $\hat q$ and ensure its value cannot be $0$ in the denominator. This constraint leads to rates in separation distance which are different from those obtained with the statistic from Equation~\eqref{eq:23IdTest} for identity testing. Our other tests tackle the case corresponding to $Y^{(3)}=0$ as well as possible, but the rates will remain worse than those in identity testing. Such a gap will prove to be intrinsic to closeness testing as we find a lower bound our upper bound.
		\item The threshold $\hat t_{2/3}$ associated with the definition of $\varphi_{2/3}$ is stochastic. But if $\pi$ is known, then the problem can be reduced to identity testing and the threshold can be made deterministic as in \cite{valiant2017automatic,balakrishnan2017hypothesis}, with value $\sqrt{\sum \pi_i^{2/3}} + 1$.
	\end{itemize}
\end{remark}


\subsection{Definition of the $L_2$-test for intermediate coefficients}
\label{ss:L2}

We now construct a test for intermediate coefficients, i.e., those that are too small to have weights computed in a meaningful way using the method in Section~\ref{ss:23}. For these coefficients, we simply suggest an $L_2$-test that is related to the one carried out in~\cite{chan2014optimal,diakonikolas2016new}. And we apply this test only on coordinates that we empirically find as being small.

Set 
\begin{equation}
	\label{eq:T2}
	T_2 = \sum_{i \leq d}   (X^{(1)}_i-Y^{(1)}_i) (X^{(2)}_i-Y^{(2)}_i) \mathbf 1\{Y^{(3)}_i = 0\},
\end{equation}
and
$$\hat t_2 = \sqrt{\|Y^{(1)}Y^{(2)} \mathbf 1\{Y^{(3)}=0\}\|_1} + \log(k)^2.$$ Write
$\varphi_2(c) := \varphi_2(X^{(1)}, Y^{(1)},X^{(2)}, Y^{(2)}, Y^{(3)}, c, k, d) = \mathbf 1\{T_2 \geq c \hat t_2\}.$

\begin{proposition}\label{th:test2}
Let $\delta \in (0,1)$. Let $c_{\delta,\infty}$ defined as in Proposition~\eqref{th:pretest} and assume $\varphi_\infty(c_{\delta,\infty}) = 0$. We write $s(.)$ such that $q_{s(.)} = q_{(.)}$. There exist $c_{\delta,2}>0, \tilde c_{\delta,2}>0$ large enough depending only on $\delta$ such that the following holds.
\begin{itemize}
\item  If $p=q$, then with probability larger than $1 - 3\delta - 7k^{-1} - 6\exp(-k/100)$,
$$\varphi_{2}(c_{\delta,{2}})=0~~~~~\mathrm{and}~~~~~\varphi_\infty(c_{\delta,\infty}) = 0.$$ 
\item If there exists $I \geq J_q$ such that
$$
\pa{\sum_{i = J_q}^I |p_{s(i)} - q_{s(i)}|}^2 \geq c_{\delta,2} \frac{I-J_q}{k}\Big[\frac{\log^2(k)}{k} \lor \Big(\sqrt{\|q^2\exp(-kq)\|_1}\Big)\Big],
$$
then with probability larger than $1 - 3\delta - 7k^{-1} - 6\exp(-k/100)$,
$$\varphi_{2}(c_{\delta,{2}})=1~~~~~\mathrm{or}~~~~~\varphi_\infty(c_{\delta,\infty}) = 1.$$
\end{itemize}
\end{proposition}

This test based on the $L_2$-statistic tackles a particular regime where the coefficients of the distribution $q$ are neither too small nor too large. 
Such an application of an $L_2$-statistic to an $L_1$-closeness testing problem is reminiscent of \cite{chan2014optimal,diakonikolas2016new}. In particular, like in \cite{diakonikolas2016new}, we restrict the application of this test to a section of the distribution that is constructed empirically.

\begin{remark}
	\begin{itemize}
		\item Note that $T_2$ defined in Equation~\eqref{eq:T2} is based on the $L_2$-separation between both samples in the same way as $T_{2/3}$ defined in Equation~\eqref{eq:T1}. However, $T_2$ is not reweighted since it focuses on the case when $Y^{(3)}=0$. This is comparable to the test statistic presented in \cite{diakonikolas2016new}. Indeed, their statistic is not rescaled using  the values of $q$, but they partition it regrouping coefficients of $q$ of the same order instead. Our statistic amounts to doing just that, except that we focus on smaller coefficients only and we partition $q$ empirically.
		\item Once again, we define an empirical threshold $\hat t_2$. With the knowledge of $\pi$, we would obtain the following deterministic threshold instead: $\sqrt{\sum_{\pi_i \leq k} (k \pi_i)^2} + \log^2(k)$.
	\end{itemize}
\end{remark}

\subsection{Definition of the $L_1$-test for small coefficients}
\label{ss:L1}

Finally we define another test to exclude situations where the $L_1$-norm of the small coefficients in $p$ and $q$ are very different.

Set
$$T_1 = \sum_{i \leq d} (X^{(1)}_i-Y^{(1)}_i) \mathbf 1\{Y^{(3)}_i = 0\}.$$
Write $\varphi_1(c) := \varphi_1(X^{(1)}, Y^{(1)}, Y^{(3)}, c, k, d) = \mathbf 1\{T_1 \geq c \sqrt{k}\}$. 

\begin{proposition}\label{th:test3}
Let $\delta \in (0,1)$. Let $c_{\delta,\infty}, \tilde c_{\delta,\infty}$ defined as in Proposition~\eqref{th:pretest}. Assume $k \geq 13 \delta^{-1}(1+ 9 \tilde c_{\delta,\infty} \log (k/3) /2)^2$. We write $s$ such that $q_{s(.)} = q_{(.)}$. Then there exist $c_{\delta,1}>0, \tilde c_{\delta,1}>0$ large enough depending only on $\delta$ such that the following holds.
\begin{itemize}
\item  If $p=q$, then with probability larger than $1 - 3\delta - 7k^{-1} - 6\exp(-k/100)$,
$$\varphi_{1}(c_{\delta,{1}})=0~~~~~\mathrm{and}~~~~~\varphi_\infty(c_{\delta,\infty}) = 0.$$ 
\item If  
$$\Big\| (p-q) \mathbf 1\{kq \geq 1\} \Big\|_1^2 \geq \frac{\tilde c_{\delta,{1}}}{k} \Big(\Big\|q^2 \frac{1}{(q\lor k^{-1})^{4/3}}\Big\|_1^{3/2}\lor \Big\|q^2 \frac{1}{(q\lor k^{-1})^{4/3}}\Big\|_1\Big),$$
then with probability larger than $1 - 3\delta - 7k^{-1} - 6\exp(-k/100)$,
$$\varphi_{1}(c_{\delta,{1}})=1~~~~~\mathrm{or}~~~~~\varphi_\infty(c_{\delta,\infty}) = 1.$$
\end{itemize}
\end{proposition}

As stated in Proposition~\ref{th:test3}, this test captures the case of large $L_1$-deviation at places where $p$ and $q$ have small coefficients. This is mainly interesting for cases where there are extremely many small coefficients, making a very crude test the most meaningful tool to use. The pathological cases addressed here contribute to the differences in separation distances with \cite{diakonikolas2016new}.

\subsection{Combination of the four tests}
\label{ss:testFinal}

To conclude, we combine all four tests by taking the maximum value that they output, effectively rejecting the null hypothesis whenever one of the tests is rejected.

Let
$\varphi(c_\infty,c_{2/3},c_2,c_1) = \varphi_{\infty}(c_{\infty}) \lor \varphi_{2/3}(c_{2/3}) \lor \varphi_2(c_2) \lor \varphi_1(c_1),$
where $c_{\infty}, c_{2/3},c_{2},c_{1}>0$.

\begin{theorem}\label{th:upper}
Let $\delta < 1$. There exist $c_{\delta,\infty},c_{\delta,{2/3}},c_{\delta,2},c_{\delta,1}, \tilde c_{\delta,\infty}, \tilde c_\delta >0$ that depend only on $\delta$ such that the following holds. Let $k \geq \cro{13 \delta^{-1}(1+ 9 \tilde c_{\delta,\infty} \log (k/3) /2)^2} \vee \cro{3\pa{80e^4/\sqrt{\delta/2}}^3}$.

\begin{itemize}
\item  If $p=q$, then with probability larger than $1-5\delta-7k^{-1}- 6\exp(-k/100)$,
$$\varphi(c_{\delta,\infty},c_{\delta,2/3},c_{\delta,2},c_{\delta,1})=0.$$ 
\item If  
\begin{align*}
&\sum |p_i-q_i| \\
&\quad\quad\geq \tilde c_\delta \Bigg\{\min_{I \geq J_\pi}\Bigg[\Big(\sqrt{I}\frac{\log(k)}{k}\Big)
\lor \Big(\frac{\sqrt{I}}{\sqrt{k}}\|q^2\exp(-kq)\|_1^{1/4}\Big) \lor  \|(q_{(i)}\mathbf 1\{i \geq I\})_i\|_1 \Bigg] \Bigg\} \\
&\quad\quad\quad\quad\quad\quad\quad\lor \Big[ \frac{\Big\|q^{2} \frac{1}{(q \lor k^{-1})^{4/3}}\Big\|_1^{3/4}}{\sqrt{k}} \Big] \lor \Big[ \sqrt{\frac{\log(k)}{k}}\Big].
\end{align*}
then with probability larger than $1-5\delta-7k^{-1}- 6\exp(-k/100)$,
$$\varphi(c_{\delta,\infty},c_{\delta,2/3},c_{\delta,2},c_{\delta,1})=1.$$
\end{itemize}
\end{theorem}
Then the theorem can be formulated as the following upper bound.

\begin{corollary}\label{cor:upper}
Let $\gamma>0$. There exists a constant $c_\gamma>0$ that depends only on $\gamma$ such that

\begin{align*} 
\rho_\gamma^*(&H^{(\emph{Clo})}_0,H^{(\emph{Clo})}_1; \pi,k)\\
 &\leq c_\gamma  \Bigg\{\min_{I \geq J_\pi}\Bigg[\Big(\sqrt{I}\frac{\log(k)}{k}\Big)
\lor \Big(\frac{\sqrt{I}}{\sqrt{k}}\|\pi^2\exp(-k\pi/2)\|_1^{1/4}\Big) \lor  \|(\pi_{(i)}\mathbf 1\{i \geq I\})_i\|_1 \Bigg]\Bigg\} \\
&\lor \left[ \frac{\Big\|(\pi_{(i)}^{2/3}\mathbf 1\{i \leq J_\pi\})_i \Big\|_1^{3/4}}{\sqrt{k}} \right] \lor \Big[ \sqrt{\frac{\log(k)}{k}}\Big].
\end{align*}

\end{corollary}
Thus, once we have aggregated all four tests, we end up with an upper bound on the local minimax separation distance for closeness testing defined in Equation~\eqref{eq:prob}. Most importantly, the knowledge of $\pi$ is not exploited by the test. So our method reaches the displayed rate adaptively to $\pi$. That is, the separation distance does not just consider the worst $\pi$. Instead, it depends on $\pi$ although it is not an input parameter in the test.
In Table~\ref{tab:rates}, the contributions of the different coefficients from $\pi$ are summarized into different regimes, along with the regimes obtained in \cite{valiant2017automatic} and \cite{diakonikolas2016new}. Our upper bound improves upon that of \cite{diakonikolas2016new} as emphasized in Section~\ref{sec:contributions}. We manage to obtain separation distances comparable to those found in identity testing defined in Equation~\eqref{eq:idTest}. In particular, the terms $\|(\pi_{(i)}\mathbf 1\{i \geq I\})_i\|_1$ and $\frac{\Big\|(\pi_{(i)}^{2/3}\mathbf 1\{i \leq J_\pi\})_i \Big\|_1^{3/4}}{\sqrt{k}}$ can also be found in identity testing. However, the differences that we point out in the upper bound turn out to be fundamental to closeness testing. Indeed, we present a matching lower bound in the following section, which represents our main contribution.

\section{Lower bound}
\label{sec:lower}
This section will focus on the presentation of a lower bound on the local minimax separation distance for closeness testing defined in Equation~\eqref{eq:prob}. Since the lower bound will match the upper bound previously presented, our test will turn out to be local minimax optimal.

\begin{theorem}\label{th:lowerall}
Let $\pi\in \mathbf{P}$ and $\gamma,v>0$. Assume $k \geq 2^8.$ There exists a constant $c_{\gamma,v}>0$ that depends only on $\gamma,v$ such that the following holds.
\begin{align*}
\rho_\gamma^*(H^{(\emph{Clo})}_0,H^{(\emph{Clo})}_1; \pi,k) &\geq  c_{\gamma,v}  \Bigg\{\min_{I \geq J_\pi}\Bigg[\frac{\sqrt{I}}{k}
\lor \Big(\sqrt{\frac{I}{k}}\|\pi^2\exp(-(2+v)k\pi)\|_1^{1/4}\Big) \\
&\lor  \|(\pi_{(i)}\mathbf 1\{i \geq I\})_i\|_1 \Bigg]\Bigg\}\lor \frac{\Big\|(\pi_{(i)}^{2/3}\mathbf 1\{i \leq J_\pi\})_i \Big\|_1^{3/4}}{\sqrt{k}} \lor \sqrt{\frac{1}{k}}.
\end{align*}
\end{theorem}
The details of the proof can be found in Section~\ref{sec:proofLower} of the Appendix. But we provide the intuition through the following sketch of the proof.

\paragraph{Sketch of the proof of Theorem~\ref{th:lowerall}.}
The construction of the lower bound can be decomposed into three propositions.
We first state Proposition~\ref{th:fromVal}, which is a corollary from \cite{valiant2017automatic,balakrishnan2017hypothesis} and it will provide an initial lower bound on the local minimax separation distance. We will refine this lower bound using Propositions~\ref{th:lower} and \ref{th:lower2}, which constitute our main contributions. The general strategy is the same for both propositions. At first, we reduce the testing problem to a smaller one that is difficult enough and which is not yet covered by Proposition~\ref{th:fromVal}. Afterwards, the idea is to hide the discrepancies between distributions in the smaller coefficients, which is justified by the thresholding effect already witnessed in the upper bound. Indeed coefficients corresponding to low probabilities have a great chance of generating 0's. So the information on the coefficients being small to different degrees is lost.
\endproof
\medskip

\noindent
Proposition~\ref{th:fromVal} relies on the fact that two-sample testing is at least as hard as its one-sample counterpart. It is also the most convenient formulation of the local minimax separation distance from \cite{valiant2017automatic,balakrishnan2017hypothesis} in order to compare it with our results.

\begin{proposition}
\label{th:fromVal}
Let $\pi\in \mathbf{P}$ and $\gamma>0$. There exists a constant $c_\gamma>0$ that depends only on $\gamma$ such that
\begin{align*}
\rho_\gamma^*(H^{(\emph{Clo})}_0,&~H^{(\emph{Clo})}_1; \pi,k) \geq c_\gamma \min_{m}\left[\frac{\|(\pi_{(i)}^{2/3}\mathbf 1\{2 \leq i< m\})_i\|_1^{3/4}}{\sqrt{k}} \lor \frac{1}{k} \vee \|(\pi_{(i)}\mathbf 1\{i\geq m\})_i\|_1\right].
\end{align*}
\end{proposition}

\noindent
The next proposition is a novel construction, which settles the case for small coefficients. 
\begin{proposition}\label{th:lower}
	Consider some $\pi\in \mathbf{P}$ and $\gamma >0$. Set for $v\geq 0$ and with the convention $\min_{j \leq d} \emptyset =d$,
\begin{align*}
I_{v,\pi} &= \min_{J_\pi \leq j \leq d} \Big\{ \{j : \pi_{(j)} \leq \sqrt{C_\pi/j}\} \cap \{ j : \sum_{i\geq j} \exp(-2k\pi_{(i)}) \pi_{(i)}^2 \leq C_\pi\}\\
&\cap \{ j :\sum_{i \geq j} \pi_{(i)} \leq \sum_{J_\pi \leq i < j} \pi_{(i)} \}\Big\},
\end{align*}
where $C_\pi = \frac{\sqrt{\sum_i \pi_i^2\exp(-2(1+v)k\pi_i)}}{k}$.  There exist constants $c_{\gamma,v}, c_{\gamma,v}', c_{\gamma,v}''>0$ that depend only on $\gamma,v$ such that the following holds.
 Assume that $\|\pi^2\exp(-2(1+v)k\pi)\|_2^2 \geq  \frac{c_{\gamma,v}''}{k^2}$ and $k \geq 2^8$. Then
\begin{align*}
\rho_\gamma^*(H^{(\emph{Clo})}_0,&~H^{(\emph{Clo})}_1; \pi,k) \\
&\quad\quad\geq c_{\gamma,v} \Bigg[ \Big[\|(\pi_{(i)} \mathbf 1\{i \geq I_{v,\pi}\})_i\|_1 \lor \frac{I_{v,\pi}-J_\pi}{\sqrt{I_{v,\pi}k}} \|\pi^2 \exp(-2(1+v)\pi)\|_1^{1/4}\Big]\\
&\quad\quad\quad\quad\quad\quad\land \|(\pi_{(i)}\mathbf 1\{i \geq J_\pi\})_i\|_1\Bigg] - \frac{c_{\gamma,v}'}{\sqrt{k}},
\end{align*}
where $J_\pi$ is defined at the end of Section~\ref{sec:setting}.
\end{proposition}
The proof of this proposition and the following one is based on a classical Bayesian approach for minimax lower bounds. It heavily relies on explicit choices of prior distributions over the couples $(p,q)$ either corresponding to hypothesis set $H^{(\text{Clo})}_0(\pi)$ or $H^{(\text{Clo})}_1(\pi, \rho)$. The goal is then to show that the chosen priors are so close that the risk $R(H^{(\text{Clo})}_0,H^{(\text{Clo})}_1,\varphi;\rho, k)$ is at least as large as $\gamma$ for a fixed $k$. Details on the general approach are provided in Appendix~\ref{sec:bayes}.

The brunt of our contribution relies on the definition of appropriate priors. The priors are enforced to have support in $\mathbf P^2$, as detailed in the proof with ideas related to the Poissonization trick. But this is only a technical difficulty which is not fundamental from an information theoretic perspective.
A more crucial step regards constructing prior distributions on "non-normalised" versions of the vectors $(p,q)$. We use the notation $(\tilde p, \tilde q)$ for the "non-normalised" vectors associated with the prior distributions.

Let us now present the prior distributions on the parameters $(\tilde p,\tilde q)$ defined for the proof of Proposition~\ref{th:lower}. $\pi$ and $k$ are fixed, and the priors critically revolve around $\pi$ and perturbations thereof in order to obtain a local minimax optimal lower bound. We start by defining an index set $\mathcal A$ corresponding to a subset of elements of $\pi$ containing a fixed proportion of each significant level set $S_\pi$. Then $\mathcal A$ is a set of indices such that $(\pi_i)_{i \in \mathcal A}$ is a vector with a similar shape to $\pi$ and the elements from $\mathcal A^C$ can be used in order to define normalised $(p,q)$. 

Under both the null and the alternative hypotheses, the prior distributions are defined such that for any $i \in \mathcal A^C$, $\tilde p_i = \tilde q_i = \pi_i$. We now consider the definition of $(\tilde p, \tilde q)$ on $\mathcal A$. Under any of both hypotheses, the elements of $\tilde q$ restricted to $\mathcal A$ are taken at random uniformly from the elements of $\pi$ restricted to $\mathcal A$.
\begin{itemize}
\item Under the null hypothesis, $\tilde p$ is set equal to $\tilde q$.
\item Under the alternative hypothesis, 
$\tilde p$ is a stochastic vector that differs from $\tilde q$ in the following way: 
\begin{itemize}
\item All coordinates larger than $1/k$ are set equal to those of $\tilde q$.
\item  For the other coordinates, set 
$\tilde p_{i} = \tilde q_{i}(1+\xi_i)$, where $\xi_i$ is uniform on $\{-\varepsilon_i^*, \varepsilon_i^*\}$, and $\varepsilon_i^*$ is defined in an implicit way in Lemma~\ref{lem:lower}.
\end{itemize}
\end{itemize}
The quantities $\varepsilon_i^*$'s are defined to satisfy the conditions in Lemma~\ref{lem:lower}. The intuition associated with those conditions are the following.
	\begin{itemize}
		\item There is no deviation for the larger coefficients, i.e., $\tilde p = \tilde q$ for coefficients larger than $1/k$.
		\item The $L_2$-separation and the $L_\infty$-distance between $\tilde p$ and $\tilde q$ are upper bounded with high probability, making the discrepancy hard to detect.
		\item The $L_1$-distance between $\tilde p$ and $\tilde q$  is lower bounded with high probability by the local minimax separation distance to be proven.
		\item The way $\tilde q$ deviates from $\pi$ creates some uncertainty. This makes it difficult to leverage any knowledge on $\tilde q$ for constructing the test besides the fact that (the normalised version of) $\tilde q$ is in $\mathbf P_\pi$.
	\end{itemize}


Finally, the following proposition complements Proposition~\ref{th:lower} in the case where the tail coefficients are very small.
\begin{proposition}\label{th:lower2}
Let $\pi\in \mathbf{P}$ and $\gamma, v>0$. There exist constants $\tilde c_{\gamma,v}, c_{\gamma,v}, c_{\gamma,v}'>0$ that depend only on $\gamma,v$ such that the following holds. Assume that $\|\pi^2 \exp(-2(1+v)\pi)\|_1 \leq \tilde c_{\gamma,v}/k^2$. Then
$$\rho_\gamma^*(H^{(\emph{Clo})}_0,H^{(\emph{Clo})}_1; \pi,k) \geq c_{\gamma,v} \|(\pi_{(i)} \mathbf 1\{i \geq J_\pi)_i\}\|_1 - c_{\gamma,v}'\sqrt{\|(\pi_{(i)}^2 \mathbf 1\{i\geq J_\pi\})_i\|_1},$$ 
\noindent
where $J_\pi$ is defined in Section~\ref{sec:setting}.
\end{proposition}
This proposition refines Proposition~\ref{th:lower} in the specific case where $\|\pi^2 \exp(-2(1+v)\pi)\|_1$ is small, and the construction of the priors is related, but simpler. 
Combining Propositions~\ref{th:fromVal}, \ref{th:lower} and \ref{th:lower2} lead to the lower bound in Theorem~\ref{th:lowerall}.

Thus a lower bound is constructed for the local minimax separation distance, which characterizes the difficulty of closeness testing defined in Equation~\eqref{eq:prob}.  In fact, the lower bound matches the upper bound up to log terms. Thus, we have a good envelope of the local minimax rate. We firstly conclude explicitly that there exist some $\pi$ such that $\rho_\gamma^*(H^{(\text{Clo})}_0,H^{(\text{Clo})}_1; \pi,k) > \rho_\gamma^*(H^{(\text{Id})}_0,H^{(\text{Id})}_1;\pi, k)$, that is, two-sample testing is strictly harder than one-sample testing for some distributions $\pi$. Secondly, the result highlights the location of the gap in further detail than the worst-case study of \cite{chan2014optimal}. We provide a detailed comparison between results in Section~\ref{sec:contributions} using Table~\ref{tab:rates}.




%% file: UBproof.tex
\section{Preliminary results on the Poisson distribution}

The proofs to our theorems will be provided for Poisson distributions which can be translated into results for multinomial distributions. Similar considerations of independent Poisson samples in order to simplify the proofs are made in \cite{chan2014optimal,valiant2017automatic}.

We first provide an equivalence result between the samples from a multinomial distribution and samples from independent Poisson distributions.


\begin{theorem}
\label{th:multinPoisson}
Let $k \in \mathbb R^+$, $p \in \mathbf P$. Let $\hat k\sim \mathcal P (k)$. Let the conditional distribution of $\xi$ be $\mathcal{M}(\hat k, p)$, conditionally on $\hat k$. For any $i \leq d$, we have $X_i = \sum_{j=1}^{\hat k} \mathbf 1\{\xi_j = i\}$. Then we have independent
$$
X_i \sim \mathcal P (kp_i).
$$
\end{theorem}

\begin{proof}
We first show that $X_j \sim \mathcal P (kp_j)$ for any $j$.
Let $\mathbf i$ be such that $\mathbf i^2 = -1$. We write the characteristic function of $X_j$ for any $t$:
$$
\E(e^{\mathbf itX_1}) = \E( \E(e^{\mathbf itX_1} | \hat k)  ) = \sum_{j \geq 0} ((1 + p_1 (e^{\mathbf it} - 1))^j \frac{k^j}{j!} e^{-k} = \exp(kp_1(e^{\mathbf it} - 1)),
$$
which corresponds to the characteristic function associated with $\mathcal P (kp_1)$.

Then let us prove that $(X_j)_{j\leq d}$ are independent.
We have
\begin{align*}
\Po(X_1 = x_1,\ldots,X_d = x_d) &= \sum_{l \geq 0} \Po(X_1=x_1,\ldots,X_d=x_d | \hat k = l) \Po(\hat k = l)\\
&= \sum_l \frac{l!}{x_1! \ldots x_d!} p_1^{x_1} \ldots p_d^{x_d} \mathbf{1}\{\sum_{i=1}^d x_i = l\} \frac{k^l}{l!} e^{-k}\\
&= \frac{(\sum_i x_i)!}{x_1! \ldots x_d!} p_1^{x_1} \ldots p_d^{x_d} \frac{k^{\sum_{i=1}^d x_i}}{(\sum_{i=1}^d x_i)!} e^{-k}\\
&= \prod_{i=1}^d \Po(X_i=x_i).
\end{align*}
\end{proof}
\begin{remark}
Note that for any $\lambda_1 > 0$, there exists $k \in \mathbb R^+$ and $p_1 \geq 0$ such that $\lambda_1 = k p_1$. And for any $c>0$, we have $\lambda_1 = (ck) (p_1/c)$. In particular, there exists $c$ such that $\sum_{i=1}^d p_i/c \leq 1$.
\end{remark}

%
%
%
The following lemma states that Poisson samples concentrate around their mean.
\begin{lemma}
\label{lem:poissonConcentration}
If $Z\sim \mathcal P(\lambda)$, where $\lambda > 0$,
$$\mathbb P(|Z-\lambda|\geq \lambda/2) \leq 2\exp\left(-\frac{\lambda}{12}\right).$$
\end{lemma}

\begin{proof}
If $Z\sim \mathcal P(\lambda)$, where $\lambda > 0$, we have, by concentration of the Poisson random variables, that for any $t \geq 0$,
$$\mathbb P(|Z-\lambda|\geq t) \leq 2\exp\left(-\frac{t^2}{2(\lambda+t)}\right).$$
In particular,
$$\mathbb P(|Z-\lambda|\geq \lambda/2) \leq 2\exp\left(-\frac{\lambda}{12}\right).$$ 
\end{proof}

\section{Proof of the upper bounds: Propositions~\ref{th:pretest}, \ref{th:test1}, \ref{th:test2}, \ref{th:test3} and Theorem~\ref{th:upper}}
\label{sec:proofUpper}

For any $i\in\{1,2,3\}$, we write $\mathbb E^{(i)}, \mathbb V^{(i)}$ for the expectation and variance with respect to $(X^{(i)},Y^{(i)})$ and $\bar k_m^{(i)}$. 
$\mathbb E$ and $\mathbb V$ denote the expectation and variance with respect to all sample sets and all $\bar k_m^{(j)}$. We write for all $i \leq d$, $\Delta_i = p_i - q_i.$ Assume without loss of generality that $q$ is ordered such that $q_1 \geq q_2 \geq \ldots \geq q_d$. We remind the reader about the following notation: $\pi_{(1)} \geq \pi_{(2)} \geq \ldots \geq \pi_{(d)}$. Throughout Section~\ref{sec:proofUpper}, let $I \geq J_q$ and we write $J:=J_q$.


\subsection{From multinomial samples to independent Poisson samples}\label{ss:poisson}
%
%

Let $\hat k \sim \mathcal P(k)$.  We define the following independent random variables $\mathcal Z_1,\ldots,\mathcal Z_{\hat k}$ each taking value in $\{1,\ldots,d\}$ according to the probability vector $p$, and we set $m = \floor{3k/2}\wedge\hat k$.

We define $Z_i = \sum_{j=1}^{m} \mathbf 1\{\mathcal Z_j = i\}$ and $\tilde Z_i = \sum_{j=1}^{\hat k} \mathbf 1\{\mathcal Z_j = i\}$  for any $i \leq d$.
By Theorem~\ref{th:multinPoisson}, we have independent
$$
\tilde Z_i \sim \mathcal{P}(kp_i),
$$
for any $i \leq d.$ Note that $(\tilde Z_i)_i$ coincides with $(Z_i)_i$ on the event where $\hat k \leq \floor{3k/2}$. 

Also, we have by Lemma~\eqref{lem:poissonConcentration}
$$
\mathbb P(\hat k \leq 3k/2) \geq 1-\exp\left(-\frac{k}{12}\right).
$$
And so on an event of probability larger than $1-\exp\left(-\frac{k}{12}\right)$, the $(Z_i)_i$ coincides with the $(\tilde Z_i)_i$, i.e.~with independent $\mathcal{P}(kp_i)$ samples.

Applying this to each of our sample sets $\mathcal X^{(j)}, \mathcal Y^{(j)}$ respectively associated with $\bar k_m^{(j)}$ for $j \in \{1,2,3\}$ and $m \in \{1,2\}$, we finally obtain that on an event of probability larger than $1-6\exp(-\bar k/18)$, $(X^{(j)}_i)_i$ coincides with independent $\mathcal{P}(\bar k p_i/6)$, and $(Y^{(j)}_i)_i$ coincides with independent $\mathcal{P}(2\bar k q_i/3)$.

From this point on, we will therefore assume that
$$
X^{(j)} \sim \mathcal{P}(2\bar k p/3), \quad Y^{(j)} \sim \mathcal{P}(2\bar kq/3),
$$
and that they are independent accross $j$. In what follows we will only consider events intersected with that event of probability larger than $1-6\exp(-k/18)$ where $\bar k_m^{(j)} \leq \bar k$.
In what follows, since we always reason up to multiplicative constants, we will write $k$ instead of $2\bar k/3$ to simplify notations.

\subsection{Proof of Proposition~\ref{th:pretest}}

In order to derive the guarantees on the pre-test stated in Proposition~\ref{th:pretest}, we first provide the following lemma. The deviation from a Poisson random variable to its expected value will be bounded depending on the outcome of the random variable, and then depending on its expected value.

\begin{lemma}\label{lem:pretest}
Let $\lambda \in (\mathbb R^+)^d$ such that $\sum_i \lambda_i = k$. Let independent $Z_i \sim \mathcal P(\lambda_i)$ for any $i \leq d$. Let $\bar z = Z/k$.
Let $\delta \in (0,1)$ and $a:=16\frac{\log(2k/\delta)}{k}$. With probability larger than $1-\delta - k^{-1}$ and for all $i \leq d$ we have 
\begin{align*}
|\bar z_i - \lambda_i/k| &\leq 2\sqrt{\frac{12(\bar z_i \lor a)\log(96\log(2k/\delta)(\bar z_i^{-1} \land a^{-1})/\delta)}{k}} \\
  &\quad+ 2\frac{\log(96\log(2k/\delta)(a^{-1})/\delta)}{k},
\end{align*}
and 
\begin{align*}
&2\sqrt{\frac{12(\bar z_i \lor a)\log(96\log(2k/\delta)(\bar z_i^{-1} \land a^{-1})/\delta)}{k}}+  2\frac{\log(96\log(2k/\delta)(a^{-1}/\delta)}{k}\\
&\quad \quad \quad \quad \quad \quad \quad \quad \quad 
 \leq 12\sqrt{\frac{((\lambda_i/k) \lor a) \log\Big(384\log(2k/\delta)\big((\lambda_i/k)^{-1} \land a^{-1}\big)/\delta\Big)}{k}} \\
&\quad\quad \quad \quad \quad \quad \quad \quad \quad \quad + 2\frac{\log\Big(384\log(2k/\delta)a^{-1}/\delta\Big)}{k}.
\end{align*}
\end{lemma}
So an immediate corollary to this lemma is the following:
\begin{corollary}
With probability larger than $1-2\delta - 2k^{-1}-6\exp(-k/18)$, \co{$\varphi_{\infty}(c_{\delta,\infty})=1$} if there exists $i\leq d$ such that
\begin{itemize}
\item if $q_i \geq 16\frac{\log(2k/\delta)}{k}$:
$$|\Delta_i| \geq 50\sqrt{\frac{q_i\log\Big(384\log(2k/\delta)q_i^{-1}/\delta\Big)}{k}} + 300\frac{\log\Big(384\log(2k/\delta)a^{-1}/\delta\Big)}{k}.$$
\item if $q_i \leq 16\frac{\log(2k/\delta)}{k}$: 
$$|\Delta_i| \geq 50\sqrt{\frac{a \log\Big(384\log(2k/\delta)a^{-1}/\delta\Big)}{k}} + 300\frac{\log\Big(384\log(2k/\delta)a^{-1}/\delta\Big)}{k}.$$
\end{itemize}
If $\Delta=0$, then $\varphi_{\infty}(c_{\delta,\infty})=0$ with probability larger than $1 -2 \delta - 2k^{-1}-6\exp(-k/18)$. 
\end{corollary}
This corollary implies that there exists a universal constant $c>0$ such that the preliminary test rejects the null hypothesis on an event of probability larger than $1-2\delta -2 k^{-1}-6\exp(-k/18)$, when $\Delta_i$ is such that
\begin{align*}
|\Delta_i| \geq c\sqrt{ q_i \frac{\log((q_i^{-1} \land k)/\delta)}{k}} + c\frac{\log( k/\delta)}{k}.
\end{align*}
This leads to the result stated in Proposition~\ref{th:pretest}.

\begin{proof}[Proof of Lemma~\ref{lem:pretest}]
\textbf{Analysis of the small $\lambda_i$'s.}
We consider every $i$ such that $\lambda_i \leq k^{-2}$.
Then for any such $i$,
$$
\Po(\bar z_i > 1/k ) = 1 - (1 + \lambda_i) e^{-\lambda_i} \leq 1 - (1 + \lambda_i)(1-\lambda_i) = \lambda_i^2.
$$
So
$$
\Po(\cup_{j: \lambda_j \leq k^{-2}} \{\bar z_j > 1/k\} ) \leq \sum_j \lambda_j^2 \leq \frac{1}{k^2}\sum_j \lambda_j= 1/k.
$$
So with probability larger than $1-1/k$, we have for all $\lambda_i \leq 1/k^2$ at the same time that
$$
\bar z_i \leq 1/k.
$$
Let $a = 16\frac{\log(2k/\delta)}{k}$. Then with probability larger than $1-1/k$, for every $\lambda_i \leq 1/k^2$ at the same time,
\begin{align*}
|\bar z_i - \lambda_i/k| &\leq 2\sqrt{\frac{12(\bar z_i \lor a)\log(96\log(2k/\delta)(\bar z_i^{-1} \land a^{-1})/\delta)}{k}} \\
  &\quad+ 2\frac{\log(96\log(2k/\delta)(\bar z_i^{-1} \land a^{-1})/\delta)}{k},
\end{align*}
and
\begin{align*}
&2\sqrt{\frac{12(\bar z_i \lor a)\log(96\log(2k/\delta)(\bar z_i^{-1} \land a^{-1})/\delta)}{k}}+  2\frac{\log(96\log(2k/\delta)(\bar z_i^{-1} \land a^{-1})/\delta)}{k}\\
&\quad \quad \quad \quad \quad \quad \quad \quad \quad 
\leq 2\sqrt{\frac{36((\lambda_i/k) \lor a) \log\Big(384\log(2k/\delta)\big((\lambda_i/k)^{-1} \land a^{-1}\big)/\delta\Big)}{k}} \\
&\quad\quad \quad \quad \quad \quad \quad \quad \quad \quad + 2\frac{\log\Big(384\log(2k/\delta)\big((\lambda_i/k)^{-1} \land a^{-1}\big)/\delta\Big)}{k}.
\end{align*}
\textbf{Analysis of the large $\lambda_i$'s.}
We consider every $i$ such that $\lambda_i > k^{-2}$.

If $Z_i\sim \mathcal P(\lambda_i)$, where $\lambda_i > 0$, we have, by concentration of the Poisson random variables, that for any $t \geq 0$,
$$\mathbb P(|Z_i-\lambda_i|\geq t) \leq 2\exp\left(-\frac{t^2}{2(\lambda_i+t)}\right).$$
We set $\tilde \delta_i$ as $2\exp\left(-\frac{t^2}{2(\lambda_i+t)}\right)$, the inequality implies that with probability larger than $1- \tilde \delta_i$,
\begin{equation}
\label{eq:poissDev}
  |\bar z_i - \lambda_i/k| \leq 2\sqrt{\frac{(\lambda_i/k)\log(2/\tilde \delta_i)}{k}} + 2\frac{\log(2/\tilde \delta_i)}{k}.  
\end{equation}
So we write $\tilde \delta_i = \lambda_i \delta/k$. Then, since $\sum_i \lambda_i = k$, we have with probability larger than $1-\delta$, for every $i$ such that $\lambda_i > k^{-2}$ at the same time
\begin{align}
\label{eq:poissDev2}
  |\bar z_i - \lambda_i/k| &\leq 2\sqrt{\frac{(\lambda_i/k)\log(2k/(\lambda_i \delta))}{k}} + 2\frac{\log(2k/(\lambda_i \delta))}{k} \nonumber\\
 &\leq 2\sqrt{\frac{(\lambda_i/k)\log(2k/[(\lambda_i \vee k^{-2}) \delta])}{k}} + 2\frac{\log(2k/[(\lambda_i \vee k^{-2}) \delta])}{k}.  
\end{align}
By considering two subcases, let us prove the following inequality on an event of probability larger than \co{$1-\delta$}, for all $i$ 
\begin{equation}\label{crude}
 ((\lambda_i/k) \vee a)/4 \leq \bar z_i \lor a \leq 3((\lambda_i/k) \lor a).
\end{equation}
\textbf{Subcase $\lambda_i/k \geq a$.}
By Equation~\eqref{eq:poissDev2}, we have on an event of probability larger than \co{$1-\delta$}, that for all $i$ such that $\lambda_i/k\geq 16\frac{\log(2k/\delta)}{k}=a$,
\begin{equation*}
  |\bar z_i - \lambda_i/k| \leq 2\sqrt{\frac{(\lambda_i/k)\log(2k/\delta)}{k}} + 2\frac{\log(2k/\delta)}{k} \leq 5\lambda_i/(8k).
\end{equation*}
So 
$$
\lambda_i/(4k) \leq \bar z_i \leq 3 \lambda_i/k.
$$
That is,
$$
 ((\lambda_i/k) \vee a) /4 \leq \bar z_i \vee a\leq 3 ((\lambda_i/k) \vee a).
$$
\textbf{Subcase $k^{-3}< \lambda_i/k < a$.}
By Equation~\eqref{eq:poissDev2}, we have on the same event of probability larger than \co{$1-\delta$}, that for all $i$ such that $k^{-3}< \lambda_i/k < 16\frac{\log(2k/\delta)}{k}=a$,
\begin{equation*}
  |\bar z_i - \lambda_i/k| \leq 14\frac{\log(2k/\delta)}{k} + 6\frac{\log(2k/\delta)}{k} \leq 20\frac{\log(2k/\delta)}{k}  \leq 2a. 
\end{equation*}
So
$
\bar z_i \leq 3 a,
$
and then,
$$
 ((\lambda_i/k) \vee a)/4 = a/4 \leq \bar z_i \vee a \leq 3a = 3((\lambda_i/k) \vee a).
$$

\noindent \textbf{Conclusion for the large $\lambda_i$'s.}

Let us first reformulate Equation~\eqref{eq:poissDev2} using the definition of $a$. We have with probability larger than \co{$1-\delta$} that for all $i$,
\begin{align*}
	|\bar z_i - \lambda_i/k| &\leq 2\sqrt{\frac{3(\lambda_i/k)\log(32 \log(2k/\delta)/[((\lambda_i/k) \vee a) \delta])}{k}}\\
	&\quad\quad\quad\quad\quad\quad\quad\quad\quad\quad\quad\quad\quad\quad\quad\quad\quad + 6\frac{\log(32 \log(2k/\delta)/[((\lambda_i/k) \vee a) \delta])}{k}.
\end{align*}
So by application of Equation~\eqref{crude}, we get that with probability larger than \co{$1-\delta$} and for all $i$ we have 
\begin{align*}
  |\bar z_i - \lambda_i/k| &\leq 2\sqrt{\frac{12(\bar z_i \lor a)\log(96\log(2k/\delta)(\bar z_i^{-1} \land a^{-1})/\delta)}{k}} \\
  &\quad\quad\quad\quad\quad\quad\quad\quad\quad\quad\quad\quad\quad\quad\quad\quad\quad+ 2\frac{\log(96\log(2k/\delta)(\bar z_i^{-1} \land a^{-1})/\delta)}{k},
\end{align*}
and 
\begin{align*}
&2\sqrt{\frac{12(\bar z_i \lor a)\log(96\log(2k/\delta)(\bar z_i^{-1} \land a^{-1})/\delta)}{k}}+  2\frac{\log(96\log(2k/\delta)(\bar z_i^{-1} \land a^{-1})/\delta)}{k}\\
&\quad \quad \quad \quad \quad \quad \quad \quad \quad 
 \leq 2\sqrt{\frac{36((\lambda_i/k) \lor a) \log\Big(384\log(2k/\delta)\big((\lambda_i/k)^{-1} \land a^{-1}\big)/\delta\Big)}{k}} \\
&\quad\quad \quad \quad \quad \quad \quad \quad \quad \quad + 2\frac{\log\Big(384\log(2k/\delta)\big((\lambda_i/k)^{-1} \land a^{-1}\big)/\delta\Big)}{k}.
\end{align*}
\end{proof}

\subsection{Proof of Proposition~\ref{th:test1}}

Proposition~\ref{th:test1} provides guarantees on the test $\varphi_{2/3}$. In order to prove it, let us first consider the associated statistic $T_{2/3}$.

\noindent \textbf{Expression of the test statistic.}

We have
$$T_{2/3} = \sum_{i \leq d} \hat{q}_i^{-2/3} (X^{(1)}_i-Y^{(1)}_i) (X^{(2)}_i-Y^{(2)}_i).$$
So taking the expectation highlights different terms associated with different independent sub-samples.
$$\mathbb E T_{2/3} = \sum_{i \leq d} \mathbb E^{(3)}(\hat{q}_i^{-2/3})  \mathbb E^{(1)}(X^{(1)}_i-Y^{(1)}_i)  \mathbb E^{(2)}(X^{(2)}_i-Y^{(2)}_i)  ,$$
and
\begin{align*}
\mathbb V T_{2/3} &= \sum_{i \leq d} \mathbb V \Big[\hat{q}_i^{-2/3}  (X^{(1)}_i-Y^{(1)}_i)  (X^{(2)}_i-Y^{(2)}_i) \Big]\\
&\leq \sum_{i \leq d} \mathbb E^{(3)}(\hat{q}_i^{-4/3})  \mathbb E^{(1)}[(X^{(1)}_i-Y^{(1)}_i)^2]  \mathbb E^{(2)}[(X^{(2)}_i-Y^{(2)}_i)^2] .
\end{align*}

\noindent \textbf{Terms that depend on the first and second sub-samples.} We have
$$
\mathbb E^{(1)}(X^{(1)}_i - Y^{(1)}_i)\mathbb E^{(2)}(X^{(2)}_i - Y^{(2)}_i) = k^2 \Delta_i^2.
$$
and
\begin{align*}
\mathbb E^{(1)}[(X^{(1)}_i - Y^{(1)}_i)^2]\mathbb E^{(2)}[(X^{(2)}_i - Y^{(2)}_i)^2]
&=\left[ \mathbb E^{(1)}[(X^{(1)}_i - Y^{(1)}_i)^2]\right]^2\\
&=\Big[ \mathbb E^{(1)}[(X^{(1)}_i - Y^{(1)}_i - k \Delta_i)^2] + k^2\Delta_i^2\Big]^2 \\
&= [k(p_i+q_i) +  k^2\Delta_i^2]^2.
\end{align*}

\noindent \textbf{Terms that depend on the third sub-sample.} Now, the following lemma will help us control the terms associated with $\hat q$.
\begin{lemma}\label{lem:estim}
Assume that $Z\sim \mathcal P(\lambda)$. Then for $r \in \{2/3,4/3\}$
$$\frac{1}{2} \Big(\frac{1}{(e^2\lambda) \lor 1}\Big)^{r} \leq \mathbb E [(Z\lor 1)^{-r}] \leq 6\Big(\frac{e^2}{\lambda \lor 1}\Big)^r .$$
\end{lemma}
\noindent
The proof of the lemma is at the end of the section.
By direct application of Lemma~\ref{lem:estim}, we have with probability larger than $1-6\exp(-k/18)$ with respect to $(\bar k_m^{(j)})_{m \leq 2, j \leq 3}$,
$$\mathbb E^{(3)}\left[\hat{q}_i^{-2/3}\right] \geq \frac{k^{2/3}}{2e^2} \Big(\frac{1}{(kq_i) \lor 1}\Big)^{2/3},$$
and 
$$\mathbb E^{(3)}\left[\hat{q}_i^{-4/3}\right] \leq 6e^4k^{4/3}\Big(\frac{1}{(kq_i) \lor 1}\Big)^{4/3}.$$

\noindent \textbf{Bound on the expectation and variance for $T_{2/3}$.} We obtain with probability larger than $1-6\exp(-k/18)$ with respect to $(\bar k_m^{(j)})_{m \leq 2, j \leq 3}$
\begin{align}
\label{eq:ET1}
\mathbb E T_{2/3} &\geq \sum_{i \leq d} \frac{k^{2/3}}{2e^2} \Big(\frac{1}{(kq_i)\lor 1}\Big)^{2/3}   \Big[k^2 \Delta_i^2\Big]
= \frac{k^{2}}{2e^2} \Big\|\Delta^2 \Big(\frac{1}{q \lor k^{-1}}\Big)^{2/3}\Big\|_1.
\end{align}
and
\begin{align*}
\mathbb V (T_{2/3}) 
&\leq \sum_{i \leq d} 6e^4k^{4/3}\Big(\frac{1}{(kq_i) \lor 1}\Big)^{4/3}   [k(p_i+q_i) +  k^2\Delta_i^2]^2\\
&\leq 12 e^4k^{4/3}\Bigg[\Big\|\Big(\frac{1}{(kq) \lor 1}\Big)^{4/3}  k^2 (p+q)^2\Big\|_1 + k^4\Big\|\Big(\frac{1}{(kq) \lor 1}\Big)^{4/3}   \Delta^4\Big\|_1\Bigg]\\
&\leq 100 e^4k^{4/3} k^2\Bigg[\Big\|\Big(\frac{1}{(kq) \lor 1}\Big)^{4/3}   q^2\Big\|_1 + \Big\|\Big(\frac{1}{(kq) \lor 1}\Big)^{4/3}   \Delta^2\Big\|_1 \\
&\quad\quad\quad\quad\quad\quad\quad
+ k^2 \Big\|\Big(\frac{1}{(kq) \lor 1}\Big)^{4/3}    \Delta^4\Big\|_1\Bigg].
\end{align*}
This implies with probability larger than $1-6\exp(-k/18)$ with respect to $(\bar k_m^{(j)})_{m \leq 2, j \leq 3}$,
\begin{align}
\label{eq:VT1}
\begin{split}
\sqrt{\mathbb V (T_{2/3})}
&\leq 10 e^2k\Bigg[ \sqrt{\Big\|\Big(\frac{1}{q \lor k^{-1}}\Big)^{4/3}   q^2\Big\|_1} + \sqrt{\Big\|\Big(\frac{1}{q \lor k^{-1}}\Big)^{4/3}   \Delta^2\Big\|_1} \\
&\quad\quad\quad\quad\quad+ k\sqrt{\Big\|\Big(\frac{1}{q \lor k^{-1}}\Big)^{4/3}    \Delta^4\Big\|_1}\Bigg].
\end{split}
\end{align}

\noindent \textbf{Analysis of $T_{2/3}$ under $H^{(\text{Clo})}_0(\pi)$ and $H^{(\text{Clo})}_1(\pi,\rho)$.} Let us inspect the behaviour of statistic $T_{2/3}$ under the null and the alternative hypotheses. We aim at showing that a test based on $T_{2/3}$ will have different outcomes under $H^{(\text{Clo})}_0$ and $H^{(\text{Clo})}_1$ with large probability.

\noindent
\textit{Under $H^{(\emph{Clo})}_0(\pi)$.} With probability larger than $1-6\exp(-k/18)$ with respect to $(\bar k_m^{(j)})_{m \leq 2, j \leq 3}$, we have
$$\mathbb E T_{2/3} = 0,~~~\text{and}~~~\sqrt{\mathbb V (T_{2/3})} \leq 20 e^2 k \sqrt{\Big\|\Big(\frac{1}{q \lor k^{-1}}\Big)^{4/3}   q^2\Big\|_1},$$
and so by Chebyshev's inequality with probability larger than $1-\alpha-6\exp(-k/18)$
$$T_{2/3} \leq  \alpha^{-1/2}20 e^2 k\sqrt{\Big\|\Big(\frac{1}{q \lor k^{-1}}\Big)^{4/3}   q^2\Big\|_1}.$$

\noindent
\textit{Under $H^{(\emph{Clo})}_1(\pi,\rho)$.}
We assume that for a large $C>0$
\begin{align*}
\Big\| \Delta (\mathbf 1\{i \leq J\})_i \Big\|_1^2  &= \Big\| \Delta \mathbf 1\{kq \geq 1\} \Big\|_1^2 \\
&\geq C \pa{\frac{\Big\|\Big(\frac{1}{q \lor k^{-1}}\Big)^{4/3}   q^2\Big\|_1^{3/2}}{k}\lor\frac{\Big\|\Big(\frac{1}{q \lor k^{-1}}\Big)^{4/3}   q^2\Big\|_1}{k}},
\end{align*}
which implies by Cauchy-Schwarz inequality,
$$\Big\| \frac{\Delta^2}{q^{2/3}} \mathbf 1\{kq \geq 1\}  \Big\|_1 \geq C \pa{ \frac{\sqrt{\Big\|\Big(\frac{1}{q \lor k^{-1}}\Big)^{4/3}   q^2\Big\|_1}}{k} \lor\frac{1}{k}},$$
and in particular that
\begin{align}
\label{eq:ET1geqC}
  \left\| \frac{\Delta^2}{(q\vee k^{-1})^{2/3}} \right\|_1 \geq C/k.  
\end{align}
Moreover if the pre-test does not reject the null, we have with probability larger than $1-2\delta - 2k^{-1}-6\exp(-k/18)$, that there exists $0<\tilde c <+\infty$ universal constant and $0<\tilde c_\delta<+\infty$ that depends only on $\delta$  such that for any $i$
$$
\frac{\Delta_i^2}{(q_i \lor k^{-1})^{2/3}} \leq \frac{\tilde c^2}{k} \frac{(q_i \lor (\log(k/\delta)/k))\log(k/\delta)}{(q_i \lor k^{-1})^{2/3}} \leq  \frac{ \tilde c_\delta^2}{k},
$$
So with probability larger than $1-2\delta - 2k^{-1}-6\exp(-k/18)$,
$$
\left\| \frac{\Delta^4 }{(q \vee 1/k)^ {4/3}} \right\|_1 \leq \frac{\tilde c_\delta^2}{k}  \left\| \frac{\Delta^2}{(q \vee 1/k)^{2/3}} \right\|_1,
$$
i.e., by Equation~\eqref{eq:ET1geqC},

\begin{equation}
\label{eq:T1H1arg2}
\sqrt{\left\| \frac{\Delta^4 }{(q \vee 1/k)^ {4/3}} \right\|_1} \leq \sqrt{\frac{\tilde c_\delta^2}{k} \left\| \frac{\Delta^2}{(q \vee 1/k)^{2/3}} \right\|_1} \leq \sqrt{\frac{\tilde c_\delta^2}{C}} \left\| \frac{\Delta^2}{(q \vee 1/k)^{2/3}} \right\|_1.    
\end{equation}

\noindent
We have from Equation~\eqref{eq:ET1}:
\begin{align*}
2e^2 \mathbb E T_{2/3}/k^2 \geq \Big\|\Delta^2 \Big(\frac{1}{q \lor k^{-1}}\Big)^{2/3}\Big\|_1.
\end{align*}

\noindent
And from Equation~\eqref{eq:VT1},
\begin{align*}
\sqrt{\mathbb V (T_{2/3})}/k^2
\leq 10 e^2 \Bigg[ &\frac{\sqrt{\Big\|\Big(\frac{1}{q \lor k^{-1}}\Big)^{4/3}   q^2\Big\|_1}}{k} +  \frac{\sqrt{\Big\|\Big(\frac{1}{q \lor k^{-1}}\Big)^{4/3}   \Delta^2\Big\|_1}}{k}   \\
&+ \sqrt{\Big\|\Big(\frac{1}{q \lor k^{-1}}\Big)^{4/3}    \Delta^4\Big\|_1}\Bigg].
\end{align*}

\noindent
Let us compare the terms involved in the upper bound on $\sqrt{\mathbb V (T_{2/3})}/k^2$ with the lower bound on $\mathbb E T_{2/3}/k^2$.

For the first term, we have
%
by Equation~\eqref{eq:ET1geqC}:
$$
10 e^2  \frac{\sqrt{\Big\|\Big(\frac{1}{q \lor k^{-1}}\Big)^{4/3}   q^2\Big\|_1}}{k} \leq \frac{20e^4}{C} \E T_{2/3}/k^2.
$$

\noindent
We have for the second term:

$$
10e^2 \frac{\sqrt{\Big\|\Big(\frac{1}{q \lor k^{-1}}\Big)^{4/3}   \Delta^2\Big\|_1}}{k} \leq 20e^4 k^{-2/3} \sqrt{\Big\|\Big(\frac{1}{q \lor k^{-1}}\Big)^{2/3} \Delta^2\Big\|_1}.
$$

\noindent
Since $a^2+b^2 \geq 2ab$ for any $a,b$,
$$
10e^2 \frac{\sqrt{\Big\|\Big(\frac{1}{q \lor k^{-1}}\Big)^{4/3}   \Delta^2\Big\|_1}}{k}  \leq 10e^4 (1/k + k^{-1/3} \E T_{2/3}/k^2).
$$

\noindent
which, from Equation~\eqref{eq:ET1geqC}, yields
$$
10e^2 \frac{\sqrt{\Big\|\Big(\frac{1}{q \lor k^{-1}}\Big)^{4/3}   \Delta^2\Big\|_1}}{k}  \leq 10e^4 (1/C + k^{-1/3} )\E T_{2/3}/k^2.
$$

\noindent
Then for the third term, we have shown in Equation~\eqref{eq:T1H1arg2} that \co{with probability larger than $1-2\delta - 2k^{-1}-6\exp(-k/18)$,}

$$
10e^2 \sqrt{\Big\|\Big(\frac{1}{q \lor k^{-1}}\Big)^{4/3}    \Delta^4\Big\|_1} \leq 20e^4 \sqrt{\frac{\tilde c_\delta^2}{C}} \E T_{2/3}/k^2.
$$

\noindent
And so we have by Chebyshev's inequality, with probability larger than \co{$1-2\delta - 2k^{-1}-\alpha-6\exp(-k/18)$}:
$$|T_{2/3} - \E T_{2/3} | \leq \frac{20e^4}{\sqrt{\alpha}}(1/(2C) + k^{-1/3}/2 +  \sqrt{\tilde c_\delta^2/C} + 1/C) \E T_{2/3}.$$

\noindent
\co{Now, if $k \geq \pa{\frac{80e^4}{\sqrt{\alpha}}}^3$ and $C \geq \frac{40e^4}{\sqrt{\alpha}} \pa{\frac{20e^4 \tilde c_\delta^2}{\sqrt{\alpha}} \vee 1}$:}

$$|T_{2/3} - \E T_{2/3} | \leq \E T_{2/3} / 2.$$

\noindent
Finally, if $k \geq \pa{\frac{80e^4}{\sqrt{\alpha}}}^3$ and $C \geq \frac{40e^4}{\sqrt{\alpha}} \pa{\frac{20e^4 \tilde c_\delta^2}{\sqrt{\alpha}} \vee 1}$, with probability greater than \co{$1-2\delta - 2k^{-1}-\alpha-6\exp(-k/18)$}:
\begin{align}
\label{eq:T1H1}
T_{2/3} &\geq \E T_{2/3}/2 \nonumber \\
& \geq  \frac{C k}{2} \Bigg(\sqrt{\left\|\Big(\frac{1}{q \lor k^{-1}}\Big)^{4/3}   q^2 \right\|_1}+1\Bigg),    
\end{align}

\noindent
where the last inequality comes from Equations~\eqref{eq:ET1} and \eqref{eq:ET1geqC}.

\noindent \textbf{Analysis of $\hat{t}_{2/3}$.}
Test $\varphi_{2/3}$ compares statistic $T_{2/3}$ with threshold $\hat{t}_{2/3}$, which is empirical. So let us study the variations of $\hat{t}_{2/3}$. Applying Corollary~\ref{cor:thresh1} below gives guarantees on the empirical threshold $\hat{t}_{2/3}$. These can be used in conjunction with the guarantees on the statistic $T_{2/3}$ in order to conclude the proof of Proposition~\ref{th:test1}.

\begin{theorem}
\label{th:thresh1}
Let $C_{2/3} = \sqrt{2\delta^{-1} e^{8/3} +1} + \sqrt{(2^{1/3} +e)}.$

With probability greater than $1-\beta$:
\begin{align*}
&(e^{-2/3}/2 + 1)\Big\|\Big(\frac{1}{q \lor k^{-1}}\Big)^{4/3}   q^2 \Big\|_1\\
&\quad\quad\quad\quad\quad\quad\quad\quad
\leq k^{-2/3} \|(Y^{(1)})^{2/3}\|_1 +C_{2/3}/\sqrt \beta \\
&\quad\quad\quad\quad\quad\quad\quad\quad\quad\quad\quad\quad\quad\quad\quad\quad\quad\quad
\leq \Big\|\Big(\frac{1}{q \lor k^{-1}}\Big)^{4/3}   q^2 \Big\|_1+ 2C_{2/3}/\sqrt{\beta} + 2\delta^{-1}.
\end{align*}
\end{theorem}

\noindent
The proof of this theorem is in Section~\ref{sec:thresh1}. And the following corollary is  obtained immediately from the theorem.

\begin{corollary}
\label{cor:thresh1}
We define

$$\hat t_{2/3} = \alpha^{-1/2} 20 e k \sqrt{k^{-2/3} \|(Y^{(1)})^{2/3}\|_1 +C_{2/3}/\sqrt \alpha}.$$

\noindent
Then if $C \geq (8 e^6/20e^{-1} \sqrt{\alpha}) \vee (8^2 e^{10} c^2_\delta \alpha / 100) \vee (\alpha^{-1/2} 40 e^7 \sqrt{2\delta^{-1}(C_{2/3} + 1)})$, we have with probability greater than $1-\alpha-6\exp(-k/18)$:
\begin{align*}
\alpha^{-1/2}20 ek\sqrt{\Big\|\Big(\frac{1}{q \lor k^{-1}}\Big)^{4/3}   q^2\Big\|_1} \leq \hat t_{2/3} \leq &e^{-6}\frac{C}{2} k \left(\sqrt{\Big\|\Big(\frac{1}{q \lor k^{-1}}\Big)^{4/3}   q^2 \Big\|_1} + 1\right).
\end{align*}
\end{corollary}
Let us now sum up the results leading to Proposition~\ref{th:test1}. 
Under $H^{(\text{Clo})}_0(\pi)$, with probability larger than \co{$1-\delta/2 - 2 \delta - 2 k^{-1}-6\exp(-k/18)$},
 $$T_{2/3} \leq  (\delta/2)^{-1/2}20 ek\sqrt{\Big\|\Big(\frac{1}{q \lor k^{-1}}\Big)^{4/3}   q^2\Big\|_1}.$$
And so if $C \geq \cro{\frac{40e^4}{\sqrt{\alpha}} \pa{\frac{20e^4 \tilde c_\delta^2}{\sqrt{\alpha}} \vee 1}} \vee \pa{\alpha^{-1/2} 40 e^7\sqrt{2C_{2/3}( \delta^{-1/2}+\alpha^{-1/2})}}$, we have with probability greater than $1-\delta/2-6\exp(-k/18)$:
\begin{align*}
(\delta/2)^{-1/2}20 ek\sqrt{\Big\|\Big(\frac{1}{q \lor k^{-1}}\Big)^{4/3}   q^2\Big\|_1} \leq \hat t_{2/3}.
\end{align*} 
So, under $H^{(\text{Clo})}_0(\pi)$, with probability larger than \co{$1 - 3 \delta - 2 k^{-1}-6\exp(-k/18)$},

$$
T_{2/3} \leq \hat t_{2/3}.
$$
Under $H^{(\text{Clo})}_1(\pi,\rho)$, if $C \geq \cro{\frac{40e^4}{\sqrt{\delta/2}} \pa{\frac{20e^4 \tilde c_\delta^2}{\sqrt{\delta/2}} \vee 1}} \vee \pa{(\delta/2)^{-1/2} 80 e^7\sqrt{2C_{2/3}\delta^{-1/2}}}$ and $k \geq \pa{\frac{80e^4}{\sqrt{\delta/2}}}^3$, we have with probability larger than \co{$1-\delta/2 - 2 \delta - 2 k^{-1}-6\exp(-k/18)$},
$$
e^{-6} \frac{Ck}{2}  \Bigg(\sqrt{\left\|\Big(\frac{1}{q \lor k^{-1}}\Big)^{4/3}   q^2\right\|_1}+1\Bigg)\leq T_{2/3}.
$$
And if $C \geq \cro{\frac{40e^4}{\sqrt{\delta/2}} \pa{\frac{20e^4 \tilde c_\delta^2}{\sqrt{\delta/2}} \vee 1}} \vee \pa{(\delta/2)^{-1/2} 80 e^7\sqrt{2C_{2/3}\delta^{-1/2}}}$ and $k \geq \pa{\frac{80e^4}{\sqrt{\delta/2}}}^3$, we have with probability larger than $1-\delta/2-6\exp(-k/18)$,
$$
\hat t_{2/3} \leq e^{-6}\frac{Ck}{2}  \Bigg(\sqrt{\left\|\Big(\frac{1}{q \lor k^{-1}}\Big)^{4/3}   q^2\right\|_{2/3}}+1\Bigg).
$$
So, under $H^{(\text{Clo})}_1(\pi,\rho)$, with probability larger than \co{$1-3\delta - 2 k^{-1}-6\exp(-k/18)$},
$$
\hat t_{2/3} \leq T_{2/3}.
$$

\begin{proof}[Proof of Lemma~\ref{lem:estim}]
We have by definition of the Poisson distribution that
\begin{align*}
\mathbb E [(Z\lor 1)^{-r}] &= \exp(-\lambda) + \exp(-\lambda)\sum_{i \geq 1} \frac{\lambda^i}{i!} i^{-r} \\
&= \exp(-\lambda) + \exp(-\lambda)\sum_{1 \leq i \leq \lambda/e^2} \frac{\lambda^i}{i!} i^{-r} + \exp(-\lambda)\sum_{1\vee(\lambda/e^2) < i} \frac{\lambda^i}{i!} i^{-r}.
\end{align*}
And so we have
\begin{align*}
\mathbb E [(Z\lor 1)^{-r}] &\leq  \exp(-\lambda) + \exp(-\lambda)\sum_{1 \leq i \leq \lambda/e^2} \frac{\lambda^i}{i!} i^{-r} + \Big(\frac{e^2}{\lambda} \land 1\Big)^r\\
&\leq  \exp(-\lambda) + \exp(-\lambda)\sum_{1 \leq i \leq \lambda/e^2} \frac{\lambda^i e^i}{i^i} + \Big(\frac{e^2}{\lambda} \land 1\Big)^r,
\end{align*}
since $i! \geq i^i/e^i$ and $i\geq 1$. Then 
\begin{align*}
\mathbb E [(Z\lor 1)^{-r}]&\leq  \exp(-\lambda) + \exp(-\lambda)\sum_{1 \leq i \leq \lambda/e^2} \exp(i\log(\lambda) + i - i\log(i)) \\
&\quad\quad+ \left(\frac{1}{1\vee(\lambda/e^2)}\right)^r\\
&\leq  \exp(-\lambda) + \lambda \exp\left(-\lambda+\frac{\lambda}{e^2}\log(\lambda) + \frac{\lambda}{e^2} - \frac{\lambda}{e^2}\log\left(\frac{\lambda}{e^2}\right)\right) \\
&\quad\quad+ \left(\frac{1}{1\vee(\lambda/e^2)}\right)^r\\
&\leq  \exp(-\lambda) + \lambda \exp(-\lambda/2) + \left(\frac{1}{1\vee(\lambda/e^2)}\right)^r \\
& \leq   5\exp(-\lambda/4) + \left(\frac{1}{1\vee(\lambda/e^2)}\right)^r\leq 6 \left(\frac{e^2}{1\vee\lambda}\right)^r,
\end{align*}
since $r \in \{2/3,4/3\}$.
Now let us prove the other inequality.
\begin{align*}
\mathbb E [(Z\lor 1)^{-r}] &\geq \exp(-\lambda) + \exp(-\lambda)\sum_{i < e^2 \lambda} \frac{\lambda^i}{i!} i^{-r}\\
& \geq \exp(-\lambda) + (e^2 \lambda)^{-r}\exp(-\lambda)\sum_{i < e^2 \lambda} \frac{\lambda^i}{i!}.
\end{align*}
So, since $i! \geq i^i/e^i$,
\begin{align*}
\mathbb E [(Z\lor 1)^{-r}] &\geq  \exp(-\lambda)  + \Big(\frac{1}{(e^2\lambda) \lor 1}\Big)^{r} \Big[1 - \exp(-\lambda)\sum_{i \geq e^2 \lambda} \frac{\lambda^i e^i}{i^i}\Big]\\
&=    \exp(-\lambda)  \\
&\quad+ \Big(\frac{1}{(e^2\lambda) \lor 1}\Big)^{r} \Big[1 - \exp(-\lambda)\sum_{i \geq e^2 \lambda}\exp(i\log(\lambda) + i - i\log(i) \Big]\\
&\geq   \exp(-\lambda)  \\
&\quad+ \Big(\frac{1}{(e^2\lambda) \lor 1}\Big)^{r} \Big[1 - \exp(-\lambda)\sum_{i \geq e^2 \lambda}\exp(i\log(\lambda) + i - i\log(e^2\lambda)) \Big]\\
&=   \exp(-\lambda)  + \Big(\frac{1}{(e^2\lambda) \lor 1}\Big)^{r} \Big[1 - \exp(-\lambda)\sum_{i \geq e^2 \lambda}\exp(-i) \Big]\\
&\geq   \exp(-\lambda)  + \Big(\frac{1}{(e^2\lambda) \lor 1}\Big)^{r} \Big[1 - 2\exp(-\lambda - e^2\lambda)\Big].
\end{align*}
If $\lambda \geq \frac{\log\pa{4}}{1+e^2}$, then
$$
\mathbb E [(Z\lor 1)^{-r}] \geq  \Big(\frac{1}{(e^2\lambda) \lor 1}\Big)^{r} \Big[1 - 2\exp(-\lambda - e^2\lambda)\Big] \geq \frac{1}{2} \Big(\frac{1}{(e^2\lambda) \lor 1}\Big)^{r}.
$$
If $\lambda < \frac{\log\pa{4}}{1+e^2}$, then 
$$\mathbb E [(Z\lor 1)^{-r}] \geq \exp(-\lambda) \geq \exp\left(- \frac{\log(4)}{1+e^2}\right) \geq 1/2 \geq \frac{1}{2} \Big(\frac{1}{(e^2\lambda) \lor 1}\Big)^{r}.$$ 
So in any case,

\begin{align*}
\mathbb E [(Z\lor 1)^{-r}] &\geq \frac{1}{2} \Big(\frac{1}{(e^2\lambda) \lor 1}\Big)^{r}.
\end{align*}

\end{proof}

\subsection{Proof of Proposition~\ref{th:test2}}

Proposition~\ref{th:test2} provides guarantees on the test $\varphi_2$. The structure of its proof will be identical to that of Proposition~\ref{th:test1}.
We first study $T_2.$

\noindent \textbf{Expression of the test statistic.}

We have
$$T_2 = \sum_{i \leq d}   (X^{(1)}_i-Y^{(1)}_i) (X^{(2)}_i-Y^{(2)}_i) \mathbf 1\{Y^{(3)}_i = 0\}.$$
And so
$$\mathbb E T_2 = \sum_{i \leq d}   \mathbb E^{(1)}(X^{(1)}_i-Y^{(1)}_i)  \mathbb E^{(2)}(X^{(2)}_i-Y^{(2)}_i)  \mathbb E^{(3)}\mathbf 1\{Y^{(3)}_i = 0\},$$
and
\begin{align*}
\mathbb V T_2 &\leq \sum_{i \leq d} \mathbb V \Big[  (X^{(1)}_i-Y^{(1)}_i)  (X^{(2)}_i-Y^{(2)}_i)  \mathbf 1\{Y^{(3)}_i = 0\}\Big]\\
&\leq \sum_{i \leq d}  \E^{(1)}[(X^{(1)}_i - Y^{(1)}_i)^2]\mathbb E^{(2)}[(X^{(2)}_i - Y^{(2)}_i)^2] \mathbb E^{(3)}\mathbf 1\{Y^{(3)}_i = 0\}.
\end{align*}
We will bound every term separately.

\noindent \textbf{Terms that depend on the first and second sub-sample.} We have with probability larger than $1-6\exp(-k/18)$ with respect to $(\bar k_m^{(j)})_{m \leq 2, j \leq 3}$,
$$
\mathbb E^{(1)}(X^{(1)}_i - Y^{(1)}_i)\mathbb E^{(2)}(X^{(2)}_i - Y^{(2)}_i) = k^2 \Delta_i^2.
$$
and
\begin{align*}
\mathbb E^{(1)}[(X^{(1)}_i - Y^{(1)}_i)^2]\mathbb E^{(2)}[(X^{(2)}_i - Y^{(2)}_i)^2]
&=\Big[ \mathbb E^{(1)}[(X^{(1)}_i - Y^{(1)}_i)^2]\Big]^2\\
&=\Big[ \mathbb E^{(1)}[(X^{(1)}_i - Y^{(1)}_i - k \Delta_i)^2] + k^2\Delta_i^2\Big]^2 \\
&= [k(p_i+q_i) +  k^2\Delta_i^2]^2.
\end{align*}

\noindent \textbf{Terms that depend on the third sub-sample.} We define
$$R_i :=\mathbb E^{(3)}\mathbf 1\{Y^{(3)}_i = 0\},\quad\quad\quad\quad
\text{and so}\quad\quad\quad\quad
R_i = \exp(-kq_i).$$

\noindent \textbf{Bound on the expectation and variance for $T_2$.} 
We have with probability larger than $1-6\exp(-k/18)$ with respect to $(\bar k_m^{(j)})_{m \leq 2, j \leq 3}$,
\begin{align}
\label{eq:ET2}
\mathbb ET_2 = \sum_{i \leq d}   \Big[k^2 \Delta_i^2 R_i\Big] = k^2\|\Delta^2R\|_1= k^2\|\Delta^2\exp(-kq)\|_1.
\end{align}
And
\begin{align*}
\mathbb V  T_2 
&\leq \sum_{i \leq d } \Big[k(p_i+q_i) + k^2 \Delta_i^2\Big]^2 R_i\\
&\leq 4\sum_{i \leq d }  \Big[k^2 q_i^2 + k^2 \Delta_i^2 + k^4 \Delta_i^4\Big] R_i\\
&\leq 4 \Big[k^2 \|q^2R\|_1 + k^2 \|\Delta^2 R\|_1 + k^4 \|R\Delta^4\|_1\Big],
\end{align*}
and so with probability larger than $1-6\exp(-k/18)$ with respect to $(\bar k_m^{(j)})_{m \leq 2, j \leq 3}$,
\begin{align}
\label{eq:VT2}
\sqrt{\mathbb V T_2} 
&\leq 2k\Big[\sqrt{\|q^2R\|_1} + \sqrt{\|\Delta^2R\|_1} + k\sqrt{ \|R\Delta^4\|_1}\Big] \nonumber\\
&\leq 2\Big[\sqrt{k^2\|q^2\exp(-kq)\|_1} + \sqrt{k^2\|\Delta^2\exp(-kq)\|_1} + k^2\sqrt{ \|\Delta^4 \exp(-kq)\|_1}\Big].
\end{align}

\noindent \textbf{Analysis of $T_2$ under $H^{(\text{Clo})}_0(\pi)$ and $H^{(\text{Clo})}_1(\pi,\rho)$.} Let us inspect the behaviour of statistic $T_2$ under both hypotheses.

\noindent
\textit{Under $H^{(\text{Clo})}_0(\pi)$.} We have then with probability larger than $1-6\exp(-k/18)$ with respect to $(\bar k_m^{(j)})_{m \leq 2, j \leq 3}$,
\begin{align*}
\mathbb E T_2 =0,
\end{align*}
\begin{align*}
\sqrt{\mathbb V T_2} 
&\leq 2 k\sqrt{\|q^2\exp(-kq)\|_1}.
\end{align*}
And so by Chebyshev's inequality, with probability larger than $1-\alpha-6\exp(-k/18)$
$$T_2 \leq 2 \alpha^{-1/2} \sqrt{\|(kq)^2\exp(-kq)\|_1}.$$

\noindent
\textit{Under $H^{(\text{Clo})}_1(\pi,\rho)$.} Assume that for $C>0$ large we have
$$\|\Delta (\mathbf 1\{I \geq i \geq J\})_i\|_1^2 \geq Ce^2\frac{I-J}{k}\Big[\frac{\log^2(k)}{k} \lor \Big(\sqrt{\|q^2\exp(-kq)\|_1}\Big)\Big].$$
By Cauchy-Schwarz inequality and since for any $I \geq i \geq J$, $kq_i \leq 1$, this implies
\begin{equation}
  \label{eq:T2geqC}
  k^2\|\Delta^2\exp(-kq)\|_1 \geq C\Big[\log^2(k) \lor \Big(k\sqrt{\|q^2\exp(-kq)\|_1}\Big)\Big].  
\end{equation}
If the pre-test accepts the null, then with probability larger than \co{$1 - 2\delta - 2k^{-1}-6\exp(-k/18)$} on the third sub-sample only, there exists $\tilde c_\delta$ that depends only on $\delta$ such that for all $i\leq d$ we have $|\Delta_i| \leq \tilde c_\delta\Big[\sqrt{\frac{q_i\log(k)}{k}} \lor \frac{\log(k)}{k}\Big]$ and so
\co{
\begin{align}
\label{eq:T2H1arg1}
&k^2\sqrt{ \|\Delta^4 \exp(-kq)\|_1} \nonumber \\
&\leq \tilde c_\delta k^2\sqrt{ 
\begin{aligned}
\| &\exp(-kq) \Big[\frac{q^2\log(k)^2}{k^2}\mathbf 1\{kq\geq 2\log(k)\} \\
&+ \frac{\Delta^2\log(k)^2}{k^2} \mathbf 1\{kq\leq 2\log(k)\}\Big]\|_1
\end{aligned}
}\nonumber\\
&\leq \tilde c_\delta k^2\Big[\sqrt{ \| \exp(-kq) \frac{q^2\log(k)^2}{k^2}\mathbf 1\{kq\geq 2\log(k)\}\|_1}\nonumber\\ 
&+ \sqrt{\|\exp(-kq)  \frac{\Delta^2\log(k)^2}{k^2} \mathbf 1\{kq\leq 2\log(k)\}\|_1}\Big]\nonumber\\
&\leq \tilde c_\delta k^2 \Big[\sqrt{ \| \exp(-kq) \frac{q^2\log(k)^2}{k^2} \mathbf 1\{kq\geq 2\log(k)\}\|_1} + \frac{\log(k)}{k}\sqrt{\|\exp(-kq)  \Delta^2 \|_1}\Big]\nonumber\\
&\leq \tilde c_\delta k^2 \Big[\sqrt{ k^{-4} \log(k)^2 \| q^2 \mathbf 1\{kq\geq 2\log(k)\}\|_1} + \frac{\log(k)}{k}\sqrt{\|\exp(-kq) \Delta^2 \|_1}\Big]\nonumber\\
&\leq \tilde c_\delta \Big[\log(k) + k\log(k)\sqrt{\|\exp(-kq)  \Delta^2 \|_1}\Big]\nonumber\\
&\leq \tilde c_\delta \log(k) \Big[ 1 + k\sqrt{\|\exp(-kq)  \Delta^2 \|_1}\Big],
\end{align}
since $\sum_i q_i = 1$.}

\noindent
We have from Equation~\eqref{eq:ET2}:
\begin{align*}
\mathbb E T_2/k^2 = \|\Delta^2\exp(-kq)\|_1.
\end{align*}

\noindent
And from Equation~\eqref{eq:VT2},
\begin{align*}
\sqrt{\mathbb V T_2}/k^2
&\leq 2\Big[\sqrt{\|q^2\exp(-kq)\|_1}/k + \sqrt{\|\Delta^2\exp(-kq)\|_1}/k + \sqrt{ \|\Delta^4 \exp(-kq)\|_1}\Big].
\end{align*}

\noindent
Let us compare the terms of $\sqrt{\mathbb V T_2}/k^2$ with $\mathbb E T_2/k^2$.

For the first term, we use Equation~\eqref{eq:T2geqC}, and we get:

$$
\sqrt{\|q^2\exp(-kq)\|_1}/k \leq \|\Delta^2\exp(-kq)\|_1 / C = \frac{1}{C} \mathbb E T_2/k^2 . 
$$

\noindent
For the second term, we have:
$$
\sqrt{\|\Delta^2\exp(-kq)\|_1}/k \leq \sqrt{\|\Delta^2\exp(-kq)\|_1}\log(k)/k.
$$

\noindent
So using Equation~\eqref{eq:T2geqC}, we have:
$$
\sqrt{\|\Delta^2\exp(-kq)\|_1}/k \leq \frac{1}{\sqrt{C}} \E T_2 /k^2.
$$

\noindent
For the third term, using Equation~\eqref{eq:T2H1arg1}, we have \co{with probability larger than $1 - 2\delta - 2k^{-1} -6\exp(-k/18)$}:
\co{
$$
\sqrt{ \|\Delta^4 \exp(-kq)\|_1} \leq  \tilde c_\delta \Big[ k^{-2}\log(k) + \frac{\log(k)}{k}\sqrt{\|\exp(-kq)  \Delta^2 \|_1}\Big].
$$
}


\noindent
So by Equation~\eqref{eq:T2geqC}, we have \co{with probability larger than $1 - 2\delta - 2k^{-1}-6\exp(-k/18)$}:
\co{
\begin{align*}
\sqrt{ \|\Delta^4 \exp(-kq)\|_1} &\leq \tilde c_\delta \left[ k^{-2}\log^2(k) + \frac{1}{\sqrt{C}}\E T_2 /k^2 \right] \\
&\leq \tilde c_\delta (1/C + 1/\sqrt{C}) \E T_2 /k^2.
\end{align*}
}
\noindent
And so we have by Chebyshev's inequality, with probability larger than \co{$1 -\alpha - 2\delta - 2k^{-1}-6\exp(-k/18)$}:
\co{$$|T_2 - \E T_2 | \leq 2/\sqrt{\alpha} (1/C + 1/\sqrt{C} + \tilde c_\delta (1/C + 1/\sqrt{C})) \E T_2.$$

\noindent
So if $C \geq 1$, with probability larger than \co{$1 -\alpha - 2\delta - 2k^{-1}-6\exp(-k/18)$}:
$$|T_2 - \E T_2 | \leq \frac{4}{\sqrt{C\alpha}} (1 + \tilde c_\delta) \E T_2.$$

\noindent
So if $C \geq \left[8\alpha^{-1/2} (1 + \tilde c_\delta) \right]^2$, we have with probability larger than \co{$1 -\alpha - 2\delta - 2k^{-1}-6\exp(-k/18)$}:

$$
|T_2 - \E T_2 | \leq \E T_2/2.
$$

\noindent
Finally, if $C \geq \left[8\alpha^{-1/2} (1 + \tilde c_\delta) \right]^2$, we have with probability larger than \co{$1 -\alpha - 2\delta - 2k^{-1}-6\exp(-k/18)$}:
}
$$
T_2 \geq \E T_2/2 \geq \frac{C}{2}\Big[\log^2(k) \lor \Big(k\sqrt{\|q^2\exp(-kq)\|_1}\Big)\Big]. 
$$

\noindent \textbf{Analysis of $\hat{t_2}$.}

Test $\varphi_2$ compares statistic $T_2$ with empirical threshold $\hat{t_2}$. So let us study the variations of $\hat{t_2}$. Applying Corollary~\ref{cor:thresh2} below gives guarantees on the empirical threshold $\hat{t}_2$. These can be used in conjunction with the guarantees on the statistic $T_2$ in order to conclude the proof of Proposition~\ref{th:test2}.

\begin{theorem}
\label{th:thresh2}
We have with probability larger than $1-\delta-6\exp(-k/18)$:
$$
|\|Y^{(1)}Y^{(2)} \mathbf 1\{Y^{(3)}=0\}\|_1 - \|(kq)^2 e^{-kq}\|_1|\leq  \frac{1}{\sqrt \delta} (\|(kq)^2 e^{-kq}\|_1 /2 + 1005\log(k)^4).
$$

\end{theorem}

\noindent
The proof of this theorem is in Section~\ref{sec:thresh2}.

\begin{corollary}
\label{cor:thresh2}
We define

$$\hat t_2 = 2\alpha^{-1/2}(1-1/(2\sqrt{\delta}))^{-1/2} \sqrt{\|Y^{(1)}Y^{(2)} \mathbf 1\{Y^{(3)}=0\}\|_1 + \frac{1005}{\sqrt \delta} \log(k)^4}.$$
If $C \geq \left[8\alpha^{-1/2} (1 + \tilde c_\delta \sqrt{2\delta^{-1}}) \right]^2 \vee \frac{8\cdot 45 \alpha^{-1/2}}{(\sqrt{\delta} - 1/2)^{1/2}}$, we have with probability greater than $1-\delta-6\exp(-k/18)$:
$$
2 \alpha^{-1/2} \sqrt{\|(kq)^2\exp(-kq)\|_1} \leq \hat t_2 \leq  \frac{C}{2}\Big[\log^2(k) \lor \Big(k\sqrt{\|q^2\exp(-kq)\|_1}\Big)\Big].
$$
\end{corollary}

\begin{proof}[Proof of Corollary~\ref{cor:thresh2}]

By application of Theorem~\ref{th:thresh2}, we have with probability greater than $1-\delta-6\exp(-k/18)$:

\begin{align*}
2 \alpha^{-1/2} \sqrt{\|(kq)^2\exp(-kq)\|_1} \leq \hat t_2 \leq  2\alpha^{-1/2}\sqrt{
\begin{aligned}
&\frac{2\sqrt{\delta} + 1}{2\sqrt{\delta} - 1} \|(kq)^2 e^{-kq}\|_1 \\
&+ 2010 (\sqrt{\delta} - 1/2)^{-1} \log(k)^4.
\end{aligned}}
\end{align*}
So,
$$
\hat t_2 \leq 2\alpha^{-1/2}\left(\sqrt{\frac{2\sqrt{\delta} + 1}{2\sqrt{\delta} - 1} \|(kq)^2 e^{-kq}\|_1} + \sqrt{2010 (\sqrt{\delta} - 1/2)^{-1} \log(k)^4}\right).
$$
Finally,
$$
\hat t_2 \leq \frac{4\cdot 45 \alpha^{-1/2}}{(\sqrt{\delta} - 1/2)^{1/2}}\left(\sqrt{\|(kq)^2 e^{-kq}\|_1} \vee \log(k)^2\right).
$$

\end{proof}
Let us now sum up the results leading to Proposition~\ref{th:test2}.  
Under $H^{(\text{Clo})}_0(\pi)$, with probability larger than \co{$1-\delta/2 - 2\delta - 2k^{-1}-6\exp(-k/18)$},
$$T_2 \leq 2 (\delta/2)^{-1/2} \sqrt{\|(kq)^2\exp(-kq)\|_1}.$$
And if $C \geq \left[8(\delta/2)^{-1/2}  (1 + \tilde c_\delta \sqrt{2\delta^{-1}}) \right]^2 \vee \frac{8\cdot 45 \alpha^{-1/2}}{(\sqrt{\delta} - 1/2)^{1/2}}$, we have with probability greater than \co{$1-\delta/2-6\exp(-k/18)$}:
\begin{align*}
2 (\delta/2)^{-1/2} \sqrt{\|(kq)^2\exp(-kq)\|_1} \leq \hat t_2
\end{align*} 
So, under $H^{(\text{Clo})}_0(\pi)$, with probability larger than $1- 3\delta - 2k^{-1}-6\exp(-k/18)$,

$$
T_2 \leq \hat t_2.
$$
Under $H^{(\text{Clo})}_1(\pi,\rho)$, if $C \geq \left[8(\delta/2)^{-1/2}  (1 + \tilde c_\delta \sqrt{2\delta^{-1}}) \right]^2$, we have with probability larger than \co{$1-\delta/2- 2\delta - 2k^{-1}-6\exp(-k/18)$}:
$$
T_2 \geq C/2\Big[\log^2(k) \lor \Big(k\sqrt{\|q^2\exp(-kq)\|_1}\Big)\Big]. 
$$
If $C \geq \left[8(\delta/2)^{-1/2}  (1 + \tilde c_\delta \sqrt{2\delta^{-1}}) \right]^2 \vee \frac{8\cdot 45 (\delta/2)^{-1/2}}{(\sqrt{\delta} - 1/2)^{1/2}}$, we have with probability greater than $1-\delta/2-6\exp(-k/18)$:
$$
\hat t_2 \leq \frac{C}{2}\Big[\log^2(k) \lor \Big(k\sqrt{\|q^2\exp(-kq)\|_1}\Big)\Big].
$$
So, under $H^{(\text{Clo})}_1(\pi,\rho)$, with probability larger than \co{$1- 3\delta - 2k^{-1}-6\exp(-k/18)$},
$$
\hat t_2 \leq T_2.
$$

\subsection{Proof of Proposition~\ref{th:test3}}
Proposition~\ref{th:test3} gives guarantees on test $\varphi_1$. This time, the proof will only focus on the variations of $T_1$ since the threshold is not empirical.

\noindent \textbf{Analysis of the moments of $T_1$.}

We have
$$T_1 = \sum_i (X^{(1)} - Y^{(1)}) \mathbf 1\{Y^{(3)}_i = 0\}.$$
So with probability larger than $1-6\exp(-k/18)$ with respect to $(\bar k_m^{(j)})_{m \leq 2, j \leq 3}$
\begin{equation}
\label{eq:ET3}
  \mathbb E T_1 = k\sum_i \Delta_i \exp(-kq_i).  
\end{equation}
And
\begin{align*}
\mathbb V T_1 &\leq \sum_i \mathbb E^{(1)}(X^{(1)} - Y^{(1)})^2 \mathbb E^{(3)}\mathbf 1\{Y^{(3)}_i = 0\}\\
&\leq \sum_i [k(p_i+q_i) + k^2\Delta_i^2]\exp(-kq_i)\\
&\leq 2 k \|q\exp(-kq)\|_1 + \big|k\sum_i \Delta_i \exp(-kq_i)\big| + k^2\|\exp(-kq)\Delta^2\|_1,
\end{align*}
which implies
\begin{align}
\label{eq:VT3}
\sqrt{\mathbb V T_1}
&\leq  \sqrt{2k \|q\exp(-kq)\|_1} + \sqrt{k \left|\sum_i \Delta_i \exp(-kq_i)\right|} + k\sqrt{\|\Delta^2\exp(-kq)\|_1}.
\end{align}

\noindent \textbf{Analysis of $T_1$ under $H^{(\text{Clo})}_0(\pi)$ and $H^{(\text{Clo})}_1(\pi,\rho)$.} Let us inspect the behaviour of statistic $T_1$ under both hypotheses.

\noindent
\textit{Under $H^{(\emph{Clo})}_0(\pi)$.} With probability larger than $1-6\exp(-k/18)$ with respect to $(\bar k_m^{(j)})_{m \leq 2, j \leq 3}$, we have $\mathbb E T_1 = 0$ and $\sqrt{\mathbb V T_1}\leq  \sqrt{2k \|q\exp(-kq)\|_1}$, and so by Chebyshev's inequality with probability larger than $1-\alpha-6\exp(-k/18)$
$$T_1 \leq \sqrt{\frac{2k \|q\exp(-kq)\|_1}{\alpha}} \leq \sqrt{2 k/\alpha}.$$
Note that the result of Proposition~\ref{th:test3} is based on the reunion of two conditions. That is the reason why we will divide the study of $H^{(\text{Clo})}_1(\pi,\rho)$ into two.

\noindent
\textit{Under $H^{(\emph{Clo})}_1(\pi,\rho)$, analysis 1.} Assume first that
\begin{equation}
\label{eq:T3H1Ass1}
  \|\Delta (\mathbf 1\{i \geq I\})_i\|_1 \geq C\Big[ \|q(\mathbf 1\{i \geq I\})_i\|_1 \lor \sqrt{\frac{\log(k)}{k}}\Big],  
\end{equation}
and
\begin{equation}
\label{eq:T3H1Ass2}
  \|\Delta (\mathbf 1\{i \geq I\})_i\|_1 \geq 2 \|\Delta (\mathbf 1\{i < I\})_i\|_1 .
\end{equation}

\noindent
We have
$$
\sum_{i\geq I} (\Delta_i + 2 q_i) = \sum_{i\geq I} (p_i + q_i) \geq \|\Delta (\mathbf 1\{i \geq I\})_i\|_1.
$$
So by Equation~\eqref{eq:T3H1Ass1},
\begin{align}
\label{eq:qDelta}
\sum_{i\geq I} \Delta_i \geq (C-2) \sum_{i\geq I} q_i,
\end{align}
and
\begin{align}
\label{eq:DeltaNorm}
\sum_{i\geq I} \Delta_i C/(C-2) \geq \|\Delta (\mathbf 1\{i \geq I\})_i\|_1.
\end{align}
Then since for any $i\geq I$, $q_i \leq 1/k$, Equation~\eqref{eq:qDelta} yields:
$$
\sum_{i} \Delta_i e^{-kq_i} \geq \sum_{i\geq I} \Delta_i e^{-kq_i} \geq \sum_{i\geq I} \Delta_i e^{-1} \geq (C-2)e^{-1} \sum_{i\geq I} q_i.
$$

\noindent
And again Equation~\eqref{eq:T3H1Ass1} gives:
$$
\sum_{i\geq I} \Delta_i + 2 \sum_{i\geq I} q_i \geq C \sqrt{\frac{\log(k)}{k}}.
$$

\noindent
So 
$$
\sum_{i\geq I} \Delta_i C/(C-2) \geq C \sqrt{\frac{\log(k)}{k}}.
$$

\noindent
So for $C$ large enough, we end up with:
$$\sum_i \Delta_i \exp (-kq_i) \geq \frac{C}{2} \Big[ \|q(\mathbf 1\{i \geq I\})_i\|_1 \lor \sqrt{\frac{\log(k)}{k}}\Big].$$
We then have by Equation~\eqref{eq:ET3}:
\begin{equation}
\label{eq:ET3fromAss1}
 \mathbb E T_1 = k \sum_i \Delta_i \exp (-kq_i)  \geq \frac{C}{2}  \Big[ \big(k\|q(\mathbf 1\{i \geq I\})_i\|_1\big) \lor \sqrt{k}\Big].   
\end{equation}
Now considering Equations~\eqref{eq:T3H1Ass2} and \eqref{eq:DeltaNorm}, we have for $C$ large enough,
$$
3\sum_{i\geq I} \Delta_i \geq 2 \sum_{i< I} |\Delta_i|.
$$
So
$$
9 \sum_{i \geq I} \Delta_i \geq 2 \sum_{i} |\Delta_i|,
$$
that is, by Equation~\eqref{eq:ET3fromAss1},
\begin{equation}
    \label{eq:ET3geqL1}
\frac{9}{2} \E T_1/k  \geq \|\exp(-kq)\Delta\|_1.
\end{equation}
And if the pre-test did not reject the null, then \co{with probability larger than $1- 2\delta - 2k^{-1}-6\exp(-k/18)$}, there exists $+\infty>c_\delta >0$ that only depends on $\delta$ and such that

$$
|\Delta_i| < c_{\delta} \left(\sqrt{q_i \frac{\log(k)}{k}}\vee \frac{\log(k)}{k}\right).
$$
If $q_i \geq \log(k) / k$, then
$|\Delta_i| < c_{\delta} \sqrt{q_i \log(k)/k}$.
So \co{
\begin{align*}
k\sqrt{\|\exp(-kq)\Delta^2\|_1} \leq c_{\delta} \sqrt{ \log(k) \|q\|_1} = c_{\delta} \sqrt{ \log(k)}.
\end{align*}}
If $q_i < \log(k) / k$, then
$|\Delta_i| < c_{\delta} \log(k)/k$.
So $$k\sqrt{\|\exp(-kq)\Delta^2\|_1} \leq \sqrt{c_{\delta} k \log(k) \|\exp(-kq)\Delta \|_1}.$$
Using Equation~\eqref{eq:VT3}, we end up \co{with probability larger than $1- 2\delta - 2k^{-1}-6\exp(-k/18)$}:

\begin{align*}
\sqrt{\mathbb V T_1}\leq  (\sqrt{2k \|q\exp(-kq)\|_1} + c_{\delta} \sqrt{\log k}) &+ \sqrt{k|\sum_i \Delta_i \exp(-kq_i)|} \\
&+ c_\delta \sqrt{k\log(k)}\sqrt{\|\exp(-kq)\Delta\|_1}. 
\end{align*}
Now let us compare the terms from the standard deviation $\sqrt{\mathbb V T_1}$ with $\E T_1$.

For the first term, we have
\co{
$$
(\sqrt{2k \|q\exp(-kq)\|_1} + c_{\delta} \sqrt{\log k }) \leq (2+c_{\delta})\sqrt{k} \leq \frac{2(2+c_{\delta})}{C} \E T_1.
$$
}
For the second term,

$$
\sqrt{k|\sum_i \Delta_i \exp(-kq_i)|} = \sqrt{\E T_1} k^{-1/4} k^{1/4}.
$$
So, since $2ab \leq a^2 +b^2$ for any $a,b$, we have:

$$
\sqrt{k|\sum_i \Delta_i \exp(-kq_i)|} \leq (k^{-1/2} \E T_1 + \sqrt{k})/2 \leq (k^{-1/2} + 2/C ) \E T_1/2.
$$
For the third term, in the same way,
$$
c_\delta \sqrt{k\log(k)}\sqrt{\|\exp(-kq)\Delta\|_1} \leq c_\delta (\sqrt{k}\|\exp(-kq)\Delta\|_1 \log(k) + \sqrt{k})/2.
$$
So we have by Equation~\eqref{eq:ET3geqL1}:
$$
c_\delta \sqrt{k\log(k)}\sqrt{\|\exp(-kq)\Delta\|_1} \leq c_\delta ( 9/2\log(k)/\sqrt{k} + 2/C)\E T_1/2.
$$
And so by Chebyshev's inequality, with probability larger than \co{$1 -\alpha - 2\delta - 2k^{-1}-6\exp(-k/18)$}, we have
$$|T_1-\E T_1| \leq \E T_1 \alpha^{-1/2} (2(2+c_{\delta})\sqrt{2\delta^{-1}}/C + (k^{-1/2} + 2/C)/2 + c_\delta( \frac{9}{2}\frac{\log(k)}{\sqrt{k}} + 2/C)/2 ).$$
So if $C \geq 4\alpha^{-1/2}(1 + 4 \sqrt{2\delta^{-1}} +c_\delta(1+2\sqrt{2\delta^{-1}}))$, and $2 k^{-1/2}\alpha^{-1/2}(1+ 9 c_\delta \log (k) /2) \leq 1$ (which is satisfied for $k$ large enough), we have with probability larger than \co{$1 -\alpha - 2\delta - 2k^{-1}-6\exp(-k/18)$}: 
$$
T_1 \geq \E T_1 / 2 \geq \frac{C}{4}  \Big[ \big(k\|q(\mathbf 1\{i \geq I\})_i\|_1\big) \lor \sqrt{k}\Big].
$$

\noindent
So we have
$$
T_1 \geq \frac{C}{4} \sqrt{k}.
$$

\noindent
\textit{Under $H^{(\emph{Clo})}_1(\pi,\rho)$, analysis 2.} The analysis remains the same as analysis 1, with $I$ replaced by $J$.

So the assumptions become:
\begin{equation*}
  \|\Delta (\mathbf 1\{i \geq J\})_i\|_1 \geq C\Big[ \|q(\mathbf 1\{i \geq J\})_i\|_1 \lor \sqrt{\frac{\log(k)}{k}}\Big],  
\end{equation*}
and
\begin{equation*}
  \|\Delta (\mathbf 1\{i \geq J\})_i\|_1 \geq 2 \|\Delta (\mathbf 1\{i < J\})_i\|_1 .
\end{equation*}
We then obtain, if $C \geq 4\alpha^{-1/2}(1 + 4 \sqrt{2\delta^{-1}} +c_\delta(1+2\sqrt{2\delta^{-1}}))$, and $2 k^{-1/2}\alpha^{-1/2}(1+ 9 c_\delta \log (k) /2) \leq 1$, we have with probability larger than \co{$1 -\alpha - 2\delta - 2k^{-1}-6\exp(-k/18)$}:
$$
T_1 \geq \frac{C}{4}  \Big[\big(k\|q(\mathbf 1\{i \geq J\})_i\|_1\big) \lor \sqrt{k}\Big].
$$

\noindent
So we have
$$
T_1 \geq \frac{C}{4} \sqrt{k}.
$$
Finally, the guarantees on the statistic $T_1$ allow us to conclude the proof.

\subsection{Proof of Theorem~\ref{th:upper} and Corollary~\ref{cor:upper}} 
\label{sec:proofThUpper}
%
%
%
Let us prove Theorem~\ref{th:upper} by combining all the guarantees on the ensemble of tests. From Propositions~\ref{th:test1}, \ref{th:test2}, \ref{th:test3}, we know that whenever $\Delta =0$, all tests accept the null with probability larger than $1-5\delta-2k^{-1}- 6\exp(-k/18)$. Besides, for $\tilde c_\delta$ large enough depending only on $\delta$, whenever there exists $I \geq J_q$ such that
\begin{align*}
&\|\Delta\|_1 \\
&\geq  \tilde c_\delta  \Bigg\{\Bigg[\Big(\sqrt{I-J_q}\frac{\log(k)}{k}\Big)
\lor \Big(\frac{\sqrt{I-J_q}}{\sqrt{k}}\|q^2\exp(-2kq)\|_1^{1/4}\Big) \lor  \|q(\mathbf 1\{i \geq I\})_i\|_1 \Bigg]\\
&\quad\quad\land \|q(\mathbf 1\{i \geq J_q\})_i\|_1  \Bigg\} \lor \Big[ \frac{\Big\|q^{2} \frac{1}{(q\lor k^{-1})^{4/3}} \Big\|_1^{3/4}}{\sqrt{k}} \Big] \lor \Big[ \sqrt{\frac{\log(k)}{k}}\Big],
\end{align*}
at least one test (and so the final test) rejects the null with probability larger than $1-5\delta-2k^{-1}- 6\exp(-k/18)$. 

\subsection{Proofs for the thresholds: Theorems~\ref{th:thresh1} and \ref{th:thresh2}}
\label{sec:proofThresh}

\subsubsection{Proof of Theorem~\ref{th:thresh1} for threshold $\hat t_{2/3}$}
\label{sec:thresh1}

\begin{lemma}\label{lem:bound2.3}
Let $Z\sim \mathcal P(\lambda)$, where $\lambda  \geq 0$. 
It holds that if $\lambda \geq 1$,
$$
e^{-2/3}\lambda^{2/3}/2 \leq \E(Z^{2/3}) \leq \lambda^{2/3},
$$
and if $\lambda \leq 1$,
$$
\lambda e^{-\lambda} \leq \E(Z^{2/3}) \leq \lambda.
$$
\end{lemma}

\begin{proof}[Proof of Lemma~\ref{lem:bound2.3}]\

\noindent \textbf{Upper bound on the expectation.} The function $t \rightarrow t^{2/3}$ is concave. So by application of Jensen's inequality, we have: $$ \E(Z^{2/3}) \leq \lambda^{2/3}.$$ Also we have by definition of the Poisson distribution
\begin{align*}
\E(Z^{2/3}) &= \sum_{i=1}^\infty \frac{\lambda^i}{(i-1)!} e^{-\lambda} i^{-1/3}\\
&= \lambda e^{-\lambda} \sum_{j \geq 0} \frac{\lambda^j}{j!} (j+1)^{-1/3} \leq \lambda.
\end{align*}
We conclude that if $\lambda \geq 1$,
$$
\E(Z^{2/3}) \leq \lambda^{2/3},
$$
and if $\lambda \leq 1$,
$$
\E(Z^{2/3}) \leq \lambda.
$$
\noindent \textbf{Lower bound on the expectation in the case $\lambda \geq e^{-2}$.} We have by definition of the Poisson distribution
\begin{align*}
\E(Z^{2/3}) &= \sum_{i=1}^\infty \frac{\lambda^i}{(i-1)!} e^{-\lambda} i^{-1/3}\\
&\geq \lambda e^{-\lambda} \sum_{0\leq j \leq e^2\lambda-1} \frac{\lambda^j}{j!} (j+1)^{-1/3}\\
&\geq e^{-2/3}\lambda^{2/3} e^{-\lambda} \sum_{0\leq j \leq e^2\lambda-1} \frac{\lambda^j}{j!},
\end{align*}
because $e^2\lambda -1 \geq 0$ here. Then since $j! \geq \sqrt{2\pi} j^je^{-j}$,
\begin{align*}
\E(Z^{2/3}) &\geq e^{-2/3}\lambda^{2/3} \left(1-\frac{e^{-\lambda}}{\sqrt{2\pi}}\sum_{j \geq \floor{e^2\lambda}} \frac{\lambda^j e^j}{j^j}\right) \\
&\geq e^{-2/3}\lambda^{2/3} \left(1-\frac{e^{-\lambda}}{\sqrt{2\pi}}\sum_{j \geq \floor{e^2\lambda}} \frac{\lambda^j e^j}{(c_{\floor{}}e^2\lambda)^j}\right) \\
&\geq e^{-2/3}\lambda^{2/3} \left(1-\frac{e^{-\lambda}}{\sqrt{2\pi}}\sum_{j \geq \floor{e^2\lambda}} (c_{\floor{}} e)^{-j}\right),
\end{align*}
where $1/2 \leq c_{\floor{}} \leq 1$ such that $c_{\floor{}} e^2 \lambda = \floor{e^2\lambda}$ because $e^2 \lambda \geq 1$.
Finally,
\begin{align*}
\E(Z^{2/3})&\geq e^{-2/3}\lambda^{2/3} \left(1-\frac{e^{-\lambda}}{\sqrt{2\pi}} (c_{\floor{}} e)^{-\floor{e^2\lambda}} \frac{1}{1-(c_{\floor{}}e)^{-1}}\right).
\end{align*}
In particular, if $\lambda \geq 1$,
$$
\E(Z^{2/3})\geq e^{-2/3}\lambda^{2/3}/2.
$$
\noindent \textbf{Lower bound in all cases.} Without any assumption on $\lambda$ it holds that 
$\E(Z^{2/3}) \geq \lambda e^{-\lambda}$.

\noindent \textbf{Conclusion on the lower bound.} 

So, if $\lambda\geq 1$, $\E(Z^{2/3}) \geq e^{-2/3}\lambda^{2/3}/2$,
and if $\lambda \leq 1$, $\E(Z^{2/3}) \geq \lambda e^{-\lambda}$

\end{proof}

\begin{lemma}\label{lem:bound4.3}
Let $Z\sim \mathcal P(\lambda)$, where $\lambda  \geq 0$. It holds if $\lambda \geq e^{-2}$ that
$$\E(Z^{4/3}) \leq \lambda^{4/3} e^{8/3} + e^{-\lambda} \frac{1}{1-e^{-1/2}},$$
and if $\lambda < e^{-2}$ that
$$\E(Z^{4/3}) \leq e^{-\lambda} \lambda (2^{1/3} + e).$$
\end{lemma}

\begin{proof}[Proof of Lemma~\ref{lem:bound4.3}]
Assume that $\lambda \geq e^{-2}$. We have by definition of the Poisson distribution
\begin{align*}
    \E(Z^{4/3}) &= \sum_{i\geq 1} \frac{\lambda^i}{i!} e^{-\lambda} i^{4/3}\\
    &= \sum_{1\leq i \leq e^2\lambda} \frac{\lambda^i}{i!} e^{-\lambda} i^{4/3} + \sum_{i> e^2\lambda} \frac{\lambda^i}{i!} e^{-\lambda} i^{4/3}\\
    &\leq \lambda^{4/3} e^{8/3} + e^{-\lambda} \sum_{i\geq e^2\lambda} \frac{\lambda^ie^i}{i^i} i^{4/3},
\end{align*}    
using the inequality: $i! \geq i^i/e^i$. Then,
\begin{align*}
    \E(Z^{4/3})&\leq \lambda^{4/3} e^{8/3} + e^{-\lambda} \sum_{i\geq e^2\lambda} \frac{\lambda^ie^i}{(e^2\lambda)^i} i^{4/3}\\
    &= \lambda^{4/3} e^{8/3} + e^{-\lambda} \sum_{i\geq e^2\lambda} i^{4/3} e^{-i}\\
    &\leq \lambda^{4/3} e^{8/3} + e^{-\lambda} \sum_{i\geq e^2\lambda} e^{-i/2}\\
    &\leq \lambda^{4/3} e^{8/3} + e^{-\lambda} \frac{1}{1-e^{-1/2}}.
\end{align*}
Now assume that $e^2 \lambda < 1$. Then
\begin{align*}
    E(Z^{4/3}) &= \sum_{i\geq 1} \frac{\lambda^i}{i!} e^{-\lambda} i^{4/3}\\
    &\leq e^{-\lambda} \lambda \left(1+ \sum_{j\geq 0} \frac{(j+2)^{1/3}}{j+1} \frac{1}{j!}\right)\\
    &\leq e^{-\lambda} \lambda \left(1+ 2^{1/3} + \sum_{j\geq 1} \frac{1}{j!}\right)\\
    &= e^{-\lambda} \lambda (2^{1/3} + e)
\end{align*}

\end{proof}

\begin{proof}[Proof of Theorem~\ref{th:thresh1}]
By application of Lemma~\ref{lem:bound2.3}, we have the following bounds on the expectation of the empirical threshold with probability larger than $1-6\exp(-k/18)$ with respect to $(\bar k_m^{(j)})_{m \leq 2, j \leq 3}$,
\begin{align*}
&\left\|\frac{e^{-2/3}q^{2/3}}{2} \mathbf 1\{q\geq 1/k\}\right\|_1 \\
&\quad\quad+ \left\| k^{1/3}q e^{-kq}  \mathbf 1\{q\leq 1/k\} \right\|_1 \leq k^{-2/3}\E \|(Y^{(1)})^{2/3}\|_1\leq \left\|q^{2/3}\mathbf 1\{q\geq 1/k\}\right \|_1 \\
&\quad\quad\quad\quad\quad\quad\quad\quad\quad
\quad\quad\quad\quad\quad\quad\quad\quad\quad\quad\quad
\quad\quad\quad\quad\quad\quad
+ \left\| q \mathbf 1\{q\leq 1/k\}\right \|_1.
\end{align*}
Now let us consider the standard deviation of the empirical threshold. We have by application of Lemma~\ref{lem:bound4.3}, with probability larger than $1-6\exp(-k/18)$ with respect to $(\bar k_m^{(j)})_{m \leq 2, j \leq 3}$
\begin{align*} 
&k^{-2/3}\sqrt{\V \|(Y^{(1)})^{2/3}\|_1} \\
&\leq \sqrt{\|[q^{4/3} e^{8/3} + \frac{1}{1-e^{-1/2}} k^{-4/3} e^{-kq}] \mathbf 1\{q\geq 1/k\} \|_1} \\
&\quad+ \sqrt{\|k^{-1/3}(2^{1/3} +e)qe^{-kq} \mathbf 1\{q\leq 1/k\} \|_1}\\
&\leq \sqrt{\|q^{4/3} e^{8/3} \mathbf 1\{q\geq 1/k\} \|_1 +1} + \sqrt{\|k^{-1/3}(2^{1/3} +e)qe^{-kq} \mathbf 1\{q\leq 1/k\} \|_1}\\
&\leq \sqrt{\|q^{4/3} e^{8/3} \mathbf 1\{q\geq 1/k\} \|_1 +1} + \sqrt{(2^{1/3} +e)}\\
&\leq \sqrt{e^{8/3} +1} + \sqrt{(2^{1/3} +e)} = C_1.
\end{align*}
Then by application of Chebyshev's inequality, we have with probability greater than $1-\beta-6\exp(-k/18)$,
\begin{align*}
&\left\|\frac{e^{-2/3}q^{2/3}}{2} \mathbf 1\{q\geq 1/k\}\right \|_1 \\
&+ \left\| k^{1/3}q e^{-kq} \mathbf 1\{q\leq 1/k\} \right\|_1 -\frac{C_1}{\sqrt \beta} \\
&\quad\quad\quad\quad\quad\quad\quad\quad\quad\quad\quad\quad\quad\quad\quad
\leq k^{-2/3} \|(Y^{(1)})^{2/3}\|_1 \\
&\quad\quad\quad\quad\quad\quad\quad\quad\quad\quad\quad\quad\quad\quad\quad\quad\quad\quad\quad\quad\quad\quad\quad\quad\quad
\leq  \left\|q^{2/3}\mathbf 1\{q\geq 1/k\}\right \|_1 \\
&\quad\quad\quad\quad\quad\quad\quad\quad\quad\quad\quad\quad\quad\quad\quad\quad\quad\quad\quad\quad\quad\quad\quad\quad
+ \left\| q \mathbf 1\{q\leq 1/k\}\right \|_1+\frac{C_1}{\sqrt \beta}.
\end{align*}
Now, on the one hand, we have that
$$k^{4/3}q^2 \mathbf 1\{q \leq 1/k\} \leq k^{1/3}q \mathbf 1\{q \leq 1/k\},$$
so $$\Big\|\Big(\frac{1}{q \lor k^{-1}}\Big)^{4/3}   q^2 \mathbf 1\{q \leq 1/k\}\Big\|_1 \leq\|k^{4/3}q^2 \mathbf 1\{q \leq 1/k\}\|_1 \leq \|k^{1/3}q \mathbf 1\{q \leq 1/k\}\|_1.$$
And on the other hand, $$\Big\|\Big(\frac{1}{q \lor k^{-1}}\Big)^{4/3}   q^2 \mathbf 1\{q\geq 1/k\}\Big\|_1 \leq \| q^{2/3} \mathbf 1\{q\geq 1/k\}\|_1.$$
So with probability larger than $1-\beta-6\exp(-k/18)$ we have
\begin{align*}
&(e^{-2/3}/2 + 1)\Big\|\Big(\frac{1}{q \lor k^{-1}}\Big)^{4/3}   q^2\Big\|_1\\
&\quad\quad\quad\quad\quad\quad\quad\quad\quad
\leq k^{-2/3} \|(Y^{(1)})^{2/3}\|_1 +C_{2/3}/\sqrt \beta\\
 &\quad\quad\quad\quad\quad\quad\quad\quad\quad\quad\quad\quad\quad\quad\quad\quad\quad
 \leq \left\|q^{2/3}\mathbf 1\{q\geq 1/k\}\right \|_1 + 1 +2C_{2/3}/\sqrt \beta.
\end{align*}
\end{proof}

\subsubsection{Proof of Theorem~\ref{th:thresh2} for threshold $\hat t_2$}
\label{sec:thresh2}

\begin{lemma}
\label{lem:Et2}
Consider three independent random vectors $Y^{(1)}$, $Y^{(2)}$ and $Y^{(3)}$ distributed according to $\mathcal P(kq)$. We obtain the following expectation:
$$
\E(\|Y^{(1)}Y^{(2)} \mathbf 1\{Y^{(3)}=0\}\|_1) = \|(kq)^2 e^{-kq}\|_1.
$$
\end{lemma}

\begin{proof}
Firstly,
$$
\E(\|Y^{(1)}Y^{(2)} \mathbf 1\{Y^{(3)}=0\}\|_1) = \|\E(Y^{(1)}Y^{(2)} \mathbf 1\{Y^{(3)}=0\})\|_1.
$$
Now
$$
\E(Y^{(1)}Y^{(2)} \mathbf 1\{Y^{(3)}=0\}) = \E(Y^{(1)})\E(Y^{(2)}) \Po(Y^{(3)}=0)= (kq)^2 e^{-kq}.
$$
So
$$
\E(\|Y^{(1)}Y^{(2)} \mathbf 1\{Y^{(3)}=0\}\|_1) = \|(kq)^2 e^{-kq}\|_1.
$$
\end{proof}

\begin{lemma}
\label{lem:Vt2}
Consider three independent random vectors $Z^{(1)}$, $Z^{(2)}$ and $Z^{(3)}$ distributed according to $\mathcal P(kq)$ and whose elements are independent, too. We obtain the following variance:
$$
\V(\|Z^{(1)}Z^{(2)} \mathbf 1\{Z^{(3)}=0\}\|_1) = \|((kq)^2 + kq)^2 e^{-kq} - (kq)^4 e^{-2kq} \|_1.
$$
\end{lemma}

\begin{proof}
Each sample $Z^{(i)}$ consists in a vector of independent elements. So
$$
\V(\|Z^{(1)}Z^{(2)} \mathbf 1\{Z^{(3)}=0\}\|_1) = \|\V(Z^{(1)}Z^{(2)} \mathbf 1\{Z^{(3)}=0\})\|_1.
$$
Now
$$
\V(Z^{(1)}Z^{(2)} \mathbf 1\{Z^{(3)}=0\}) = \cro{\E((Z^{(1)})^2) \E((Z^{(2)})^2) - \E((Z^{(1)}))^2 \E((Z^{(2)}))^2} \Po(Z^{(3)} = 0)
$$
by independence between $Z^{(1)}$, $Z^{(2)}$ and $Z^{(3)}$.

And 
$$
\E((Z^{(1)})^2) = E((Z^{(2)})^2) = (kq)^2 + kq.
$$
So 
$$
\V(\|Z^{(1)}Z^{(2)} \mathbf 1\{Z^{(3)}=0\}\|_1) = \|((kq)^2 + kq)^2 e^{-kq} - (kq)^4 e^{-2kq} \|_1.
$$
\end{proof}

\begin{proof}[Proof of Theorem~\ref{th:thresh2}]
By application of lemma~\ref{lem:Et2}, we have for the expectation of the empirical threshold with probability larger than $1-6\exp(-k/18)$ with respect to $(\bar k_m^{(j)})_{m \leq 2, j \leq 3}$:
$$
\E(\|Y^{(1)}Y^{(2)} \mathbf 1\{Y^{(3)}=0\}\|_1) = \|(kq)^2 e^{-kq}\|_1.
$$
Then by application of lemma~\ref{lem:Vt2}, we have for the standard deviation of the empirical threshold with probability larger than $1-6\exp(-k/18)$ with respect to $(\bar k_m^{(j)})_{m \leq 2, j \leq 3}$:
$$
\sqrt{\V(\|Y^{(1)}Y^{(2)} \mathbf 1\{Y^{(3)}=0\}\|_1)} \leq \sqrt{2} (\sqrt{\|(kq)^4 e^{-kq}\|_1} + \sqrt{\|(kq)^2 e^{-kq}\|_1}).
$$
In particular,
$$
\E(\|Y^{(1)}Y^{(2)} \mathbf 1\{Y^{(3)}=0\}\|_1/k^2) = \|q^2 e^{-kq}\|_1,
$$
and
$$
\sqrt{\V(\|Y^{(1)}Y^{(2)} \mathbf 1\{Y^{(3)}=0\}\|_1/k^2)} \leq \sqrt{2} (\sqrt{\|q^4 e^{-kq}\|_1} + \sqrt{\|q^2 e^{-kq}\|_1}/k).
$$
Let us compare both terms of the standard deviation with the expectation.

Firstly,
\begin{align*}
\sqrt{\|q^2 e^{-kq}\|_1}/k &\leq 1/2 (\|q^2 e^{-kq}\|_1\sqrt{\delta}/4 + 4/(\sqrt{\delta} k^2)) \\
&\leq 1/2 (\|q^2 e^{-kq}\|_1\sqrt{\delta}/4 + 4\log(k)^4/(\sqrt{\delta} k^2)).
\end{align*}
Secondly, for an upper bound on $\sqrt{\|q^4 e^{-kq}\|_1}$, we consider two regimes.

\noindent \textbf{Study of the large $q_i$'s.}

We consider $q_i \geq 5\log (k)/k$.
Then we have the following upper bound on the number of such $q_i$'s, $$\#\{i | q_i \geq 5\log (k)/k\} \leq 1/(5\log (k)/k).$$
So
$$
\sqrt{\|q^4 e^{-kq} \mathbf 1\{q \geq 5\log (k)/k\}\|_1} \leq k^{-2}/ \sqrt{5 \log k} \leq 3 \log(k)^4/k^2.
$$

\noindent \textbf{Study of the small $q_i$'s.}

We consider $q_i < 5\log (k)/k$.
\begin{align*}
    \sqrt{\|q^4 e^{-kq} \mathbf 1\{q < 5\log (k)/k\}\|_1} &\leq  5\log (k)/k \sqrt{\|q^2 e^{-kq}\|_1} \\
    &\leq 1/2 (\|q^2 e^{-kq}\|_1\sqrt{\delta}/4 + 100\log (k)^2/(\sqrt{\delta}k^2)) \\
    &\leq  1/2 (\|q^2 e^{-kq}\|_1\sqrt{\delta}/4 + 2000\log(k)^4/(\sqrt{\delta}k^2)).
\end{align*}
Finally,
\begin{align*}
&\sqrt{\V(\|Y^{(1)}Y^{(2)} \mathbf 1\{Y^{(3)}=0\}\|_1/k^2)} \\
&\leq \frac{\sqrt{\delta}}{2} \E(\|Y^{(1)}Y^{(2)} \mathbf 1\{Y^{(3)}=0\}\|_1/k^2) + 1005\log(k)^4/(\sqrt{\delta}k^2).
\end{align*}
So by application of Chebyshev's inequality, we have with probability greater than $1-\delta-6\exp(-k/18)$:
$$
|\|Y^{(1)}Y^{(2)} \mathbf 1\{Y^{(3)}=0\}\|_1 - \|(kq)^2 e^{-kq}\|_1|\leq 1/2 \|(kq)^2 e^{-kq}\|_1 + \frac{1005}{\delta} \log(k)^4.
$$
\end{proof}

%% file: LBproof.tex
\section{Proofs of the lower bounds: Propositions~\ref{th:fromVal}, \ref{th:lower}, \ref{th:lower2} and Theorem~\ref{th:lowerall}}
\label{sec:proofLower}

\subsection{Proof of Propositions~\ref{th:fromVal}}

The lower bound obtained in \cite{valiant2017automatic,balakrishnan2017hypothesis} for identity testing will also be useful to us as a lower bound for closeness testing.

\begin{proof}[Proof of Proposition~\ref{th:fromVal} and adaptation to Theorem~\ref{th:lowerall}]

As a corollary from Theorem~1 in \cite{balakrishnan2017hypothesis}, we have that there exists a constant $c_\gamma'>0$ that depends only on $\gamma$ such that for any $q \in \mathbf P_\pi$
\begin{align*}
	&\rho_\gamma^*(H^{(\text{Clo})}_0(\pi),~H^{(\text{Clo})}_1(\pi,\rho_\gamma^*); k) \\
	&\quad\quad\quad\quad\quad\quad\quad\quad\quad\quad \geq c_\gamma' \min_{I}\Big[\frac{\|q_{(.)}^{2/3}(\mathbf 1\{2 \leq i< I\})_i\|_1^{3/4}}{\sqrt{k}} \lor \frac{1}{k} + \|q_{(.)}(\mathbf 1\{i\geq I\})_i\|_1\Big].
\end{align*}
In particular, taking $q = \pi$, there exists a constant $c_\gamma'>0$ that depends only on $\gamma$ and such that
\begin{align*}
	&\rho_\gamma^*(H^{(\text{Clo})}_0(\pi),~H^{(\text{Clo})}_1(\pi,\rho_\gamma^*); k)  \\
	&\quad\quad\quad\quad\quad\quad\quad\quad\quad\quad\geq c_\gamma' \min_{I}\Big[\frac{\|\pi_{(.)}^{2/3}(\mathbf 1\{2 \leq i< I\})_i\|_1^{3/4}}{\sqrt{k}} \lor \frac{1}{k} + \|\pi_{(.)}(\mathbf 1\{i\geq I\})_i\|_1\Big].
\end{align*}
\end{proof}

\noindent
We then adapt Proposition~\ref{th:fromVal} to the purpose of obtaining Theorem~\ref{th:lowerall}.

\begin{proposition}
\label{th:lower3}
Let $\pi\in \mathbf{P}$ and $\gamma>0$. There exists a constant $c_\gamma>0$ that depends only on $\gamma$ such that 
$\rho_\gamma^*(H^{(\emph{Clo})}_0(\pi),~H^{(\emph{Clo})}_1(\pi,\rho_\gamma^*); k) \geq \frac{c_\gamma}{\sqrt{k}} \left[\left\|\pi^{2/3} \frac{1}{(\pi \lor k^{-1})^{4/3}} \right\|_1^{3/4} \lor 1\right].$
\end{proposition}

\begin{proof}[Proof of Proposition~\ref{th:lower3}]

From Proposition~\ref{th:fromVal}, there exists a constant $c_\gamma'>0$ that depends only on $\gamma$ and such that
\begin{align*}
	&\rho_\gamma^*(H^{(\text{Clo})}_0(\pi),~H^{(\text{Clo})}_1(\pi,\rho_\gamma^*); k) \\
	&\quad\quad\quad\quad\quad\quad\quad\quad\quad\quad\geq c_\gamma' \min_{I}\Big[\frac{\|\pi_{(.)}^{2/3}(\mathbf 1\{2 \leq i< I\})_i\|_1^{3/4}}{\sqrt{k}} \lor \frac{1}{k} \vee \|\pi_{(.)}(\mathbf 1\{i\geq I\})_i\|_1\Big].
\end{align*}
Let $I^*$ denote one of the $I$'s where the minimum  from the right-hand side of the previous inequality is attained.

\textit{Case 1: $\|\pi_{(.)}(\mathbf 1\{i> I^*\})_i\|_1 > 1/2$.} The result follows immediately.

\textit{Case 2: $\|\pi_{(.)}(\mathbf 1\{i> I^*\})_i\|_1 \leq 1/2$}. So $\|\pi_{(.)}^{2/3}(\mathbf 1\{i\leq I^*\})_i\|_1 \geq 1/2$ since $\|\pi\|_1 = 1$, implying that $\rho_{\gamma}^*  \geq c_\gamma' (1/2)^{3/4}/\sqrt{k}$.

\textit{Subcase 1: $I^* \geq J_{\pi}$.} We have $\|\pi_{(.)}(\mathbf 1\{i> I^*\})_i\|_1 \geq \|\pi^2 (\mathbf 1\{i> I^*\})_i\|_1\sqrt{k}$.

\textit{Subcase 2: $I^* < J_{\pi}$.} Having for all $J_{\pi} \geq i> I^*$, $\pi_{(i)} \geq 1/k$ implies that we have 
$$\|\pi_{(.)}^{2/3}(\mathbf 1\{J_{\pi} \geq i> I^*\})_i\|_1 \leq k^{1/3}\|\pi_{(.)}(\mathbf 1\{J_{\pi} \geq i> I^*\})_i\|_1.$$
And so
$$\|\pi_{(.)}^{2/3}(\mathbf 1\{J_{\pi} \geq i> I^*\})_i\|_1^{3/4} \leq k^{1/2}\|\pi_{(.)}(\mathbf 1\{J_{\pi} \geq i> I^*\})_i\|_1$$
since $\|\pi_{(.)}(\mathbf 1\{J_{\pi} \geq i> I^*\})_i\|_1 \leq 1.$

 Finally $$\|\pi_{(.)}^{2/3}(\mathbf 1\{J_{\pi} \geq i> I^*\})_i\|_1^{3/4} /\sqrt{k}\leq \|\pi_{(.)}(\mathbf 1\{J_{\pi} \geq i> I^*\})_i\|_1,$$
  which implies that $I^*$ must be larger than $J_{\pi}$. This concludes the proof in any case.
\end{proof}

This concludes the proof of Proposition~\ref{th:fromVal} .

\subsection{Classical method for proving lower bounds: the Bayesian approach}
\label{sec:bayes}

Let us fix some $\gamma \in(0,1)$. Finding a lower bound on $\rho_\gamma^*(H^{(\text{Clo})}_0(\pi),H^{(\text{Clo})}_1(\pi); k)$ amounts to finding a real number $\rho$ such that $R(H^{(\text{Clo})}_0,H^{(\text{Clo})}_1,\varphi;\rho, k) > \gamma$ for any test $\varphi$.
%

Let us now apply a Bayesian approach. Let $0\leq \alpha < 1$. Let $\nu_0$ be a distribution such that $(p,q) \in H^{(\text{Clo})}_0$ is true $\nu_0$-almost surely and $\nu_1$ such that: 
$$
\Po_{\nu_1}((p,q)\in H^{(\text{Clo})}_1) \geq 1-\alpha.
$$
Then
\begin{align*}
R(H^{(\text{Clo})}_0,H^{(\text{Clo})}_1,\varphi;\rho, k) &= \sup_{(p,q)\in H^{(\text{Clo})}_0} \mathbb{P}_{p,q}(\varphi(\mathcal X, \mathcal Y)=1) + \sup_{(p,q)\in H^{(\text{Clo})}_1} \mathbb{P}_{p,q}(\varphi(\mathcal X, \mathcal Y)=0) \\
&\geq \Po_{\nu_0} (\varphi(\mathcal X, \mathcal Y) = 1) + \Po_{\nu_1}(\varphi(\mathcal X, \mathcal Y)=0,\ (p,q)\in H^{(\text{Clo})}_1)\\
&\geq \Po_{\nu_0} (\varphi(\mathcal X, \mathcal Y) = 1) + \Po_{\nu_1}(\varphi(\mathcal X, \mathcal Y)=0) - \alpha.
\end{align*}
Let us define the total variation distance as in \cite{baraud2002non}, $d_{TV}: (\nu_0,\nu_1) \rightarrow 2 \sup_{A} |\nu_0(A)-\nu_1(A)|.$
So for any $\varphi$
\begin{equation}
\label{eq:riskTV}
R(H^{(\text{Clo})}_0,H^{(\text{Clo})}_1,\varphi;\rho, k) \geq 1 - d_{TV}(\Po_{\nu_0},\Po_{\nu_1})/2 - \alpha = 1 -  d_{TV}(\nu_0,\nu_1)/2 - \alpha.
\end{equation}
Thus, the lower bound that is obtained heavily relies on the choice of $\nu_0$ and $\nu_1$.

\newpage

\subsection{Proof of Proposition~\ref{th:lower}}

Let us prove the lower bound stated in Proposition~\ref{th:lower}. It heavily relies on the Bayesian approach presented in Section~\ref{sec:bayes}.
Let us recall the definition of $I_{v,\pi}$.
Set for $v\geq 0$, with the convention $\min_{j \leq d} \emptyset =d$,
\begin{align}
\label{eq:defI}
\begin{split}
I_{v,\pi} = \min_{J_\pi \leq j \leq d} \Big\{ \{j : \pi_{(j)} \leq \sqrt{C_\pi/j}\} &\cap \{ j : \sum_{i\geq j} \exp(-2k\pi_{(i)}) \pi_{(i)}^2 \leq C_\pi\} \\
&\cap \{ j :\sum_{i \geq j} \pi_{(i)} \leq \sum_{J_\pi \leq i < j} \pi_{(i)} \}\Big\},
\end{split}
\end{align}
where 
\begin{equation}
\label{eq:defC}
C_\pi = \frac{\sqrt{\sum_i \pi_i^2\exp(-2(1+v)k\pi_i)}}{k}.
\end{equation}
In what follows, we also consider for any $i\leq d$, the quantity $\varepsilon_i^* \in [0,1/2]$ that we will specify later.

\paragraph{Definition of some measures.}
We assume that $\pi \in \mathbf P$ is fixed such that for any $i \leq j \leq d$, we have $\pi_i \geq \pi_j$. 
Let $0<\delta \leq 1/8$ and $4[1 \vee (32 \log(1/\delta))^2] \leq M \leq \sqrt k$.
We define $\mathcal A \subset \{1, \ldots, d\}$ such that for any $i \in \mathbb Z$ where $|S_\pi(\floor{\log_2(k)} + i)| > a \sqrt{\log((|i|+1)/\gamma)}$, we have that $\mathcal A \cap S_\pi(\floor{\log_2(k)} + i)$ are the $\floor{|S_\pi(\floor{\log_2(k)}  + i)| / M}$ largest elements of $S_\pi(\floor{\log_2(k)}  + i)$. Note that $\sum_{\mathcal A} \pi_i \leq 2/M \leq  1/2$. Let $\mathcal A' = \mathcal A \cap \{J_\pi, \ldots, d\}$.

Let $U_{\pi,{\mathcal A}}$ be the discrete distribution that is uniform over $\{\pi_i, i\in \mathcal A\}$, and we will write for any $i \in {\mathcal A}$, $U_{\pi,{\mathcal A}}(\{\pi_i\}) = 1/|{\mathcal A}|$. We will now work on the definition of appropriate measures corresponding to $(q,p)$.
Conditional to two vectors $q,p\in \mathbb R^{+d}$, we define 
$$\Lambda_{q,p} = \prod_{i} \pa{\mathcal P(kq_{i}) \otimes \mathcal P(kp_{i})}.$$

\noindent
\underline{Definition of $\Lambda_0$:} First, for any $i \in {\mathcal A}$ we will consider independent $q_i \sim U_{\pi,{\mathcal A}}$, and otherwise set $q_i= \pi_i$ for any $i \not\in {\mathcal A}$. 
We write $\Lambda_0$ for the distribution $\Lambda_{q,q}$ when $q$ is defined as before:
$$\Lambda_0 = \E_{q}(\Lambda_{q,q}),$$
where $\E_{q}$ is the expectation according to the distribution of $q$.\\

\noindent
\underline{Definition of $\Lambda_1$:} We consider $ q$ defined as above. For any $i \in {\mathcal A}$, we know that there exists $j_i$ such that $q_i = \pi_{j_i}$. Let us write $\xi_{i}$ for independent random variables that are uniform in $\{\varepsilon_{j_i}^*,-\varepsilon_{j_i}^*\}$ if $i \in \mathcal A'$, and $0$ otherwise. Then set for any $i \in \mathcal A$: define $p_{i} = q_i(1+\xi_{i})$ and for any $i \not\in \mathcal A$: define $p_{i} = \pi_i$. We write $\Lambda_1$ for the distribution $\Lambda_{q,p}$ averaged over $q,p$. So
$$\Lambda_1 = \E_{q,p}(\Lambda_{q,p}),$$
where $\E_{q,p}$ is the expectation according to the distribution of $q,p$ -- i.e.~according to $q,\xi$.\\

\noindent
\underline{Definition of $\tilde \Lambda_0$:} Now, for any $i \in {\mathcal A}$ we will consider independent $\tilde q_i = q_i\sim U_{\pi,{\mathcal A}}$. Then for any $i \not\in {\mathcal A}$: set $\tilde q_{i} = \pi_i\pa{1-\frac{\sum_{l\in {\mathcal A}}(\tilde q_l - \pi_l)}{\sum_{j \notin {\mathcal A}} \pi_j}}$. We write $\tilde \Lambda_0$ for the distribution $\Lambda_{\tilde q,\tilde q}$ when $\tilde q$ is defined as before:
$$\tilde \Lambda_0 = \E_{\tilde q}(\Lambda_{\tilde q,\tilde q}),$$
where $\E_{\tilde q}$ is the expectation according to the distribution of $\tilde q$ -- i.e.~according to $q$.\\

\noindent
\underline{Definition of $\tilde \Lambda_1$:} We consider $q,\xi,\tilde q$ defined as above. Then for any $i \in \mathcal A$, set $\tilde p_{i} =  p_i  = q_i(1+\xi_{i})$. For any $i \not\in \mathcal A$, set $\tilde p_{i} = \pi_{i}\pa{1-\frac{\sum_{l\in {\mathcal A}} (\tilde p_l - \pi_l)}{\sum_{j \notin {\mathcal A}} \pi_{j}}}$. We write $\Lambda_1$ for the distribution $\Lambda_{\tilde q,\tilde p}$ averaged over $\tilde q,\tilde p$. So
$$\Lambda_1 = \E_{\tilde q,\tilde p}(\Lambda_{\tilde q,\tilde p}),$$
where $\E_{\tilde q,\tilde p}$ is the expectation according to the distribution of $\tilde q,\tilde p$ -- i.e.~according to $q, \xi$.
Note that $\sum_{i} \tilde q_i = 1 = \sum_{i} \tilde p_i.$

\paragraph*{Properties of $\tilde \Lambda_0$ and $\tilde \Lambda_1$, and bound on their total variation distance.}

We first prove the following lemma, which implies that $\tilde \Lambda_0$ and $\tilde \Lambda_1$ take values in $\mathbf P_\pi$ with high probability.

\begin{lemma}\label{lem:lam0lam1}
Assume that $M \geq 4(16 \log{2/\delta})^2$, and that $a>2$. There exists a universal constant $c>0$ such that we have with probability larger than $1-\delta - c \gamma$ with respect to $\tilde q$ that $\tilde q \in \mathbf P_\pi$, and with probability larger than $1-\delta - c \gamma$ with respect to $\tilde p$ that $\tilde p \in \mathbf P_\pi$.
\end{lemma}

We now turn to $d_{TV}(\Lambda_0,\Lambda_1)$ and state the following lemmas which will help us conclude on a bound on $d_{TV}(\Lambda_0,\Lambda_1)$.
\begin{lemma}\label{lem:lower}
	Let $\pi\in (\mathbb R^+)^{d}$ such that $\sum_i \pi_i \leq 1$ and such that it is ordered in decreasing order, i.e.~$\forall i \leq j \leq d$, $\pi_i \geq \pi_j$. We remind the reader that $J := J_\pi$.
	Let $1>u>0, v\geq 0$. Then there exists $\varepsilon^* \in \mathbb R^d$ such that for any $i \leq d$
	\begin{itemize}
		\item $\varepsilon^*_i \in [0,1/2]$ and $\varepsilon^*_i = 0$ for any $i$ such that $\pi_i \geq 1/k$.
		\item $\sum_{i} \pi_i^2 \varepsilon^{*2}_i \exp(-2k\pi_i) \leq u\frac{\sqrt{\sum_i \pi_i^2\exp(-2(1+v)k\pi_i)}}{k} :=uC_\pi$.
		\item we have $\pi_i \varepsilon^*_i \leq \sqrt{u}\Big[ (1/k) \land \sqrt{C_\pi/(2I_{v,\pi})} \land \pi_i/2\Big]$.
		\item and we have
		$$\sum_i \pi_i \varepsilon_i^* \geq \Bigg[\Big[\sum_{i\geq I_{v,\pi}} \frac{\sqrt{u}\pi_i}{\sqrt{2}} \Big]\vee \frac{\sqrt{uC_\pi} (I_{v,\pi}-J)}{\sqrt{2I_{v,\pi}}}\Bigg] \land \Big[ \sqrt{\frac{u}{8}} \sum_{i \geq J} \pi_i\Big].$$
	\end{itemize}	
\end{lemma}

We show the following lemma which will help us conclude on a bound on $d_{TV}(\Lambda_0,\Lambda_1)$.
\begin{lemma}
\label{lem:TV01marginal}
Let $\pi$ satisfying the hypotheses of Lemma~\ref{lem:lower}. We take $\varepsilon^*$ associated with $\pi$ as in Lemma~\ref{lem:lower} for some $u>0, v > 0$. Write $\lambda = k\pi$. There exists a constant $\tilde c_v>0$ that depends only on $v$ such that the following holds. Let $\xi = \xi_{|\theta}$ be a random variable that depends on $\theta$ and takes a random value uniformly in $\{\varepsilon_j^*,-\varepsilon_j^*\}$ when $\theta = \lambda_j= k\pi_j$. Write $\mathcal U_{\lambda}$ for the uniform distribution over $\{\lambda_i = k\pi_i, i \leq d\}$. We set for $\varepsilon^* \in [0,1]^d$, conditionally on $\theta$:
$$\nu_{0|\theta} = \mathcal P(\theta)^{\otimes 2}, \quad\quad \nu_0 = \E_{\theta \sim \mathcal U_{\lambda}}(\nu_{0|\theta}),$$
and
$$\nu_{1|\theta, \xi} = \mathcal P(\theta) \otimes \frac{\Big[\mathcal P(\theta(1+\xi)) + \mathcal P(\theta(1-\xi))\Big]}{2}, \quad \nu_1 = \E_{\theta \sim \mathcal U_{\lambda},\xi}(\nu_{1|\theta, \xi}).$$

We have
\begin{align*}
d_{TV}(\nu_0^{\otimes d},\nu_1^{\otimes d}) \leq \sqrt{\tilde c_vu},
\end{align*}
for $u \leq \tilde c_v^{-1}$, i.e.~for $u$ smaller than a constant that depends only on $v$.
\end{lemma}
Let $u > 0$ and $v>0$. We first apply Lemma~\ref{lem:lower} to $\pi$ sorted in decreasing order, which leads to the definition of a vector denoted as $\bar \varepsilon^*$. Then we apply Lemma~\ref{lem:lower} to $\pi$ restricted to $\mathcal A$ and sorted in decreasing order, which defines a vector denoted as $\tilde \varepsilon^*$.

Since
$$\sum_{i: |S_\pi(\floor{\log_2(k)} + i)| \leq  a \sqrt{\log((|i|+1)/\gamma)}} \sum_{j\in S_\pi(\floor{\log_2(k)} + i)} \pi_j^2\exp(-2(1+v)k\pi_j) \leq 4a \frac{\sqrt{\log(1/\gamma)}}{k^2},$$
we have by definition of $\varepsilon^*$ in Lemma~\ref{lem:lower} that if $\|\pi^2\exp(-2(1+v)k\pi\|_2^2 \geq 8aM \frac{\sqrt{\log(1/\gamma)}}{k^2}$ then $\bar \varepsilon^*/(8M) \leq \tilde \varepsilon_i^*$, where we assume that $M \geq 4(16 \log{2/\delta})^2$, and that $a>2$.

From now on we take $\varepsilon^* = \bar \varepsilon^*/(8M)$. Since $\bar \varepsilon^*/(8M) \leq \tilde \varepsilon_i^*$, we can apply Lemma~\ref{lem:TV01marginal} to the restriction of $\pi$ and $\varepsilon^*$ to $\mathcal A$, so we have that there exists $\tilde c_v>0$ such that for $u \leq \tilde c_v^{-1}$
$$d_{TV}(\Lambda_0,\Lambda_1)\leq \sqrt{\tilde c_v u}.$$

\paragraph*{Total variation distance between Poisson distributions, and Multinomial distributions.} 
We define the following distribution:
$M_{\tilde q,\tilde p|k_1,k_2} = \mathcal M(k_1, \tilde q) \otimes \mathcal M(k_2, \tilde p)$. Let $\mathcal D=\mathcal P(k\sum_i q_i)$ and $\mathcal D' = \mathcal P(k\sum_i p_i)$.

By definition of the total variation distance, we have
\begin{align*}
&d_{TV}\cro{\E_{\tilde q,(\hat k_1, \hat k_2)\sim \mathcal D^{\otimes 2}}(M_{\tilde q,\tilde q|\hat k_1, \hat k_2}),\E_{\tilde q,\tilde p, (\hat k_1, \hat k_2) \sim \mathcal D \otimes \mathcal D'}(M_{\tilde q,\tilde p|\hat k_1,\hat k_2})} \\
&\leq \frac{\begin{aligned}
		d_{TV}\Bigg[\E_{\tilde q,(\hat k_1,\hat k_2) \sim \mathcal D^{\otimes 2}|(\hat k_1,\hat k_2) \in [k/2, 3k/2]^2}&(M_{\tilde q,\tilde q|\hat k_1,\hat k_2}),\\
		&\E_{\tilde q,\tilde p,(\hat k_1, \hat k_2) \sim \mathcal D \otimes \mathcal D'|(\hat k_1, \hat k_2) \in [k/2, 3k/2]^2}(M_{\tilde q,\tilde p|\hat k_1, \hat k_2})\Bigg]
	\end{aligned}}{\min_{\mathcal D'' \in \{\mathcal D^{\otimes 2}, \mathcal D \otimes \mathcal D'\}}\mathbb P_{\tilde q,\tilde p,(\hat k_1,\hat k_2) \sim \mathcal D''}(\hat k \in [k/2, 3k/2])}\\
&\quad+ \max_{\mathcal D'' \in \{\mathcal D^{\otimes 2}, \mathcal D\otimes \mathcal D'\}}\mathbb P_{\tilde q,\tilde p,(\hat k_1, \hat k_2) \sim \mathcal D''}((\hat k_1, \hat k_2) \not\in [k/2, 3k/2]^2).
\end{align*}
This implies for $M \geq 4(32 \log(1/\delta))^2$ 
\begin{align*}
&d_{TV}\cro{\E_{\tilde q,(\hat k_1, \hat k_2)\sim \mathcal D^{\otimes 2}}(M_{\tilde q,\tilde q|\hat k_1, \hat k_2}),\E_{\tilde q,\tilde p,(\hat k_1, \hat k_2) \sim \mathcal D \otimes \mathcal D'}(M_{\tilde q,\tilde p|\hat k_1, \hat k_2})} \\
&\leq \frac{\begin{aligned}
		d_{TV}\Bigg[\E_{\tilde q,(\hat k_1, \hat k_2)\sim\mathcal D^{\otimes 2}|(\hat k_1, \hat k_2) \in [k/2, 3k/2]^2}&(M_{\tilde q,\tilde q|\hat k_1, \hat k_2}),\\
		&\E_{\tilde q,\tilde p,(\hat k_1, \hat k_2) \sim \mathcal D \otimes \mathcal D'|(\hat k_1, \hat k_2) \in [k/2, 3k/2]^2}(M_{\tilde q,\tilde p|\hat k_1, \hat k_2})\Bigg]
	\end{aligned}}{1 - 2\exp(-k/48) - \delta}\\
&\quad+ 2\exp(-k/48)+\delta\\
&\leq d_{TV}\Bigg[\E_{\tilde q, (\hat k_1, \hat k_2) \sim \mathcal D^{\otimes 2}|(\hat k_1, \hat k_2) \in [k/2, 3k/2]^2}(M_{\tilde q,\tilde q|\hat k_1, \hat k_2}),\\
	&\quad\quad\quad\quad\quad\quad\quad\E_{\tilde q,\tilde p, (\hat k_1, \hat k_2) \sim \mathcal D \otimes \mathcal D'|(\hat k_1, \hat k_2) \in [k/2, 3k/2]^2}(M_{\tilde q,\tilde p|\hat k_1, \hat k_2})\Bigg] + 8\exp(-k/48) +4\delta,
\end{align*}
for $k \geq 2^8$ and $\delta\leq 1/8$.
Now we have by Theorem~\ref{th:multinPoisson},
$$
d_{TV}\cro{\E_{\tilde q,(\hat k_1, \hat k_2)\sim\mathcal D^{\otimes 2}}(M_{\tilde q,\tilde q|\hat k_1, \hat k_2}),\E_{\tilde q,\tilde p,(\hat k_1, \hat k_2)\sim\mathcal D \otimes \mathcal D'}(M_{\tilde q,\tilde p|\hat k_1, \hat k_2})} = d_{TV}(\Lambda_0, \Lambda_1).
$$
And so when combined with the previously displayed equation
\begin{align}
\label{eq:TVcool}
d_{TV}( \Lambda_0, \Lambda_1) 
&\geq -8\exp(-k/48) - 4\delta \nonumber\\
&\quad+ d_{TV}\Bigg[\E_{\tilde q, (\hat k_1, \hat k_2) \sim \mathcal D^{\otimes 2}|(\hat k_1 \hat k_2) \in [k/2, 3k/2]^2}(M_{\tilde q,\tilde q|\hat k_1, \hat k_2}),\nonumber\\
&\quad\quad\quad\quad\quad\quad\quad\quad\quad
\E_{\tilde q,\tilde p,(\hat k_1, \hat k_2) \sim \mathcal D \otimes \mathcal D'|(\hat k_1, \hat k_2) \in [k/2, 3k/2]^2}(M_{\tilde q,\tilde p|\hat k_1, \hat k_2})\Bigg].
\end{align}

\paragraph*{Conclusion.}

Consider an event of probability larger than $1 -\delta$ with respect to $q,\xi$ such that for some $\bar \rho>0$ we have
$$\sum_{i} |\tilde q_i - \tilde p_i| = \sum_{i \in {\mathcal A}} q_i |\xi_{i}| = \sum_{i \in {\mathcal A'}} q_i |\xi_{i}| \geq \bar \rho.$$
So we have by Equation~\eqref{eq:TVcool} and Lemma~\ref{lem:lam0lam1} that under the condition that $k \geq 2^8 \vee M^2$, $\delta \leq 1/8$, $M \geq 4[1 \vee (32 \log(1/\delta))^2]$, and $\|\pi^2\exp(-2(1+v)k\pi\|_2^2 \geq 8aM \frac{\sqrt{\log(1/\gamma)}}{k^2}$, then for any test $\varphi$
\begin{align*}
R(H^{(\text{Clo})}_0(\pi),H^{(\text{Clo})}_1(\pi, \bar\rho),\varphi; \bar\rho, \lfloor k/2\rfloor) &\geq 1 -  8\exp\left(-\frac{k}{48}\right) - 7\delta - c\gamma - \sqrt{\tilde c_v u}.
\end{align*}
We now present the following lemma.
\begin{lemma}\label{lem:barrho}
It holds with probability larger than $1-\delta$ that
\begin{align*}
\sum_{i \in {\mathcal A}} |\xi_i| q_i &\geq \frac{1}{8M^2}\Bigg[\Big[\sum_{i\geq I_{v,\pi}} \frac{\sqrt{u}\pi_i}{\sqrt{2}} \Big]\vee \frac{\sqrt{u} (I_{v,\pi}-J)}{\sqrt{2I_{v,\pi}}}\sqrt{\frac{\sqrt{\sum_i \pi_i^2\exp(-2(1+v)k\pi_i)}}{k} }\Bigg] \land \sqrt{\frac{u}{8}}\Bigg]  \\
&- \frac{1}{\sqrt{kM\delta}} - 8\frac{a(1+ \log(1/\gamma))}{k}.
\end{align*}
\end{lemma}
This implies that we can take $\bar \rho$ as in the lemma, which concludes the proof.

\begin{proof}[Proof of Lemma~\ref{lem:lam0lam1}]

\underline{Study of $|S_{\tilde q}( \floor{\log_2(k)} + i)\cap \mathcal A|$ and $|S_{\tilde p}(\floor{\log_2(k)} + i)\cap \mathcal A|$.}	For any $i$ such that $|S_\pi(\floor{\log_2(k)} + i)| > a \sqrt{\log((|i|+1)/\gamma)}$, we have $|S_{\tilde q}( \floor{\log_2(k)} + i)\cap \mathcal A| \sim \text{Bin}(|{\mathcal A}|, \floor{|S_\pi(\floor{\log_2(k)} + i)| / 2} / |{\mathcal A}| )$, by definition of the distribution of ${\tilde q}$ on $\mathcal A$. By Hoeffding's inequality, we have for any $\varepsilon > 0$, that with probability larger than $1-2\exp(-2 \varepsilon^2 n)$ with respect to the distribution of $\tilde q$
	$$
	|S_{{\tilde q}}(\floor{\log_2(k)} + i)\cap \mathcal A| - \floor{|S_\pi( \floor{\log_2(k)} + i)\cap \mathcal A| / 2}| \leq \varepsilon n.
	$$
	So with probability  larger than $1 - 2\exp(-2 |S_\pi( \floor{\log_2(k)} + i)|^2/16)$ according to the distribution of $\tilde q$
	$$
	||S_{\tilde q}( \floor{\log_2(k)} + i)\cap \mathcal A| - \floor{|S_\pi( \floor{\log_2(k)} + i)\cap \mathcal A| / 2}| \leq |S_\pi(\floor{\log_2(k)} + i)\cap \mathcal A|/4.
	$$
	So for any $i$ such that $|S_\pi(\floor{\log_2(k)} + i)| > a\sqrt{\log((|i|+1)/\gamma)}$, we have with probability larger than $1 - 2\exp(- a^2 \log((|i|+1)/\gamma))$ with respect to the distribution of $\tilde q$ that
	$$
	||S_{\tilde q}(  \floor{\log_2(k)} + i)\cap \mathcal A| - \floor{|S_\pi( \floor{\log_2(k)} + i)\cap \mathcal A| / 2}| \leq |S_\pi( \floor{\log_2(k)} + i)\cap \mathcal A|/4.
	$$		
So whenever $a>1$, there exists a constant $c_a>0$ that depends only on $a$ and such that we have with probability larger than $1 - 2\sum_{i\in \mathbb Z} \exp(- a^2 \log((|i|+1)/\gamma)) \geq 1 - c_a \gamma$ with respect to the distribution of $\tilde p$ that for all $i$ such that $|S_\pi(  \floor{\log_2(k)} + i)| > a \sqrt{\log((|i|+1)/\gamma)}$ at the same time,
\begin{equation}
\label{eq:boundSq}
\frac34 |S_\pi( \floor{\log_2(k)} + i)\cap \mathcal A| \leq |S_{\tilde q}(  \floor{\log_2(k)} + i)\cap \mathcal A| \leq \frac54 |S_\pi( \floor{\log_2(k)} + i)\cap \mathcal A|.
\end{equation}
Now since $\varepsilon_i^*\in [0,1/2]$ we know that for any $i \in \mathcal A$ we have
$$\frac{1}{2} \tilde q_i\leq \tilde p_i \leq \frac{3}{2} \tilde q_i.$$
And so we also know that with probability larger than $1 - c_a \gamma$ with respect to the distribution of $\tilde p$ for all $i\in \mathbb Z$ such that $|S_\pi( \floor{\log_2(k)} + i)| > a \sqrt{\log((|i|+1)/\gamma)}$ at the same time,
\begin{equation}
\label{eq:boundSp2}
\frac34 |S_\pi( \floor{\log_2(k)} + i)\cap \mathcal A| \leq \sum_{j=-1}^{j+1}|S_{\tilde q}(\floor{\log_2(k)} + j)\cap \mathcal A| \leq \frac54 \sum_{j=-2}^{j+2}|S_\pi(\floor{\log_2(k)} + i) \cap \mathcal A|.
\end{equation}
\underline{Study of the rescaled coefficients outside $\mathcal A$.}
We define the following events
\begin{equation}\label{eq:eventH}
H_p=\{|\sum_{l\in {\mathcal A}} (p_l - \pi_l)| \leq 16 M^{-1} \log(2/\delta)\},~~~~~~~H_q=\{|\sum_{l\in {\mathcal A}} (q_l - \pi_l)| \leq 16 M^{-1} \log(2/\delta)\}.
\end{equation}
We remind that for $i \in \mathcal A$ we have $\tilde p_i = p_i$ and $\tilde q_i = q_i$. So the events $H_p$ and $H_q$ are very informative with respect to $\tilde p, \tilde q$. Now, $\max_{j \in {\mathcal A}} \pi_j \leq 1$ and since $|S_{\tilde q}( i)\cap \mathcal A|/|S_{\tilde q}( i)| \leq 1/M$ for any $i \in \mathbb Z$ and $\sqrt k \geq M$. So by Bernstein's inequality, we have with probability larger than $1-\delta$ with respect to $p$ and $\tilde p$ that $H_p$ holds and with probability larger than $1-\delta$ with respect to $q$ and $\tilde q$ that $H_q$ holds.

So if $\sqrt k \geq M \geq 4 (16\log(2/\delta))^2$, we have with probability larger than $1-\delta$ with respect to $\tilde q$ that for any $j \not\in \mathcal A$, 
$$\frac12 \pi_j \leq \tilde q_{j} \leq \frac32 \pi_j,$$
and also with probability larger than $1-\delta$ with respect to $\tilde p$ that for any $j \not\in \mathcal A$, 
$$\frac12 \pi_j \leq \tilde p_{j} \leq \frac32 \pi_j.$$
And so finally we have with probability larger than $1-\delta$ with respect to $\tilde q$ that for any $i$,
$$|S_\pi(i) \cap \mathcal A^C| \leq \sum_{j= i-1 }^{i+1}|S_{\tilde q}(j)\cap \mathcal A^C| \leq \sum_{j= i-2 }^{i+2} |S_\pi(i)\cap \mathcal A^C|.$$
Similarly we have with probability larger than $1-\delta$ with respect to $\tilde p$ that for any $i$,
$$|S_\pi(i) \cap \mathcal A^C| \leq \sum_{j= i-1 }^{i+1}|S_{\tilde p}(j)\cap \mathcal A^C| \leq \sum_{j= i-2 }^{i+2} |S_\pi(i)\cap \mathcal A^C|.$$
\underline{Conclusion.} Combining both studies on $\mathcal A$ and $\mathcal A^C$, we get that if $M \geq 4 (16\log{2/\delta})^2$, we have with probability larger than $1-\delta - c_a \gamma$ with respect to $\tilde q$ that $\tilde q \in \mathbf P_\pi$, and with probability larger than $1-\delta - c_a \gamma$ with respect to $\tilde p$ that $\tilde p \in \mathbf P_\pi$.

\end{proof}

\begin{proof}[Proof of Lemma~\ref{lem:TV01marginal}]
Define the discrete uniform distribution $\mathcal U_\lambda$ such that $\mathcal U_\lambda(\{\lambda_i\}) = 1/d$. We will now work on the definition of appropriate measures for $(p,q)$. Let $\theta\sim \mathcal U_\lambda$ and $\xi$ taking value $\varepsilon_i^*$ when $\theta$ takes value $\lambda_i$.
We reparametrize $\xi$ by $\lambda$, and we set 
$$\xi_\theta = \frac{1}{\Big|\{i : k\pi_i = \theta\}\Big|} \sum_{\{i : k\pi_i = \theta\}} \varepsilon_i^*,$$
with the convention $0/0 = 0$. Note that by definition of $\varepsilon^*$, we have from Lemma~\ref{lem:lower}
\begin{itemize}
\item $\xi_\theta \in [0,1]$ and $\xi_\theta = 0$ for any $\theta \geq 1$.
\item $\xi_\theta \theta \leq \sqrt{u}\Big[k\sqrt{C/I_{v,\pi}}\land 1\Big]$.
\item By definition of $\mathcal U_\lambda$ and Lemma~\ref{lem:lower}
\begin{align}\label{eq:defmom}
\int \theta^2 \xi_\theta^2 e^{-2\theta} d\mathcal U_\lambda(\theta) = \frac{k^2}{d}\sum_{i} \pi_i^2 \varepsilon^{*2}_i e^{-2k\pi_i} &\leq \frac{k^2}{d}  u\frac{\sqrt{\sum_i \pi_i^2e^{-2(1+v)k\pi_i}}}{k} \nonumber\\
&= u\sqrt{\frac{\int \theta^2 e^{-2(1+v)\theta} d\mathcal U_\lambda(\theta)}{d}}.
\end{align}
\end{itemize}

\paragraph{Bound on the total variation.}

Let us dominate the total variation distance with the chi-squared distance $\chi_2$.

For two distributions $\tilde \nu_1, \tilde \nu_0$ such that $\tilde \nu_1$ is absolutely continuous with respect to $\tilde \nu_0$, then
\begin{align*}
d_{TV}(\tilde \nu_0,\tilde \nu_1) &= \int \left|\frac{d\tilde \nu_1}{d\tilde \nu_0} - 1 \right| d\tilde \nu_0 = \E_{\tilde \nu_0}\left[\left|\frac{d\tilde \nu_1}{d\tilde \nu_0}-1\right|\right] \\
&\leq \left(\E_{\tilde \nu_0}\left[\left(\frac{d\tilde \nu_1}{d\tilde \nu_0}\right)^2\right]-1\right)^{1/2} = \sqrt{\chi_2(\tilde \nu_0,\tilde \nu_1)}.
\end{align*}
By the tensorization property of the chi-squared distance and by application of the inequality above to $\nu_0^{\otimes d},\nu_1^{\otimes d}$, we have
\begin{equation}\label{eq:TV}
d_{TV}(\nu_0^{\otimes d},\nu_1^{\otimes d}) \leq \sqrt{\chi_2(\nu_0^{\otimes d},\nu_1^{\otimes d})} =  \sqrt{(1 + \chi_2(\nu_0,\nu_1))^d - 1}.
\end{equation}
Now, we have by the law of total probability for any $m, m' \geq 0$
$$
\nu_0(m,m') = \int\frac{e^{-2\theta} \theta^{m+m'}}{m!m'!} d\mathcal U_\lambda(\theta),
$$
and
$$
\nu_1(m,m') =  \int \frac12 \frac{e^{-2\theta} \theta^{m+m'}}{m!m'!} (e^{\xi_\theta \theta}(1-\xi_{\theta})^{m'} + e^{-\xi_\theta \theta}(1+\xi_{\theta})^{m'})d\mathcal U_\lambda(\theta).
$$
So \begin{align}
&\chi_2(\nu_0,\nu_1) \nonumber\\
&= \sum_{m,m'} \frac{(\int \theta^{m+m'}e^{-2\theta}[-e^{\xi_\theta \theta}(1-\xi_\theta)^{m'}/2 - e^{-\xi_\theta \theta} (1+\xi_\theta)^{m'}/2 + 1] d\mathcal U_\lambda(\theta))^2}{m!m'! \int \theta^{m+m'} e^{-2\theta}d\mathcal U_\lambda(\theta)} \nonumber\\
&= \sum_{m,m'} \frac{\int\int (\theta\theta')^{m+m'}e^{-2(\theta+\theta')} D_\theta(m) D_{\theta'}(m) d\mathcal U_\lambda(\theta)d\mathcal U_\lambda(\theta')}{m!m'!\int \theta^{m+m'}e^{-2\theta}d\mathcal U_\lambda(\theta)},\label{eq:chi2}
\end{align}
where 
$$D_\theta(m)=-\frac{e^{\xi_\theta \theta}(1-\xi_\theta)^{m}}{2} - \frac{e^{-\xi_\theta \theta} (1+\xi_\theta)^{m}}{2} + 1.$$
We will analyse the terms of this sum depending on the value of $m+m'$.
\bigskip\\

\paragraph{Analysis of the terms in Equation~\eqref{eq:chi2}.}\label{ss:chi2}

\noindent
{\bf Term for $m+m' = 0$.} We have
$$
D_\theta(0) = -\cosh(\xi_\theta \theta) + 1
\quad\quad\quad
\text{and} 
\quad\quad\quad
D_\theta(0) D_{\theta'}(0) \leq (\theta \theta' \xi_\theta \xi_{\theta'})^2,
$$
since $\xi_\theta\theta \leq 1$.

And so
\begin{align*}
\frac{\int\int e^{-2(\theta+\theta')} D_\theta(0) D_{\theta'}(0) d\mathcal U_\lambda(\theta)d\mathcal U_\lambda(\theta')}{\int e^{-2\theta}d\mathcal U_\lambda(\theta)} &\leq \frac{\Big(\int e^{-2\theta} (\theta\xi_\theta)^2 d\mathcal U_\lambda(\theta)\Big)^2}{\int e^{-2\theta}d\mathcal U_\lambda(\theta)}\\
&\leq  u \frac{\int \theta^2 e^{-2(1+v)\theta} d\mathcal U_\lambda(\theta)}{d\int e^{-2\theta}d\mathcal U_\lambda(\theta)} \leq \frac{c_vu}{d},
\end{align*}
where we obtained the second inequality by Equation~\eqref{eq:defmom} and for $c_v<+\infty$ that depends only on $v>0$ and such that
\begin{align}\label{eq:cv}
c_v = \sup_{\theta >0}\Big[e^{-2v \theta} (1\lor \theta^2)\Big].
\end{align}

\noindent
{\bf Term for $m+m' = 1$.} We have then
$$
D_\theta(1) = -\cosh(\xi_\theta \theta) + 1 + \xi_\theta\sinh(\theta \xi_\theta)
$$
and so since $\xi_\theta \in [0,1]$ and $\theta \xi_\theta \in [0,\xi_\theta]$, we have
$$
D_\theta(1) D_{\theta'}(1) \leq \theta \theta' (\xi_\theta \xi_{\theta'})^2.
$$

\noindent
So the term for $m+m'=1$ can be bounded as
\begin{align*}
&\frac{\int\int \theta\theta' e^{-2(\theta+\theta')} (D_\theta(0) D_{\theta'}(0) + D_\theta(1) D_{\theta'}(1)) d\mathcal U_\lambda(\theta)d\mathcal U_\lambda(\theta')}{\int \theta e^{-2\theta}d\mathcal U_\lambda(\theta)} \\
&\quad\quad\quad\quad\quad\quad\quad\quad\quad\quad\quad\quad\quad\quad\quad\quad
\leq  \frac{1}{\int \theta e^{-2\theta} d\mathcal U_\lambda(\theta)} \Big(\int e^{-2\theta} 2(\theta\xi_\theta)^2d\mathcal U_\lambda(\theta)\Big)^2\\
&\quad\quad\quad\quad\quad\quad\quad\quad\quad\quad\quad\quad\quad\quad\quad\quad
\leq  4u \frac{\int \theta^2 e^{-2(1+v)\theta}  d\mathcal U_\lambda(\theta)}{d\int \theta e^{-2\theta}d\mathcal U_\lambda(\theta)} \leq 4c_v\frac{ u}{d},
\end{align*}
where we obtained the second inequality by Equation~\eqref{eq:defmom} and the last by Definition of $c_v$ in Equation~\eqref{eq:cv}.\bigskip\\

\noindent
{\bf Term for $m+m' = 2$.} We have then
$$
D_\theta(2) = -\cosh(\xi_\theta \theta) + 1 + 2\xi_\theta\sinh(\theta \xi_\theta) - \xi_\theta^2\cosh(\theta \xi_\theta)
$$
and again since $\xi_\theta \in [0,1]$ and $\theta \xi_\theta \in [0,\xi_\theta]$, we have
$$
D_\theta(2) D_{\theta'}(2) = 4(\xi_\theta \xi_{\theta'})^2.
$$
So the term for $m+m'=2$ can be bounded as
\begin{align*}
&\frac{\int\int (\theta\theta')^2 e^{-2(\theta+\theta')} (D_\theta(0) D_{\theta'}(0)/2 + D_\theta(1) D_{\theta'}(1) + D_\theta(2) D_{\theta'}(2)/2) d\mathcal U_\lambda(\theta)d\mathcal U_\lambda(\theta')}{\int \theta^2 e^{-2\theta}d\mathcal U_\lambda(\theta)} \hspace{2cm}\\
&\hspace{6.4cm}\leq  \frac{1}{\int \theta^2 e^{-2\theta} d\mathcal U_\lambda(\theta)} \Big(\int e^{-2\theta} 4(\theta\xi_\theta)^2d\mathcal U_\lambda(\theta)\Big)^2\\
&\hspace{6.4cm}\leq \frac{16u}{d} c_v,
\end{align*}
where we obtain the second inequality by Equation~\eqref{eq:defmom}.\bigskip\\

\noindent
{\bf Term for $m+m' \geq 3$.} 
We have
\begin{equation}\label{eq:dlam}
D_\theta(m) \leq 2^{m+2} \xi_\theta^2.
\end{equation}
\textit{Subcase 1: $m+m' = 3$.} We have by Equation~\eqref{eq:dlam}
\begin{align*}
\frac{\int \int (\theta\theta')^{3} e^{-2\theta}D_\theta(m)D_{\theta'}(m)d\mathcal U_\lambda(\theta)d\mathcal U_\lambda(\theta')}{m!m'!\int \theta^{3} e^{-2\theta}d\mathcal U_\lambda(\theta)} &\leq \frac{\Big(\int e^{-2\theta} \theta^{3} (2^{m+2}\xi_\theta^2)d\mathcal U_\lambda(\theta)\Big)^2}{m!m'!\int \theta^{3}e^{-2\theta} d\mathcal U_\lambda(\theta)}\\
&\leq 2^{2m+4} \frac{\Big(\int e^{-2\theta} \theta (\theta \xi_\theta)^2d\mathcal U_\lambda(\theta)\Big)^2}{\int \theta^{3} e^{-2\theta}d\mathcal U_\lambda(\theta)} \\
&=  2^{2m+4} \frac{k^3 \Big[\sum_{i \geq J} \pi_i (\pi_i \varepsilon_i^{*})^2\Big]^2}{d \sum_i \pi_i^{3} e^{-2\pi_i k}}.
\end{align*}
Finally, by definition of $\xi_\theta$ and $\varepsilon_i^*$ and since in any case $\varepsilon_i^* \pi_i \leq  \sqrt{uC_\pi/I_{v,\pi}}$ (see Lemma~\ref{lem:lower}), we have
$$
\frac{\int \int (\theta\theta')^{3} e^{-2\theta}D_\theta(m)D_{\theta'}(m)d\mathcal U_\lambda(\theta)d\mathcal U_\lambda(\theta')}{m!m'!\int \theta^{3} e^{-2\theta}d\mathcal U_\lambda(\theta)} \leq 2^{2m+4} \frac{k^3 \Big[\sum_{i \geq J } \pi_i \frac{uC_\pi}{2I_{v,\pi}}\Big]^2}{d \sum_i \pi_i^{3} e^{-2\pi_i k}}.
$$
This implies, since $\sum_{J \leq i < I_{v,\pi} } \pi_i  \geq \sum_{I_{v,\pi} \leq i } \pi_i$ in the definition of $I_{v,\pi}$ (see Lemma~\ref{lem:lower}),
\begin{align*}
\frac{\int \int (\theta\theta')^{3} e^{-2\theta}D_\theta(m)D_{\theta'}(m)d\mathcal U_\lambda(\theta)d\mathcal U_\lambda(\theta')}{m!m'!\int \theta^{3} e^{-2\theta}d\mathcal U_\lambda(\theta)} &\leq  2^{2m+2} \frac{u^2C_\pi^2}{I_{v,\pi}^2} \frac{k^3 \Big[\sum_{J \leq i < I_{v,\pi} } \pi_i\Big]^2}{n \sum_i \pi_i^{3} e^{-2\pi_i k}}\\
&\leq  2^{2m+2} \frac{u^2C_\pi^2}{I_{v,\pi}} \frac{k^3 \Big[\sum_{J \leq i < I_{v,\pi} } \pi_i^2\Big]}{n \sum_i \pi_i^{3} e^{-2\pi_i k}},
\end{align*}
by Cauchy-Schwarz inequality. Then
\begin{align*}
&\frac{\int \int (\theta\theta')^{3} e^{-2\theta}D_\theta(m)D_{\theta'}(m)d\mathcal U_\lambda(\theta)d\mathcal U_\lambda(\theta')}{m!m'!\int \theta^{3} e^{-2\theta}d\mathcal U_\lambda(\theta)} \\
&\quad\quad\quad\quad\quad\quad\quad\quad\quad\quad
\leq  2^{2m+2} \frac{u^2}{I_{v,\pi}}  \frac{k \Big[\sum_i \pi_i^2 \exp(-2(1+v)k\pi_i)\Big]\Big[\sum_{J \leq i < I_{v,\pi} } \pi_i^2\Big]}{n \sum_i \pi_i^{3} e^{-2\pi_i k}},
\end{align*}
by Definition of $C_\pi$ in Equation~\eqref{eq:defC}. In particular,
\begin{align*}
\sum_i \pi_i^2 \exp(-2(1+v)k\pi_i)  &= \sum_{i < I_{v,\pi}} \pi_i^2 \exp(-2(1+v)k\pi_i) + \sum_{I_{v,\pi} \leq i} \pi_i^2 \exp(-2(1+v)k\pi_i)  \\
&\leq 2 e^{2(1+v)} \sum_{i < I_{v,\pi}} \pi_i^2 \exp(-2(1+v)k\pi_i),
\end{align*}
 since $\sum_{I_{v,\pi} \leq i } \pi_i \leq \sum_{J \leq i < I_{v,\pi} } \pi_i$ and for all $i \geq J$ we have $k\pi_i \leq 1$. So, once we plug the last inequality in, we obtain: 
\begin{align*}
&\frac{\int \int (\theta\theta')^{3} e^{-2\theta}D_\theta(m)D_{\theta'}(m)d\mathcal U_\lambda(\theta)d\mathcal U_\lambda(\theta')}{m!m'!\int \theta^{3} e^{-2\theta}d\mathcal U_\lambda(\theta)} \\
&\quad\quad\quad\quad\quad\quad\quad
\leq  2^{2m+3} e^{2(1+v)}\frac{u^2}{I_{v,\pi}} \frac{k \Big[\sum_{ i < I_{v,\pi} } \pi_i^2 \exp(-2(1+v)k\pi_i)\Big] \Big[\sum_{J \leq i < I_{v,\pi} } \pi_i^2\Big]}{d \sum_i \pi_i^{3} e^{-2\pi_i k}}\\
&\quad\quad\quad\quad\quad\quad\quad
\leq  2^{2m+3} e^{2(1+v)}\frac{u^2}{I_{v,\pi}} \frac{k \Big[\sum_{ i < I_{v,\pi} } \pi_i^2 \exp(-2(1+v)k\pi_i)\Big] \Big[\sum_{J \leq i < I_{v,\pi} } \pi_i^2\Big]}{d \sum_{i\leq I_{v,\pi}} \pi_i^{3} e^{-2\pi_i k}}\\
&\quad\quad\quad\quad\quad\quad\quad
\leq  2^{2m+3} e^{2(1+v)} u^2 \frac{k \Big[\sum_{ i \leq I_{v,\pi} } \pi_i^4 \exp(-2(1+v)k\pi_i)\Big] }{d \sum_i \pi_i^{3} e^{-2\pi_i k}},
\end{align*}
because for any $a_1 \geq \ldots \geq a_{I_{v,\pi}}\geq 0$, $b_1 \geq \ldots \geq b_{I_{v,\pi}} \geq 0$, we have $\sum a_i \sum b_j \leq I_{v,\pi} \sum a_i b_i$.
 Then $k\pi_i e^{-2vk\pi_i} \leq c_v$ by Equation~\eqref{eq:cv} and for any $i$, this implies
\begin{align*}
\frac{\int \int (\theta\theta')^{3} e^{-2\theta}D_\theta(m)D_{\theta'}(m)d\mathcal U_\lambda(\theta)d\mathcal U_\lambda(\theta')}{m!m'!\int \theta^{3} e^{-2\theta}d\mathcal U_\lambda(\theta)} 
&\leq  2^{2m+3}
e^{2(1+v)} c_v \frac{u^2}{d} \leq 2^{9} e^{2(1+v)}c_v  \frac{u^2}{d}.
\end{align*}

\noindent
\textit{Subcase 2: $m+m' \geq 4$.} We have by Equation~\eqref{eq:dlam}
\begin{align*}
&\frac{\int \int (\theta\theta')^{m+m'} e^{-2(\theta+\theta')}D_\theta(m)D_{\theta'}(m)d\mathcal U_\lambda(\theta)d\mathcal U_\lambda(\theta')}{m!m'!\int \theta^{m+m'} e^{-2\theta}d\mathcal U_\lambda(\theta)} \\
&\hspace{6cm}\leq \frac{\Big(\int e^{-2\theta} \theta^{m+m'} (2^{m+2}\xi_\theta^2)d\mathcal U_\lambda(\theta)\Big)^2}{m!m'!\int \theta^{m+m'}e^{-2\theta} d\mathcal U_\lambda(\theta)}\\
&\hspace{6cm}= \frac{2^{2m+4}}{m!m'!}\frac{\Big(\int e^{-2\theta} \theta^{m+m'} \xi_\theta^2d\mathcal U_\lambda(\theta)\Big)^2}{\int \theta^{m+m'}e^{-2\theta} d\mathcal U_\lambda(\theta)}\\
&\hspace{6cm}\leq \frac{2^{2m+4}}{m!m'!} \int e^{-2\theta} \theta^{m+m'} \xi_\theta^4 d\mathcal U_\lambda(\theta),
\end{align*}
where the last inequality comes by application of Cauchy-Schwarz inequality. And so since $\varepsilon_\theta = 0$ for any $\theta \geq 1$, we have
\begin{align*}
&\frac{\int \int (\theta\theta')^{m+m'} e^{-2(\theta+\theta')}D_\theta(m)D_{\theta'}(m)d\mathcal U_\lambda(\theta)d\mathcal U_\lambda(\theta')}{m!m'!\int \theta^{m+m'} e^{-2\theta} d\mathcal U_\lambda(\theta)} \\
&\quad\quad\quad\quad\quad\quad\quad\quad\quad\quad\quad\quad\quad\quad
\leq \frac{2^{2m+4}}{m!m'!} \int e^{-2\theta} \theta^{m+m'-4} (\theta\xi_\theta)^4 \mathbf 1\{\theta \leq 1\} d\mathcal U_\lambda(\theta).
\end{align*}
Then, since $m+m' \geq 4$,
\begin{align*}
&\frac{\int \int (\theta\theta')^{m+m'} e^{-2(\theta+\theta')}D_\theta(m)D_{\theta'}(m)d\mathcal U_\lambda(\theta)d\mathcal U_\lambda(\theta')}{m!m'!\int \theta^{m+m'} e^{-2\theta} d\mathcal U_\lambda(\theta)}\\
&\hspace{6cm} \leq \frac{2^{2m+4}}{m!m'!} \int  (\theta\xi_\theta)^4 \mathbf 1\{\theta \leq 1\} d\mathcal U_\lambda(\theta)\\
&\hspace{6cm}= \frac{2^{2m+4}}{m!m'!}\frac{k^4}{d} \Big[ \sum_{i} (\pi_i \varepsilon_i^*)^4  \Big]\\
&\hspace{6cm}\leq \frac{2^{2m+4} e^2}{m!m'!}\frac{k^4}{d} \frac{u^2C_\pi^2}{I_{v,\pi}},
\end{align*}
since, by definition of $\varepsilon_i^*$ in Lemma~\ref{lem:lower}, $\pi_i \varepsilon_i^* \leq \sqrt{uC_\pi/I_{v,\pi}}$ and $ \sum_{i} (\pi_i \varepsilon_i^*)^2  \leq uC_\pi e^2$ using the fact that $\varepsilon_i^* = 0$ for $\pi_i \geq 1/k$.
By Equation~\eqref{eq:boundCsurI}, this implies
\begin{align*}
\frac{\int \int (\theta\theta')^{m+m'} e^{-2(\theta+\theta')}D_\theta(m)D_{\theta'}(m)d\mathcal U_\lambda(\theta)d\mathcal U_\lambda(\theta')}{m!m'!\int \theta^{m+m'} e^{-2\theta}d\mathcal U_\lambda(\theta)} 
&\leq \frac{2^{2m+4} e^2}{m!m'!}\frac{2u^2}{d}.
\end{align*}

\paragraph{Conclusion on the distance between the two distributions.}

Now, we plug the bounds we found for each term back in Equation~\eqref{eq:chi2} and we obtain
\begin{align*}
 \chi_2(\nu_0,\nu_1)  
&\leq \sum_{m,m'} \frac{\int\int (\theta\theta')^{m+m'}e^{-2(\theta+\theta')} D_\theta(m) D_{\theta'}(m) d\mathcal U_\lambda(\theta)d\mathcal U_\lambda(\theta')}{m!m'!\int \theta^{m+m'}e^{-2\theta}d\mathcal U_\lambda(\theta)}\\
&\leq \frac{c_vu}{d} + 4\frac{c_vu}{d} + \frac{16u}{d} c_v+ 2^{9}e^{2(1+v)} \frac{u}{d}c_v+\sum_{m,m' : m+m' \geq 4} \frac{2^{2m+4} e^2}{m!m'!}\frac{2u^2}{d} \\
&\leq (5 + 16 + 2^{9}e^{2(1+v)})\frac{u}{d}  c_v +\sum_{m,m'} \frac{2^{2m+4} e^2}{m!m'!}\frac{2u}{d} \\
&= (5 + 16 + 2^{9}e^{2(1+v)})\frac{u}{d} c_v +\frac{2^{5}e^7u}{d}\\
&\leq (5c_v + 16c_v + 2^{9}e^{2(1+v)} c_v + 2^{5}e^7)\frac{u}{d} \leq \tilde c_{v} \frac{u}{d}, 
\end{align*}
for $u\leq 1$ and where $\tilde c_v$ is a constant that depends only on $v$ and $\tilde c_v$ is bounded away from 0 for $v>0$. And so by Equation~\eqref{eq:TV} we have
\begin{align*}
d_{TV}(\nu_0^{\otimes d},\nu_1^{\otimes d}) \leq \sqrt{\left(1 + \tilde c_{v} \frac{u}{d}\right)^d - 1} \leq \sqrt{\exp(\tilde c_v u) - 1} \leq \sqrt{\tilde c_vu},
\end{align*}
for $u \leq \tilde c_v^{-1}$, i.e.~for $u$ smaller than a constant that depends only on $v$.
\end{proof}

\begin{proof}[Proof of Lemma~\ref{lem:barrho}]

Note that
$$\mathbb E_{\tilde q,\tilde p}  \sum_{i } |\tilde p_i - \tilde q_i| = \mathbb E_{q,\xi} \Big[\sum_{i \in {\mathcal A}} |\xi_i| q_i \Big] = \sum_{i \in {\mathcal A}} \varepsilon^*_i \pi_i,$$
and
$$\mathbb V_{q} \Big[\sum_{i\in {\mathcal A}} |\xi_i| q_i \Big]= \sum_{i \in {\mathcal A}} \mathbb V_{q_1\sim U_{\pi,{\mathcal A}}} (|\xi_1| q_1) \leq \sum_{i \in {\mathcal A}} (\varepsilon^*_i \pi_i)^2.$$
Let $\alpha>0$. By Chebyshev's inequality, we know that with probability larger than $1-\alpha$ with respect to $(q,\xi)$,
$$\sum_{i\in {\mathcal A}} |\xi_i| q_i \geq \sum_{i\in {\mathcal A}} \varepsilon^*_i \pi_i - \sqrt{\frac{\sum_{i\in {\mathcal A}} \varepsilon_i^{*2}\pi_i^2}{\alpha}}.$$
Now, by definition of $\mathcal A$ and $(\varepsilon_i^*)_i$ we have 
$$\sum_{i\in {\mathcal A}} \varepsilon^*_i \pi_i \geq  \frac{1}{8 M^2}\sum_{i} \bar \varepsilon^*_i \pi_i - \sum_{i \in \mathbb N: |S_\pi(\floor{\log_2(k)} + i)| \leq  a \sqrt{\log(i/\gamma)}} \sum_{j\in S_\pi(\floor{\log_2(k)} + i)} \pi_j.$$
First note that
\begin{align*}
&\sum_{i \in \mathbb N: |S_\pi(\floor{\log_2(k)} + i)| \leq  a \sqrt{\log(i/\gamma)}} \sum_{j\in S_\pi(\floor{\log_2(k)} + i)} \pi_j \\
&\quad\quad\quad\quad\quad\quad\quad\quad\quad\quad\quad\quad\quad\quad\quad\quad\quad
\leq \sum_{i \in \mathbb N} a \sqrt{\log(i/\gamma)} \frac{2^{-i}}{k} \leq 8\frac{a(1+ \log(1/\gamma))}{k}.
\end{align*}
Also by application of Lemma~\ref{lem:lower}, where we associate $\bar \varepsilon^*_i$ to the ordered version of $\pi$, we have 
$$
\sum_{i} \bar \varepsilon^*_i \pi_i \geq \Bigg[\Big[\sum_{i\geq I_{v,\pi}} \frac{\sqrt{u}\pi_i}{\sqrt{2}} \Big]\vee \frac{\sqrt{u} (I_{v,\pi}-J)}{\sqrt{2I_{v,\pi}}}\sqrt{\frac{\sqrt{\sum_i \pi_i^2\exp(-2(1+v)k\pi_i)}}{k} }\Bigg] \land \sqrt{\frac{u}{8}}\Bigg].
$$
So we conclude that
\begin{align*}
\sum_{i\in {\mathcal A}} \varepsilon^*_i \pi_i \geq &~\frac{1}{8M^2}\Bigg[\Big[\sum_{i\geq I_{v,\pi}} \frac{\sqrt{u}\pi_i}{\sqrt{2}} \Big]\vee \frac{\sqrt{u} (I_{v,\pi}-J)}{\sqrt{2I_{v,\pi}}}\sqrt{\frac{\sqrt{\sum_i \pi_i^2\exp(-2(1+v)k\pi_i)}}{k} }\Bigg] \land  \sqrt{\frac{u}{8}}\Bigg] \\
&- 8\frac{a(1+ \log(1/\gamma))}{k}.
\end{align*}
So we have with probability larger than $1-\delta$,
\begin{align*}
\sum_{i \in {\mathcal A}} |\xi_i| q_i &\geq \frac{1}{8M^2}\Bigg[\Big[\sum_{i\geq I_{v,\pi}} \frac{\sqrt{u}\pi_i}{\sqrt{2}} \Big]\vee \frac{\sqrt{u} (I_{v,\pi}-J)}{\sqrt{2I_{v,\pi}}}\sqrt{\frac{\sqrt{\sum_i \pi_i^2\exp(-2(1+v)k\pi_i)}}{k} }\Bigg] \land \sqrt{\frac{u}{8}}\Bigg]  \\
&- \frac{1}{\sqrt{kM\delta}} - 8\frac{a(1+ \log(1/\gamma))}{k}.
\end{align*}

\end{proof}

\begin{proof}[Proof of Lemma~\ref{lem:lower}]
	We prove this lemma by defining suitable $\varepsilon_i^*$'s.
	
	\medskip
	
	\noindent
	{\bf Step 1: Proof that $\sqrt{C_\pi/I_{v,\pi}} \leq \frac{\sqrt{2}}{k}$.} We have
	\begin{align*}
	C_\pi^2 k^2 &\leq \sum_i  \pi_i^2 \exp(-2(1+v)k\pi_i) \\
	&= \sum_{i \leq I_{v,\pi}}  \pi_i^2 \exp(-2(1+v)k\pi_i) + \sum_{i \geq I_{v,\pi}}  \pi_i^2 \exp(-2(1+v)k\pi_i)\\
	&\leq  \frac{I_{v,\pi}}{k^2} + uC_\pi,
	\end{align*}
	as $\pi_i^2 \exp(-2k\pi_i) \leq \frac{1}{k^2}$. So we have that $C_\pi \leq \frac{2u}{k^2}$, or $C_\pi \leq  \frac{\sqrt{2I_{v,\pi}}}{k^2}$, and so in any case
	\begin{align*}
	C_\pi \leq  \frac{\sqrt{2I_{v,\pi}}}{k^2} \lor \frac{2u}{k^2},
	\end{align*}
	which implies
	\begin{align}\label{eq:boundCsurI}
	\frac{\sqrt{C_\pi}}{I_{v,\pi}^{1/4}} \leq  \frac{2^{1/4}}{k } \lor \frac{\sqrt{2u}}{k I_{v,\pi}^{1/4}} \leq \frac{\sqrt{2}}{k},
	\end{align}
	since $0<u <1$.
	
	\medskip
	\noindent
	{\bf Step 2: Definition of $\varepsilon_i^*$ for $i \geq I_{v,\pi}$ or $i < J$.} Take for all $i<J$ that $\varepsilon_i^* = 0$. Take for all other $i \geq I_{v,\pi}$ 
	$$\varepsilon_i^* = \sqrt{u/2}.$$
	We have for any $i \geq I_{v,\pi}$
	\begin{itemize}
		\item $\varepsilon^*_i \in [0,1]$, and $\varepsilon^*_i \pi_i \leq \sqrt{u}\Big[(1/k) \land \sqrt{C_\pi/(2I_{v,\pi})}\Big]$, since by definition of $I_{v,\pi}$ we know that $\pi_i \leq (1/k)\land \sqrt{C_\pi/I_{v,\pi}}$ if $i \geq I_{v,\pi}$
		\item by definition of $I_{v,\pi}$ we have
		$$\sum_{i\geq I_{v,\pi}} \pi_i^2 \varepsilon^{*2}_i \exp(-2k\pi_i) \leq  \frac{uC_\pi}{2}$$
		\item and also 
		\begin{align}\label{eq:sat1}
		\sum_{i\geq I_{v,\pi}} \varepsilon_i^* \pi_i = \sqrt{\frac{u}{2}}\sum_{i\geq I_{v,\pi}} \pi_i.
		\end{align}
	\end{itemize}

	\noindent
	{\bf Step 3: Definition of $\varepsilon_i^*$ for $i < I_{v,\pi}$ in three different cases.} If $I_{v,\pi} \leq J$, the $\varepsilon_i^*$ are already defined for all $i \geq J$, and by definition of $\varepsilon_i^*$,
	$$\sum_i \varepsilon_i^* \pi_i \geq \sum_{i\geq J} \frac{\sqrt{u}\pi_i}{\sqrt{2}} =\Big[\sum_{i\geq J} \frac{\sqrt{u}\pi_i}{\sqrt{2}} \Big]\lor\Big[\frac{(I_{v,\pi}-J)\sqrt{u C_\pi}}{\sqrt{2I_{v,\pi}}}\Big]$$
	This concludes the proof in that case. We assume from now on that $I_{v,\pi}>J$, then by definition of $I_{v,\pi}$, at least one of the constraints in Equation~\eqref{eq:defI} must be saturated.
	
	\medskip
	
	\noindent
	{\bf Case 1: third constraint saturated but not the first one: $\sum_{i \geq I_{v,\pi}-1} \pi_i > \sum_{J \leq i < I_{v,\pi}-1} \pi_i$ and $\pi_{I_{v,\pi}-1} \leq \sqrt{C_\pi/I_{v,\pi}} \land (1/k)$.} We set $\varepsilon_{I_{v,\pi}-1}^* = \sqrt{\frac{u}{2}}$ and for any $i < I_{v,\pi}-1$, we set $\varepsilon_i^* = 0$. Note that $\varepsilon_{I_{v,\pi}-1}^* \leq 1$ and $\varepsilon_{I_{v,\pi}-1}^*\pi_{I_{v,\pi}-1} \leq \sqrt{u}\Big[ (1/k) \land \sqrt{C_\pi/I_{v,\pi}}\Big]$. We also have by definition of $ \varepsilon_i^{*}$ for $i \geq I_{v,\pi}$ and by Equation~\eqref{eq:defI}
	\begin{align*}
	\sum_{i\geq I_{v,\pi}} \pi_i^2 \varepsilon^{*2}_i \exp(-2k\pi_i) \leq  \frac{uC_\pi}{2},
	\end{align*}
	and so
	\begin{align*}
	\sum_{i} \pi_i^2 \varepsilon^{*2}_i \exp(-2k\pi_i) = \sum_{i\geq I_{v,\pi}-1} \pi_i^2 \varepsilon^{*2}_i \exp(-2k\pi_i) \leq uC_\pi.
	\end{align*}
	Moreover by saturation of the third constraint
	\begin{align*}
	\sum_{i\geq I_{v,\pi}-1} \pi_i \varepsilon^{*}_i = \sqrt{\frac{u}{2}} \sum_{i\geq I_{v,\pi}-1} \pi_i  \geq \sqrt{\frac{u}{2}} \sum_{J \leq i < I_{v,\pi}-1} \pi_i,
	\end{align*}
	and so
	\begin{align*}
	\sum_{i} \pi_i \varepsilon^{*}_i = \sum_{i\geq I_{v,\pi}-1} \pi_i \varepsilon^{*}_i \geq  \sqrt{\frac{u}{8}} \sum_{J \leq i } \pi_i.
	\end{align*}
	This concludes the proof in this case.
	
	\medskip
	\noindent
	{\bf Case 2: second constraint saturated but not the first one: $\sum_{i \geq I_{v,\pi}-1} \pi_i^2 \exp(-2k\pi_i) > C_\pi$ and $\pi_{I_{v,\pi}-1} \leq \sqrt{C_\pi/I_{v,\pi}}$.}  We have 
	\begin{align}\label{eq:sumImoins1}
	\sum_{i \geq I_{v,\pi}-1} \pi_i^2  \geq \sum_{i \geq I_{v,\pi}-1} \pi_i^2 \exp(-2k\pi_i) \geq C_\pi.
	\end{align}
	Moreover by definition of $ \varepsilon_i^{*}$ for $i \geq I_{v,\pi}$ and by Equation~\eqref{eq:defI} we have
	\begin{align*}
	\sum_{i \geq I_{v,\pi}} \varepsilon_i^{*2} \pi_i^2 \exp(-2k\pi_i) \leq \frac{uC_\pi}{2}.
	\end{align*}
	Set $\varepsilon_{I_{v,\pi}-1}^* = \sqrt{u/2}$ and for all $i < I_{v,\pi}-1$, we set $\varepsilon_i^* = 0$. Note that $\varepsilon_{I_{v,\pi}-1}^* \leq 1$ and $\varepsilon_{I_{v,\pi}-1}^*\pi_{I_{v,\pi}-1} \leq \sqrt{u}\Big[ \sqrt{C_\pi/(2I_{v,\pi})} \land (1/k)\Big]$. So from the last displayed equation and the definition of $\varepsilon_i^*$
	\begin{align*}
	\sum_{i \geq I_{v,\pi}-1} \varepsilon_i^{*2} \pi_i^2 \exp(-2k\pi_i) \leq uC_\pi,
	\end{align*}
	and by Equation~\eqref{eq:sumImoins1}
	\begin{align*}
	\sum_{i} \varepsilon_i^{*2} \pi_i^2 = \sum_{i \geq I_{v,\pi}-1} \varepsilon_i^{*2} \pi_i^2 = \frac{u}{2}\sum_{i \geq I_{v,\pi}-1} \pi_i^2   \geq \frac{uC_\pi}{2}.
	\end{align*}
	Since for all $i \geq I_{v,\pi}-1$ we have $\pi_i \leq \pi_{I_{v,\pi}-1} \leq \sqrt{C_\pi/I_{v,\pi}}$ and $\varepsilon_i^* \leq 1$, we have thus
	\begin{align*}
	\sum_{i} \varepsilon_i^{*} \pi_i = \sum_{i \geq I_{v,\pi}-1} \varepsilon_i^{*} \pi_i  \geq \sqrt{\frac{u}{2}}\frac{C_\pi}{\pi_{I_{v,\pi}-1}} \geq \sqrt{\frac{uC_\pi I_{v,\pi}}{2}} \geq \sqrt{\frac{uC_\pi}{2}} \frac{I_{v,\pi}-J}{\sqrt{I_{v,\pi}}}.
	\end{align*}
	This concludes the proof in this case with Equation~\eqref{eq:sat1}.
	
	\paragraph{Case 3: first constraint saturated, i.e.~$\pi_{I_{v,\pi}-1} > \sqrt{C_\pi/I_{v,\pi}}$.}  
	We set for any $i < J$, $\varepsilon_i^* = 0$ and for any $J\leq i <I_{v,\pi}$,
	$$\varepsilon_i^* = \frac{\sqrt{uC_\pi}}{\sqrt{2I_{v,\pi}} \pi_i}.$$
	Note that for any $i$
	$$\varepsilon_i^* \in [0,1],~~\text{and},~~ \varepsilon_i^*\pi_i  \leq \frac{\sqrt{uC_\pi}}{\sqrt{2I_{v,\pi}}} \leq \sqrt{u} \Big[\sqrt{C_\pi/(2I_{v,\pi})} \land (1/k)\Big],$$
	by Equation~\eqref{eq:boundCsurI}. Moreover we have
	\begin{align*}
	\sum_{J\leq i < I_{v,\pi}} \varepsilon_i^{*2} \pi_i^2 \exp(-2k\pi_i) \leq \frac{uC_\pi}{2},
	\end{align*}
	and so by definition of $I_{v,\pi}$ in Equation~\eqref{eq:defI} and of the $\varepsilon_i^*$ we have
	\begin{align*}
	\sum_{i} \varepsilon_i^{*2} \pi_i^2 \exp(-2k\pi_i) \leq uC_\pi.
	\end{align*}
	Moreover
	\begin{align*}
	\sum_{J \leq i < I_{v,\pi}} \varepsilon_i^{*} \pi_i \geq \sqrt{\frac{uC_\pi}{2}} \frac{I_{v,\pi}-J}{\sqrt{I_{v,\pi}}}.
	\end{align*}
	This concludes the proof in this case with Equation~\eqref{eq:sat1}.
\end{proof}

\subsection{Proof of Proposition~\ref{th:lower2}}

The proof of this proposition is similar to the proof of Proposition~\ref{th:lower}, except that the measures $\Lambda_1$ and $\tilde \Lambda_1$ change, and that we need to adapt Lemma~\ref{lem:TV01marginal}. We therefore take the same notations as in the proof of Proposition~\ref{th:lower}, but redefine $\xi, p, \tilde p,  \Lambda_1, \tilde\Lambda_1$.

\underline{Definition of $\Lambda_1$:} We consider $q, \mathcal A, \mathcal A'$ defined as in the proof of Proposition~\ref{th:lower}. Write 
	$$\bar m= k \int \bar p \mathbf 1\{\bar p \leq 1/k\} d\mathcal U_{\pi,\mathcal A}(\bar p) \leq 1.$$
	For any $i \in {\mathcal A}$, let $\xi_i$ be a random variable that is uniform in $\{0,2\bar m\}$ if $i \in \mathcal A'$, and equal to $q_i$ otherwise. For any $i \not\in \mathcal A$: define $p_{i} = q_i = \pi_i$. We write $\Lambda_1$ for the distribution $\Lambda_{q,p}$ averaged over $q,p$ as
$$\Lambda_1 = \E_{q,p}(\Lambda_{q,p}),$$
where $\E_{q,p}$ is the expectation according to the distribution of $(q,p)$, i.e.~with respect to $q,\xi$.\\
$\tilde \Lambda_1$ and $\tilde p$ are then redefined as in Proposition~\ref{th:lower} as a renormalised version of $ \Lambda_1,p$ using $\mathcal A^C$ such that $\sum_i \tilde p_i = 1$.

We prove the following lemma in order to conclude on a bound on $d_{TV}(\Lambda_0,\Lambda_1)$.
\begin{lemma}
\label{lem:TV01marginalBis}
Let $\pi \in \mathbb R^{+d}$ be such that $\sum_i \pi_i \leq 1$ be a vector ordered in decreasing order, and let $v>0$. There exists a universal constant $h>0$ such that if $\| \pi^2\exp(-2(1+v)k\pi)\|_1 \leq h/k^2$, then the following holds.

Write $\lambda = k\pi$. Let $\mathcal U_\lambda$ be the uniform distribution over the values of the vector $\lambda = k\pi$.
Define the probability distribution
$$V = \frac{1}{2} [\delta_{2\bar m} + \delta_0],$$
i.e.~$2\bar m$ times the outcome of a Bernoulli distribution of parameter $1/2$. We now consider
$$\nu_0' = \nu_0 = \int \mathcal P(\theta)^{\otimes 2} d\mathcal U_\lambda(\theta),$$
and
\begin{align*}
\nu_1' &= \int \int \mathcal P(\theta)\otimes \mathcal P(\theta')dV(\theta')d\mathcal U_\lambda(\theta).
\end{align*}
Then we have
\begin{align*}
d_{TV}(\nu_0^{'\otimes d},\nu_1^{'\otimes d})\leq 69e^{2(1+v)} h.
\end{align*}
\end{lemma}

We use this lemma instead of Lemma~\ref{lem:TV01marginal} in the proof of Proposition~\ref{th:lower}.
By application of Lemma~\ref{lem:TV01marginalBis} on $\pi$ restricted to $\mathcal A$, we have $d_{TV}(\Lambda_0,\Lambda_1)\leq 69e^{2(1+v)} h$. We can then proceed as in the proof of Proposition~\ref{th:lower} to prove the proposition, together with the use of the following lemma instead of Lemma~\ref{lem:barrho}, to conclude the proof.

\begin{lemma}\label{lem:barrho2}
It holds with probability larger than $1-\alpha$ that
\begin{align*}
\sum_{i } |\tilde p_i - \tilde q_i| &\geq \frac{1}{4M}\|\pi\mathbf 1\{ i\geq J\}\|_1 - 2\sqrt{\frac{1}{\alpha k}}- 8a\frac{(1+\log(1/\gamma))}{k}:= \tilde \rho.
\end{align*}
\end{lemma}

\begin{proof}[Proof of Lemma~\ref{lem:barrho2}]

We remind that $\tilde q \sim \mathcal U_\pi^{\otimes d}$, and $k\tilde p \sim V_{(kq)}^{\otimes d}$ and are independent. Note that as in the proof of Lemma~\ref{lem:barrho}
$$\mathbb E_{(q, p)} \sum_{i} |\tilde q_i- \tilde p_i| =  \sum_{i \in \mathcal A'}  \frac{1}{2} \Big[\pi_i + |\pi_i - 2\bar m|\Big] \geq \frac{1}{4M}\|\pi\mathbf 1\{\pi \leq 1/k\}\|_1 - 8a\frac{(1+\log(1/\gamma))}{k},$$
and
$$\mathbb V_{(q, p)} \sum_{i\leq d} |q_i- p_i| = d\cdot \mathbb V_{(q_1,p_1)} |q_1- p_1| \leq 4\|\pi^2\mathbf 1\{\pi \leq 1/k\}\|_1 \leq 4/k.$$
We set for $\alpha>0$
$$\Theta = \Big\{\sum_{i\leq d} |q_i-p_i| \geq \frac{1}{4M}\|\pi\mathbf 1\{\pi \leq 1/k\}\|_1 -  2\sqrt{\frac{\|\pi^2\mathbf 1\{\pi \leq 1/k\}\|_1}{\alpha}} \Big\}.$$
So by Chebyshev's inequality, we know that with probability larger than $1-\alpha$,
\begin{align*}
\sum_{i\leq d} |q_i- p_i| &\geq \frac{1}{4M}\|\pi\mathbf 1\{ i\geq J\}\|_1 - 2\sqrt{\frac{1}{\alpha k}} - 8a\frac{(1+\log(1/\gamma))}{k}:= \tilde \rho.
\end{align*}
\end{proof}

\begin{proof}[Proof of Lemma~\ref{lem:TV01marginalBis}]
By assumption, we have that
\begin{align*}
\|\pi^2 \mathbf 1\{\pi k \leq 1\}\|_1 \leq e^{2(1+v)} h/k^2.
\end{align*}
We also define
\begin{align}\label{eq:boundhihi}
\kappa=\int \theta^2 \mathbf 1\{\theta \leq 1\}d\mathcal U_\lambda(\theta) \leq e^{2(1+v)} h/n.
\end{align}

\paragraph{Definition of two measures for $(p,q)$.}

\noindent
Now, we have by definition for any $m, m' \geq 0$
$$
\nu_0'(m,m') = \int \frac{e^{-2\theta} \theta^{m+m'}}{m!m'!} d\mathcal U_\lambda(\theta),
$$
and for any $\theta$ in the support of $\mathcal U_\lambda$, we have $\theta \leq 1$. So
\begin{align*}
\nu_1'(m,m') =&  \int \frac{e^{-\theta} \theta^{m}}{m!} \mathbf 1\{\theta \leq 1\}d\mathcal U_\lambda(\theta) \cdot \frac{1}{2}\Big[\frac{e^{-2\bar m} (2\bar m)^{m'}}{m'!} + \mathbf 1\{m'=0\}\Big].
\end{align*}


\paragraph{Bound on the total variation.}
We have
\begin{align*}
&d_{TV}(\nu_0^{'\otimes d},\nu_1^{'\otimes d})\leq d \cdot d_{TV}(\nu_0',\nu_1')\\
&\leq d\sum_{k,k'} \Big| \big[\int \mathbf 1\{\theta\leq 1\} \frac{e^{-2\theta} \theta^{m+m'}}{m!m'!} d\mathcal U_\lambda(\theta) \big]\\ 
&\quad\quad\quad\quad -  \big[\int \frac{e^{-\theta} \theta^{m}}{m!} \mathbf 1\{\theta \leq 1\}d\mathcal U_\lambda(\theta) \cdot \frac{1}{2}\big(\frac{e^{-2\bar M} (2\bar m)^{m'}}{m'!} + \mathbf 1\{m'=0\}\big)\big]\Big|\\
&= d\sum_{m,m'} \Big|  \int \mathbf 1\{\theta\leq 1\} \frac{e^{-2\theta} \theta^{m+m'}}{m!m'!} d\mathcal U_\lambda(\theta) \\
&\quad\quad\quad\quad- \int \frac{e^{-\theta} \theta^{m}}{m!} \mathbf 1\{\theta \leq 1\}d\mathcal U_\lambda(\theta) \cdot \frac{1}{2}\big(\frac{e^{-2\bar m} (2\bar m)^{m'}}{m'!} + \mathbf 1\{m'=0\}\big)\Big|.
\end{align*} 
And so
\begin{align*}
&d_{TV}(\nu_0^{'\otimes d},\nu_1^{'\otimes d})\\
&\leq d\Bigg[ \Big|  \int \mathbf 1\{\theta\leq 1\} e^{-2\theta}  d\mathcal U_\lambda(\theta) - \int e^{-\theta}\mathbf 1\{\theta \leq 1\}d\mathcal U_\lambda(\theta) \cdot \frac{1}{2}\big(e^{-2\bar m}  + 1\big)\Big|\\
&\quad +  \Big| \int \mathbf 1\{\theta\leq 1\} e^{-2\theta} \theta d\mathcal U_\lambda(\theta) - \int e^{-\theta} \theta \mathbf 1\{\theta \leq 1\}d\mathcal U_\lambda(\theta) \cdot \frac{1}{2}\big(e^{-2\bar m} + 1\big)\Big|\\
&\quad +  \Big| \int \mathbf 1\{\theta\leq 1\} e^{-2\theta} \theta d\mathcal U_\lambda(\theta) - \int e^{-\theta}  \mathbf 1\{\theta \leq 1\}d\mathcal U_\lambda(\theta) \cdot \frac{1}{2} e^{-2\bar m} (2\bar m)\Big|\\
&\quad + \sum_{m,m' : m+m' \geq 2} \Big|  \int \mathbf 1\{\theta\leq 1\} \frac{e^{-2\theta} \theta^{m+m'}}{m!m'!} d\mathcal U_\lambda(\theta)\\
&\quad\quad\quad\quad\quad\quad\quad- \int \frac{e^{-\theta} \theta^{m}}{m!} \mathbf 1\{\theta \leq 1\}d\mathcal U_\lambda(\theta) \cdot \frac{1}{2}\big(\frac{e^{-2\bar m} (2\bar m)^{m'}}{m'!} + \mathbf 1\{m'=0\}\big)\Big|
\Bigg].
\end{align*}
Since for any $0 \leq x\leq 2$ we have $|e^{-x} - 1 + x| \leq x^2/2$ and $|e^{-x}-1| \leq x$, we have
\begin{align*}
&d_{TV}(\nu_0^{'\otimes d},\nu_1^{'\otimes d}) \\
&\leq d\Bigg[ \Big|  \int \mathbf 1\{\theta\leq 1\} (1-2\theta)  d\mathcal U_\lambda(\theta) - \int (1-\theta)\mathbf 1\{\theta \leq 1\}d\mathcal U_\lambda(\theta) \cdot \frac{1}{2}\big((1-2\bar m)  + 1\big)\Big| \\
&\quad\quad\quad
+ 2\bar m^2\zeta + 3\kappa\\
&+  \Big| \int \mathbf 1\{\theta\leq 1\}  \theta d\mathcal U_\lambda(\theta) - \int  \theta \mathbf 1\{\theta \leq 1\}d\mathcal U_\lambda(\theta) \cdot \frac{1}{2}\big((1-2\bar m) + 1\big)\Big| + 3\kappa + \bar m^2\zeta\\
&+  \Big| \int \mathbf 1\{\theta\leq 1\}  \theta d\mathcal U_\lambda(\theta) - \int (1-\theta) \mathbf 1\{\theta \leq 1\}d\mathcal U_\lambda(\theta) \cdot \bar m\Big| + 2\bar m^2\zeta + 4\kappa\\
&+ \sum_{m,m' : m+m' \geq 2} \frac{1}{m!m'!} \Big|  \int \mathbf 1\{\theta\leq 1\} \theta^{m+m'} d\mathcal U_\lambda(\theta) \\
&\quad\quad\quad\quad\quad\quad\quad\quad\quad\quad\quad +\int \theta^{m} \mathbf 1\{\theta \leq 1\}d\mathcal U_\lambda(\theta) \cdot \frac{1}{2}\big((2\bar m)^{m'} + \mathbf 1\{m'=0\}\big)\Big|
\Bigg].
\end{align*}
Since by Cauchy-Schwarz inequality we have
\begin{align*}
(\bar m\zeta)^2 &= \Big[\int \theta \mathbf 1\{\theta \leq 1\}d\mathcal U_\lambda(\theta)\Big]^2 \\
&\leq \Big[\int \theta^2 \mathbf 1\{\theta \leq 1\}d\mathcal U_\lambda(\theta)\Big]\Big[ \int  \mathbf 1\{\theta \leq 1\}d\mathcal U_\lambda(\theta)\Big] =\zeta \kappa,
\end{align*}
then we have
\begin{align*}
&d_{TV}(\nu_0^{'\otimes d},\nu_1^{'\otimes d})\\
&\leq d\Bigg[ 18\kappa  
+ \sum_{m,m' : m+m' \geq 2} \frac{1}{m!m'!} \Big(  \int \mathbf 1\{\theta\leq 1\} \theta^{m+m'} d\mathcal U_\lambda(\theta)\\
&\quad\quad\quad\quad\quad\quad\quad\quad\quad\quad+ \int \theta^{m} \mathbf 1\{\theta \leq 1\}d\mathcal U_\lambda(\theta) \cdot \frac{1}{2}\big((2\bar m)^{m'} + \mathbf 1\{m'=0\}\big)\Big)\Bigg].
\end{align*}
Then, considering the cases $(m=0, m'\geq 2)$, $(m=1, m'\geq 2)$ and $(m\geq 2, m'=0)$, we have
\begin{align*}
&d_{TV}(\nu_0^{'\otimes d},\nu_1^{'\otimes d})\\
&\leq d\Bigg[ 18\kappa + \sum_{m,m' : m+m' \geq 2} \frac{2^{m'}}{m!m'!} \Big(  \int \mathbf 1\{\theta\leq 1\} \theta^{2} d\mathcal U_\lambda(\theta) + \int \theta^{2} \mathbf 1\{\theta \leq 1\}d\mathcal U_\lambda(\theta)\Big) \\
&\quad\quad\quad\quad  +e^2\zeta \bar m^2 + e^1 \kappa\Bigg]\\
&\leq d\big[ 18\kappa + 2e^3\kappa + e^2\kappa + e^1 \kappa \big] \leq 69d\kappa.
\end{align*}
And so finally
\begin{align*}
d_{TV}(\nu_0^{'\otimes d},\nu_1^{'\otimes d})\leq 69e^{2(1+v)} h,
\end{align*}
by Equation~\eqref{eq:boundhihi}.
\end{proof}

\subsection{Proof of Theorem~\ref{th:lowerall}}

Combining Propositions~\ref{th:lower3}, \ref{th:lower}, and~\ref{th:lower2}, we obtain that no test exists for the testing problem~\eqref{eq:prob} with type I plus type II error smaller than $1 - \alpha - 4\tilde c_v u- 34e^{2(1+v)} \varepsilon$ whenever
\begin{align*}
\rho 
\leq c'' &\Bigg\{\Bigg[\|\pi(\mathbf 1\{i \geq I_{v,\pi}\})\|_1 \vee \Big(\frac{I_{v,\pi}-J}{\sqrt{I_{v,\pi}}}\frac{\Big(\| \pi^2\exp(-2(1+v)k\pi)\|_1 \lor k^{-2}\Big)^{1/4}}{\sqrt{k} } \Big)\Bigg] \\
&\land \|\pi(\mathbf 1\{i \geq J\})\|_1\Bigg\} \lor \frac{\|\pi^{2} \frac{1}{(\pi\lor k^{-1})^{4/3}}\|_1^{3/4}}{\sqrt{k}}\lor \frac{1}{\sqrt{k}},
\end{align*}
where $c''>0$ is some small enough constant that depends only on $u,\alpha,v,\varepsilon$.

And so there exists constants $c_{\gamma,v}>0$ that depend only on $\gamma,v$ such that there is no test $\varphi$ which is uniformly $\gamma$-consistent, for the problem \eqref{eq:prob} with
\begin{align*}
\rho \leq 
c_{\gamma,v}  &\Bigg\{\min_{I\geq J_\pi}\Bigg[\frac{\sqrt{I}}{k}
\lor \Big(\sqrt{\frac{I}{k}}\|\pi^2\exp(-2k\pi)\|_1^{1/4}\Big)\lor  \|\pi_{(.)}(\mathbf 1\{i \geq I)_i\|_1 \Bigg] \\
&\land  \|\pi_{(.)}(\mathbf 1\{i \geq J_{\pi})_i\|_1 \Bigg\} \lor \frac{\Big\|\pi^{2} \frac{1}{(\pi\lor k^{-1})^{4/3}}\Big\|_1^{3/4}}{\sqrt{k}} \lor \sqrt{\frac{1}{k}},
\end{align*}
since for any $I \geq J_\pi$ we have
$$\frac{J_\pi}{\sqrt{I}}\frac{1}{k} \leq k^{-1/2},$$
and 
$$\frac{J_\pi}{\sqrt{Ik}}\|\pi^2\exp(-2k\pi)\|_1^{1/4} \leq  \Big[ \frac{\Big\|\pi^{2} \frac{1}{(\pi\lor k^{-1})^{4/3}} \Big\|_1^{3/4}}{\sqrt{k}} \Big] \lor  \sqrt{\frac{1}{k}}.$$
The final result follows if we take $I^*$ as an $I$ where the minimum is attained as in the theorem, since 
\begin{align*}
\Bigg[\Big(\sqrt{I^* - J_\pi}\frac{\log(k)}{k}\Big)
\lor \Big(\frac{\sqrt{I^* - J_\pi}}{\sqrt{k}}\|\pi^2\exp(-k\pi)\|_1^{1/4}\Big) &\lor  \|\pi_{(.)}(\mathbf 1\{i \geq I^*\})_i\|_1 \Bigg] \\
&\leq  \|\pi_{(.)}(\mathbf 1\{i \geq J_\pi\})_i\|_1,
\end{align*}
and since $$\Big\|\pi^{2} \frac{1}{(\pi\lor k^{-1})^{4/3}} \Big\|_1 \geq \|\pi_{(.)}^{2/3} (\mathbf 1\{i\geq J_\pi\})_i\|_1.$$ 